\documentclass[a4paper,11pt,reqno]{amsart}

\usepackage{xurl}
\usepackage{graphicx}
\usepackage{graphics}
\usepackage{relsize}
\usepackage{caption}
\usepackage{amsmath}
\usepackage{amssymb}
\usepackage{amsthm}
\usepackage{float}
\usepackage{pst-node}
\usepackage{tikz-cd} 
\usepackage[original]{imakeidx}
\makeindex[columns=2, title=Index]
\usepackage{hyperref}
\usepackage{cleveref}
\usepackage{needspace}
\usepackage{etoolbox}

\usepackage{geometry}
\author{Francisco Araújo}
\title{Sarnak's Program for Erd\H{o}s Sieves. Part II: Measure Systems and Applications}

\theoremstyle{plain}
\newtheorem{theorem}{Theorem}[section]

\newtheorem{lemma}[theorem]{Lemma}
\newtheorem{definition}[theorem]{Definition}
\newtheorem{proposition}[theorem]{Proposition}
\newtheorem{corollary}[theorem]{Corollary}
\newtheorem{conjecture}[theorem]{Conjecture}

\theoremstyle{remark}
\newtheorem{example}[theorem]{Example}
\newtheorem{remark}[theorem]{Remark}

\makeatletter

\pretocmd{\theorem}{\needspace{4\baselineskip}}{}{}
\pretocmd{\lemma}{\needspace{3\baselineskip}}{}{}
\pretocmd{\proposition}{\needspace{4\baselineskip}}{}{}
\pretocmd{\corollary}{\needspace{3\baselineskip}}{}{}
\pretocmd{\conjecture}{\needspace{3\baselineskip}}{}{}

\pretocmd{\definition}{\needspace{2\baselineskip}}{}{}

\makeatother

\DeclareMathOperator{\Gen}{Gen}
\DeclareMathOperator{\vol}{vol}

\newcommand{\q}{\mathbb{Q}}
\newcommand{\z}{\mathbb{Z}}

\newcommand{\cb}{\mathcal{B}}
\newcommand{\cbr}{\mathcal{B}_R}
\newcommand{\cc}{\mathcal{C}}
\newcommand{\co}{\mathcal{O}}

\newcommand{\cf}{\mathcal{F}}

\newcommand{\cp}{\mathcal{P}}
\newcommand{\cs}{\mathcal{S}}

\newcommand{\fr}{\mathcal{F}_R}

\newcommand{\frf}{\mathcal{F}_{R^f}}

\newcommand{\frp}{\mathcal{F}_{R'}}

\newcommand{\ga}{\mathfrak{a}}
\newcommand{\gb}{\mathfrak{b}}
\newcommand{\gc}{\mathfrak{c}}
\newcommand{\gp}{\mathfrak{p}}

\newcommand{\gr}{G_{R}}
\newcommand{\grf}{G_{R,F}}
\newcommand{\vrf}{\varphi_{R,F}}
\newcommand{\mmp}{\mathbb{P}}

\newcommand*{\one}{\text{\usefont{U}{bbold}{m}{n}1}}

\begin{document}
	
	\begin{abstract}
		This paper is the second part of a two-part article where we generalize Sarnak's program to sets where we remove congruence classes modulo some infinite set $\mathcal{B}$ of ideals of an étale $\mathbb{Q}-$algebra $K$, which we denote by Erdős sieves. Given a sieve $R$ we define the set $\mathcal{F}_R$ of algebraic integers in $K$ not contained in any of the congruence classes of $R$. We associate to each sieve two measure-theoretical dynamical systems $X_R$ (the orbit closure of $\mathcal{F}_R$) and $\Omega_R$ (the set of $R-$admissible sets) and show how they are related. We show that the system associated to $\Omega_R$ is isomorphic to an ergodic rotation of a compact abelian group, and compute its spectrum. As applications we show results about infinite sumsets in the integers, investigate the case where $\mathcal{F}_R$ is the squarefree values of some polynomial, and show a prime number theorem for $R-$free numbers. 
	\end{abstract}
	
	\maketitle

	\section{Introduction}
	
	In this paper we continue our investigation of Erdős Sieves, following what we have already done in Part I. We will conclude our generalization of Sarnak's program, and then provide number theoretic applications of our results.
	
	As in Part I, by a sieve $R$ we mean a collection of congruence classes $R_\gb$ indexed in some infinite set $\cb_R$ of pairwise coprime ideals of some étale $\q-$algebra $K$. We say a sieve is Erdős if $\sum_{\gb \in \cb_R}|R_\gb|/N(\gb)<\infty$, where $|R_\gb|$ is the number of congruence classes in this set, and $N(\gb)$ is the norm of the ideal, which equals $|\co_K/\gb|$, the total number of possible congruence classes modulo $\gb$.  We want to study the set of $R-$free numbers $\fr$, which correspond to elements of the ring of integers of $K$, denoted  $\co_K$,  not contained in any of the congruence classes in $R_\gb$ for any $\gb \in \cb_R$. To do this, we investigate the orbit closure of $\fr$ in $\{0,1\}^{\co_K}$, which we denote by $X_R$, and the set of $R-$admissible sets (those $A\subset \co_K$ such that $-A+R_\gb \neq \co_K$ for all $\gb \in \cb_R$), which we denote by $\Omega_R$. 
	
	In the case where $R$ is the squarefree sieve, that is $R$ is the collection of the congruence classes $R_p = p^2\z$ for every prime $p$, Sarnak pointed out that $X_R = \Omega_{R}$, which we referred to as Point $(3)$ of Sarnak's program in the introduction of Part I.  This does not have a straightforward generalization for general sieves. Sieves with strong light tails (see the definition in Section 2 or in Part I) are our most `well behaved' sieves, and even for these, it might be the case that $X_R \neq \Omega_{R}$, as shown in \Cref{ex: sieve such that X different Omega}. Yet, in Section 5 we provide a number of results which clarify the relation between $X_R$ and $\Omega_{R}$ when $R$ is a sieve with weak light tails. We show the following result, given in \Cref{thm: nu_R(X_R) = 1}.
	
	\begin{theorem}
		\label{thm: intro nuR XR equal 1}
		Let $R$ be an Erdős sieve. Then, there is a Følner sequence $I_N$ with respect to which $R$ has weak light tails if and only if $$\nu_R(X_R) = 1.$$
		
	\end{theorem}
	
	In order to present our other major result from Section 5, we must introduce the concept of a \textit{minimal} sieve (this is done in Section 3 of this article). By using different ideals, the same sets can be expressed by different unions of congruence classes, for example $\{0,2\}+4\z$ is the same set as $2\z$, or $\{0,2,3,4\}+6\z$ is the same as $2\z \cup 3\z$. This means that there can exist two sieves $R$ and $R'$ that are effectively the same, but that are technically distinct because $\cb_R \neq \cb_{R'}$. In this case, we define the notion of a contraction (see \Cref{def: contraction}) which takes a sieve $R$ and returns a sieve $R'$ where the congruence classes being sieved out are exactly the same, but every ideal in $\cb_{R'}$  divides some unique ideal in $\cb_R$. When a sieve can no longer be contracted, we say it is minimal. We have the following result (see \Cref{thm: equivalence XR =XRp and OmegaR = OmegaRp}).
	
	\begin{theorem}
		\label{thm: intro XR equal XRp implies OmegaR equal OmegaRp}
		Let $R$ and $R'$ be minimal Erdős sieves with weak light tails for some (not necessarily common) Følner sequences. The following are equivalent. 
		
		\begin{enumerate}
			\item $\cb_R = \cb_{R'}$, and for every $\gb \in \cb_R$, there is some $\delta_\gb\in \co_K$ such that $R_\gb = \delta_\gb+ R'_\gb,$ 
			\item $X_R = X_{R'}$,
			\item $\Omega_{R} = \Omega_{R'}.$
		\end{enumerate} 
	\end{theorem}
	
	We also used the notion of minimal sieve to characterize equivalence of sieves. We say $R$ and $R'$ are equivalent, and write $R \sim R'$, if $\fr = \frp$. The following theorem (see \Cref{thm: equivalence minimal sieves main}) describes the relation of uniqueness between a sieve $R$ and $\fr$.
	
	\begin{theorem}
		\label{thm: intro minimal sieves}
		Let $R$ be an Erdős sieve. 
		\begin{itemize}
			\item There exists a minimal Erdős sieve $R'$ such that $R \sim R'$.
			\item If $R$ has weak light tails for some Følner sequence $I_N$, there exists a minimal Erdős sieve $R'$  with weak light tails for $I_N$, such that if $W$ is minimal and $W \sim R$, then $W = R'$, or $W$ does not have weak light tails for any Følner sequence.
			\item If $R$ has strong light tails for $I_N$, then there exists a unique minimal sieve $R'$ (which will have strong light tails for $I_N$) such that $R \sim R'$. 
		\end{itemize}
	\end{theorem}
	
 As was done in \cite{Bartnicka} and \cite{celarosi2} for particular $\cb-$free systems, we also show that for every sieve $R$ there is a group $\grf$ and a rotation $T^F$ of this group such that we have an isomorphism of dynamical systems, which constitutes the central result of section 4 (see \Cref{thm: main theorem}).
	
	\begin{theorem}
		\label{thm: intro main thm}
		Let $R$ be an Erdős sieve. For any $\gb \in \cb_R$, let $$F(R_\gb) = \{x \in \co_K:x+R_\gb = R_\gb\},$$ and define the group $$\grf := \prod_{\gb \in \cb_R}\co_K/F(R_\gb).$$ Letting $T^F$ be the action of $\co_K$ on $\grf$ given by $T^F_a(g)_\gb = g_\gb+a$ and $\mmp^F$ the Haar measure on $\grf$, we have that  $(\Omega_R,S,\nu_R)$ is isomorphic to $(\grf,T^F, \mmp^F)$.
	\end{theorem}

Given $g  \in \grf$, if we write $R(g)$ to be the sieve defined by $$R(g)_\gb = g_\gb + R_\gb,$$ then the isomorphism is exactly the map that sends $g$ to $\cf_{R(g)}$. Using this result, we compute the spectrum of the system $(\Omega_R,S,\nu_R)$.

	In Section 6 we provide applications to number theory of our work into Sarnak's program for sieves. In \cite{Moreira} Moreira, Richter and Robertson showed that  for any $C\subset \mathbb{N}$ such that $\overline{d}(C)>0$, there are infinite $A,B \subset \mathbb{N}$ such that $A+B \subset C$. Host has showed in \cite{Host} that there are sets $C$ with $\overline{d}(C)>0$ such that if there are infinite $A$ and $B$ such that $A+B \subset C$, then we must have $\overline{d}(A) = \overline{d}(B) = 0$. For this reason, it is interesting to study sets $C\subset \mathbb{N}$ such that there are $A,B\subset \mathbb{N}$ which satisfy $A+B \subset C$ and $\overline{d}(B) >0$. We show that $R-$free numbers for sieves with strong light tails provide plenty of examples of such sets.
	
	\begin{theorem}
		\label{thm: intro infinite sumsets}
		Let $A$ be a subset of $\z$ and $R$ an Erdős sieve such that $$\prod_{b\z \in \cb_R} \left(1-\frac{|-A+R_b|}{b}\right) >0.$$ Then, there exists a sequence $g \in\grf$ and some $B\subset \z$ with $d(B)>0$ such that $$A+B \subset \cf_{R(g)}.$$ 
	\end{theorem}
	
	Inspired by the work in \cite{Wang} we also show, using a result from \cite{Bergelson2}, a Prime Number Theorem for $R-$free numbers. Let $v_p$ be the $p-$adic valuation\index{$p-$adic Valuation} of the integers, that is, $v_p(m)$ is the largest non-negative integer $k$ such that $p^k \mid m$, and write $$\Omega(m) = \sum_{p \text{ prime}} v_p(m).$$ We have the following result (see \Cref{thm: ergodic prime number theorem for weak light tails}).
	\begin{theorem}
		\label{thm: intro prime num theorem rfree}
		Let $R$ be an Erdős sieve with weak light tails for $I_N = [1,N]$. Let $(X,T)$ be a uniquely ergodic dynamical system, and $\mu$ its unique invariant measure. Then for every function $f \in C(X)$ and $x \in X$ we have
		$$\lim_{N\rightarrow\infty} \frac{1}{N}\sum_{m \in \fr \cap I_N} f(T^{\Omega(m)}x) = d_I(\fr)\int_X f \,d \mu.$$
	\end{theorem}
	
	This article is divided as follows. In Section 2 we present some basic results about dynamical systems and give an overview of the definitions and results from Part I that we will use. In Section 3 we introduce minimal sieves and show a number of results about these, namely \Cref{thm: intro minimal sieves}. We also define a notion of union of sieves, which allows us to build simpler sieves from more complex ones. In section $4$ we show \Cref{thm: intro main thm} and compute the spectrum of $(\Omega_R,S,\nu_R)$. Additionally, we determine under which conditions two sieves $R$ and $R'$ can be such that $\Omega_R = \Omega_{R'}$. In Section $5$, we investigate the relationship between the spaces $X_R$ and $\Omega_R$, showing both \Cref{thm: intro nuR XR equal 1} and \Cref{thm: intro XR equal XRp implies OmegaR equal OmegaRp}. Finally in Section 6 we show multiple number theoretic results, namely \Cref{thm: intro infinite sumsets}, \Cref{thm: intro prime num theorem rfree}, and a number of facts about the square free values of polynomials.\\
	
\textbf{Acknowledgments}\\

The author would like to thank Jürgen Klüners, Joanna Kułaga-Przymus, Fabian Gundlach, Aurelia Dymek and Michael Baake for many helpful discussions and comments that greatly contributed for this work. This research was supported by the Deutsche Forschungsgemeinschaft (DFG, German Research Foundation) - Project-ID 491392403 - TRR 358 (project A2).

	\section{Preliminaries and results from Part I}

	In order to make this article as self contained as possible, we provide a number of definitions and results from Part I.
	
	By a number field $K$ we mean a finite field extension of the rationals $\q$. An étale $\q-$algebra is a finite product of number fields. If $K = K_1\times \dots \times K_m$ is an étale $\q-$algebra, then its ring of integers $\co_K$ is equal to  $\co_{K_1} \times \dots \co_{K_l}$, and any ideal $I$ of $\co_K$ can be written as a product $I_1 \times \dots \times I_l$, where each $I_i$ is an ideal of $\co_{K_i}$. If each $I_i$ is different from $0$, then we say that $I$ is an \textit{invertible ideal}. Given an invertible ideal $I$, we write $N(I)$ for its norm, which equals $|\co_K/I|$.
	
		If $K$ has degree $n$, then we have the Minkowski embedding $\sigma: \co_K \rightarrow \mathbb{C}^n$ given by $$\sigma(x):= (\phi(x))_{\phi \in \text{Hom}_\q(K, \mathbb{C})}.$$
	We can define a (vector field) norm on $K$ by taking the norm inherited from the supremum norm in the Minkowski embedding, that is
	\begin{equation}
		\label{eq: minkowsky norm}
		\|x\| := \sup_{\phi \in \text{Hom}_\q(K,\mathbb{C})}|\phi(x)|.
	\end{equation} 
	We will write for any $N \in \mathbb{R}_{\geq 0}$\index{Balls in Number Fields}
	\begin{equation}
		\label{eq: definition of BN}
		B_N := \{x \in \co_K: \|x\| \leq N\}.
	\end{equation} 

  We now give an overview of the dynamical systems results we will use. Let $G$ be a group. Since we will only be working with $G$ isomorphic to $\z^n$, we will assume that $G$ is finitely generated, abelian and locally compact. We say a triple $(X,T,\mu)$ is a measure theoretical dynamical system, if $X$ is a compact topological space, $T$ is an action of $G$ on $X$ by homeomorphisms, and $\mu$ is a $T-$invariant probability measure on $X$ (we always assume that $\mu$ is defined on the Borel $\sigma-$algebra of $X$). Looking at $T$ as a map $T:G \times X \rightarrow X$, we will always write  $T_g(x) := T(g,x)$. By $T-$invariant, we mean that for any measurable set $U$, and $g\in G$ we have $\mu(T_g^{-1}(A)) = \mu(A)$. We say a $T$-invariant measure is ergodic if for any measurable set $U$ such that $\mu(U \Delta T_g(U)) = 0$ for every $g \in G$, we have $\mu(U) \in \{0,1\}$.

Given a measure dynamical system $(X,T,\mu)$, we can consider the induced \index{Koopman Representation}Koopman representation $U$ of $G$ on $L^2(X,\mu)$ given by $$U_g(f)(x) = f(T_{-g}x).$$ We can associate to any dynamical system the point spectrum $\sigma_p$ of the corresponding Koopman representation, which is given by\index{Dynamical System ! Spectrum of} \[
\sigma_p(X,T,\mu)
=
\Bigl\{
\chi\in\widehat{G}\;:\;
\exists\, f\in L^2(X,\mu)\setminus\{0\}
\ \text{such that}\ 
U_g f = \chi(g)\,f
\ \text{for all } g\in G
\Bigr\}.
\] where $\widehat{G}$ denotes the group of characters of $G$, that is, the homomorphisms $\chi:G \rightarrow S^1$ (where $S^1$ denotes the unitary circle in $\mathbb{C}$).

A system is said to have \index{Discrete Spectrum}\textit{discrete spectrum} if $L^2(X,\mu)$ has an orthonormal basis formed by eigenfunctions of $U$ (that is, those $f\in L^2(X,\mu)$ such that $U_g(f) = \chi(g)f$ for some character $\chi$). When working on systems with discrete spectrum, the Halmos-von Neumann Theorem\index{Halmos-von Neumann Theorem} is a powerful tool for determining which systems are isomorphic. The theorem states the following (see Theorem 5.5 in \cite{Hermle}).
\begin{theorem}
	\label{thm: Halmos Von Neumann}
	Let $(X_1,T_1,\mu_1)$ and $(X_2,T_2,\mu_2)$ be ergodic systems with discrete spectrum. The systems are isomorphic if and only if $\sigma_p(X_1,T_1,\mu_1) = \sigma_p(X_2,T_2,\mu_2)$.
\end{theorem}

 We say that a sequence of finite, non-empty sets $I_N \subset G$, $N \geq 1$, is a \index{Følner Sequence}\textit{Følner sequence} if for every $x\in G$,  $$\lim_{N \rightarrow \infty} \frac{|(x+I_N) \Delta I_N|}{|I_N|} = 0.$$
	
	 We say that a point $x\in X$ is generic with respect to $I_N$ if for every continuous function $f \in C(X)$ we have $$\lim_{N \rightarrow \infty} \frac{1}{|I_N|}\sum_{a\in I_N} f(T_a(x)) = \int_X f \; d \mu.$$

	The Ergodic Theorem gives conditions for almost every point in a system to be generic. The following very general form was shown by Lindenstrauss in \cite{Lindestrauss}. We say that a Følner sequence $I_N$ is \index{Følner Sequence!Tempered}\textit{tempered} if there is some constant $C$ such that for all $N$,
	\begin{equation}
		\label{eq:tempered Følner sequence}
		\left|\bigcup_{L<N} -I_L + I_N\right| < C|I_N|.
	\end{equation}  Note that every Følner sequence has a tempered subsequence (see Proposition 1.4 in \cite{Lindestrauss}). The Pointwise Ergodic Theorem states the following.
	
	\begin{theorem}
		\label{thm: ergodic theorem}
		Let $(X,T,\mu)$ be an ergodic dynamical system, and $I_N$ a tempered Følner sequence. Then, for every $f\in L^1(X)$, we have $$\lim_{N \rightarrow \infty} \frac{1}{|I_N|} \sum_{a\in I_N} f(T_a(x)) = \int_X f d \mu$$ for $\mu-$almost ever $x \in X$.
		
	\end{theorem}  
	
	\begin{remark}
		\label{rmk: ergodic measures are mutualy singular}
		In particular, we have that given an ergodic measure $\mu$ and a tempered Følner sequence, the set $\Gen(\mu,I_N)$ of generic points with respect to $I_N$ satisfies $\mu(\Gen(\mu,I_N)) = 1$ (see Corollary 8 in \cite{Host}). 
	\end{remark}
	
	Given a point $x \in X$, let $\mathbb{O}_T(x)$ be the orbit closure\index{Orbit Closure} of $x$, that is $$\mathbb{O}_T(x) := \overline{\{T_a(x):a \in G\}}.$$
	We will use the following property of generic points in Section 5.

	\begin{lemma}
		\label{lm: x generic implies full meausure orbit closure}
		Let $(X,T,\mu)$ be an ergodic dynamical system. If there exists a  Følner sequence $I_N$ such that $x\in \Gen(\mu,I_N)$ then $$\mu\left(\mathbb{O}_T(x)\right) = 1.$$
	\end{lemma}
	
	\begin{proof}	Define the sequence of measures $$\mu_N := \frac{1}{|I_N|}\sum_{a\in I_N} \delta_{T_a(x)}. $$ Since $x$ is generic with respect to $I_N$, we know that the sequence $\mu_N$ converges weakly to $\mu$. We have that $\mu_N(\{T_a(x):a \in G\}) = 1$ for every $N$.
		
		We now use Portmanteau's theorem \index{Portmanteau's theorem} (see Theorem 2.1 in \cite{Billingsley}), which states that if $\mu_N $ converges weakly to $\mu$, and $C$ is a closed set, then $$\limsup_N \mu_N(C) \leq \mu(C). $$  By applying it with $C =\mathbb{O}_T(x)$, we get $$1 =\limsup_N \mu_N(\mathbb{O}_T(x)) \leq \mu(\mathbb{O}_T(x)), $$ which concludes the proof.
	\end{proof}

\subsection{Results from Part I}

We will provide a number of definitions and results from Part I that we will use throughout. Most of these definitions are given in Section 3 of \cite{part1}. Throughout, we assume $K$ is an étale $\q-$algebra of degree $n$. We start with the definition of sieve.

\begin{definition}
	\label{def: Sieve}
	A \index{Sieve}sieve over an étale $\q-$algebra $K$ is a pair $(\cb_R,(R_\gb)_{\gb \in \cb_R})$ where $\cb_R$ is an infinite set of pairwise coprime invertible ideals of $\co_K$, and each $R_\gb$ is a set of the form $R_\gb =  S_\gb + \gb $ with $S_\gb$ finite sets such that $R_\gb \neq \co_K$ for every $\gb \in \cb_R$.
\end{definition}

We will say that $R$ is supported on the set $\cb_R$. Additionally, we will many times just assume that $\cb_R$ is ordered, and write $R_i$ for $R_{\gb_i}$. We always use the notation $|R_\gb|$ for the number of congruence classes modulo $\gb$ in $R_\gb$, that is, the cardinality of this set inside of $\co_K/\gb$. We say a sieve $R$ is a $\cb-$free system if $R_\gb = \gb$ for all $\gb \in \cb_R$.

 We say $R$ is an Erdős sieve if $$\sum_{\gb \in \cb_R} \frac{|R_\gb|}{N(\gb)} < \infty.$$ By the $R-$free numbers we mean the set $$\fr := \co_K\setminus\bigcup_{\gb \in \cb_R} R_\gb.$$ We will study this set by considering two distinct shift spaces. First, we identify $\{0,1\}^{\co_K}$ with the powerset of $\co_K$, and define the shift action $S$ of $\co_K$ on $\{0,1\}^{\co_K}$, by $S_a(A) = A-a$ for any $a\in \co_K$ and $A \subset \co_K$. We say a set $A$ is $R-$admissible if for all $\gb \in \cb_R$ we have $-A+R_\gb \neq \co_K$, and write $\Omega_{R}$ for the set of all $R-$admissible sets. We write $$X_R := \overline{\{S_a(\fr):a \in \co_K\}}$$ for the orbit closure of $\fr$ in $\{0,1\}^{\co_K}$.

Both sets are compact and they become dynamical systems with the Mirsky measure $\nu_R$. This is defined as follows. For any $R$, let $G_R$ be the group
\begin{equation}
	\label{eq:GR}
	G_R = \prod_{\gb \in \cb_R} \co_K/\gb,
\end{equation}
  and let $\mmp$ be the Haar measure in $G_R$. Consider the map $\varphi_R:G_R \rightarrow \Omega_R$, defined by the relation \begin{equation}
	a \in \varphi_R(g) \Leftrightarrow \mathlarger \forall_{\gb \in \cbr} \hspace{5pt} a+g_\gb \not \in R_\gb.
\end{equation} 	We define the Mirsky measure $\nu_R$ as $(\varphi_R)_*\mmp$ (that is, for every measurable $U$, $\nu_R(U) = \mmp((\varphi_R)^{-1}(U))$). This measure has the following simpler description. Given disjoint sets $A,B\subset \co_K$, let $C^R_{A,B}$ be the set of all $R-$admissible sets that contain $A$ and are disjoint from $B$, that is,  $$C^R_{A,B} := \{Y\in \Omega_R: A \subset Y,  Y \cap B = \emptyset\}.$$ Taking finite $A$ and $B$, sets of the form $C^R_{A,B}$ generate the topology of $\Omega_{R}$, and \begin{equation}
\label{eq:Formula for nu_R with B empty}
\nu_R(C^R_{A,\emptyset}) =\prod_{\gb \in \cbr} \left(1-\frac{|-A+R_\gb|}{N(\gb)}\right).
\end{equation} Using the inclusion-exclusion principle, we have more generally that if $B$ is finite, then $$\nu_R(C^R_{A,B})  = \sum_{A \subset D \subset  A\cup B} (-1)^{|D\setminus A|}\nu_R(C^R_{D,\emptyset}).$$

Throughout, we will use the following definition of density with respect to a Følner sequence.

\begin{definition}
	\label{def: density}
	Given a set $A \subset \co_K$ and a Følner sequence $I_N$, we define the upper and lower \index{Density}densities with respect to $I_N$ to be respectively,  $$\overline{d}_I(A) := \limsup_{N \rightarrow \infty} \frac{|A \cap I_N|}{|I_N|} \text{ and } \underline{d}_I(A) := \liminf_{N \rightarrow \infty} \frac{|A \cap I_N|}{|I_N|}.$$ If these agree, we write the limit as $d_I(A)$, which we call the density of $A$ with respect to $I_N$.
\end{definition}

In the case where $I_N = B_N$ (as defined in \Cref{eq: definition of BN}), we omit the $I$, and write $\overline{d}(A),\underline{d}(A),d(A)$ for the corresponding densities. We can use the following lemma to count points that don't belong to any of a finite number of congruence classes.

\begin{lemma}
	\label{lm: equidistribution for Følner sequences}
	Let $L \geq 1$ be an integer,  $I_N$ a Følner sequence, $\gb_1,\dots,\gb_L$ a collection of pairwise coprime ideals. For each $i$, take $A_i \subset \co_K$ and let $R_i = A_i + \gb_i$. If  $$C_L := \{x \in \co_K: \mathlarger \forall_i \hspace{2pt} x\not \in R_i \},$$ then $$\frac{1}{|I_N|}\sum_{a\in I_N} \one_{C_L}(a) \rightarrow d_I(C_L) = \prod_{i=1}^L\left(1-\frac{|R_i|}{N(\gb_i)}\right),$$ as $N \rightarrow \infty$.
\end{lemma}

 Let $R$ be an Erdős sieve. We can assume that $\cb_R$ has been ordered. We then say that $R$ has \textit{weak light tails} with respect to a Følner sequence $I_N$ if 	$$\lim_{L \rightarrow \infty} \overline{d}_I\left(\bigcup_{i >L } R_i \setminus \bigcup_{j \leq L } R_j\right) = 0.$$ We say it has strong light tails if $$\lim_{L \rightarrow \infty} \overline{d}_I\left(\bigcup_{i >L } R_i\right) = 0.$$ By Lemmas 5.15 and 5.16 in \cite{part1}, both of these properties are invariant under ordering of $\cb_R$. In Example 5.14 of \cite{part1} we show that the sieve $R$ defined by $R_{i} = 1+4(i-1)+p_{i}^2\z$ has weak light tails for the sequence $I_N = [0,N]$, but does not have strong light tails with respect to this Følner sequence. 
 
 As shown in Theorem 5.7 of \cite{part1}, every Erdős $\cb-$free system has strong light tails.
 
 \begin{theorem}
 	\label{thm: K etale algebra light tails for B_n finite}
 	Let $R$ be a sieve over an étale $\q-$algebra $K$ of degree $n$, such that $\cb_R =\{\gb_1,\gb_2,\dots\}$ and there is a finite set $T$ for which $R_i \subset T + \gb_i$ for every $i$. If $\sum_i \frac{1}{N(\gb_i)} < \infty$, then $R$ is Erdős and has strong light tails for $B_N$.
 \end{theorem}

 We will need the following results about sieves with weak light tails. First, we have Theorem 3.21 of \cite{part1}, which generalizes point (1) of Sarnak's program for Erdős sieves.
 
  	\begin{theorem}
 	\label{thm:Fr is generic}
 	Let $R$ be an Erdős sieve. For a given Følner sequence $I_N$, the following are equivalent.
 	\begin{enumerate}
 		\item $\fr$ is a generic point of $(\Omega_{R},S,\nu_R)$ with respect to $I_N$.
 		\item $R$ has weak light tails with respect to $I_N$.
 		\item The set $\fr$ has a density with respect to $I_N$ given by $$d_I(\fr) = \nu_R(C^R_{\{0\},\emptyset}) =  \prod_{\gb \in \cbr}  \left(1-\frac{|R_\gb|}{N(\gb)}\right). $$
 	\end{enumerate}
 	
 \end{theorem}

As a consequence of showing that a sieve $R$ has weak light tails, we get that $\fr$ contains infinitely many copies of any finite $R-$admissible set $A$.

\begin{theorem}
	\label{thm: density of x+A in fr}
	Let $R$ be an Erdős sieve with weak light tails for some Følner sequence $I_N$, and $A$ a finite $R-$admissible. Then,
	$$d_I(\{x: x+A \subset \fr\}) = \prod_{i} \left(1-\frac{|-A+R_i|}{N(\gb_i)}\right)>0.$$ In particular, there are infinitely many $x$ such that $x+A \subset \fr$.
\end{theorem} 

In what follows we will many time want to show that $\fr$ has elements in some congruence class, so we will use Lemma 5.19 from \cite{part1}.

\begin{lemma}
	\label{lm: Fr intercepted by ideal}
	Let $\mathcal{B} = \{\gb_1, \gb_2, \cdots\}$ be an infinite set of pairwise coprime ideals. Let $R$ be an Erdős sieve over $\mathcal{B}$ with weak light tails for $I_N$. If $\gb$ is an ideal coprime to every ideal in $\mathcal{B}$, then for every $x,b \in \co_K$, $$d_I((x+\fr) \cap (b+\gb)) = \frac{d_I(\fr)}{N(\gb)}.$$
	Additionally, we have for any $i$ that  $$d_I(\fr \cap(x+\gb_i))  = \frac{d_I(\fr)}{|R_i^c|}$$ if $x + \gb_i \not \in R_i $ 
\end{lemma}

Finally, we will need Theorem 5.29 of \cite{part1}.

	\begin{theorem}
	\label{thm: R has strong light tails if finite translations}
	Let $R$ be an Erdős sieve with weak light tails with respect to some Følner sequence $I_N$. Suppose that for any finite set $A \subset \mathbb{N}$ and choice of $x_i \in \co_K$ for $i \in A$, the sieve $R'$ defined by $R'_i = x_i + R_i$ with $i \in A $, and $R'_i = R_i$ otherwise, also has weak light tails. Then $R$ has strong light tails for $I_N$.
\end{theorem}

		\section{Minimal Sieves and Union of Sieves}

		In this section we study sieves as objects of interest in themselves. Our overarching objective is to study sets of the form $\fr$ for some sieve $R$, so one of the first questions we ask is when two sieves produce the same set $\fr$.  To study this, we define two sieves to be equivalent if they sieve out the same elements.
		
		\begin{definition}
			\label{def: equivalent sieves}
			We say two sieves $R$ and $R'$ over the same étale $\q-$algebra are \index{Equivalence of Sieves}\textit{equivalent} and write $R \sim R'$ if $\fr = \frp$.
		\end{definition}
		
		We want to know when two sieves are equivalent. We start with the following consideration. Notice that $2\z$ and $\{0,2\}+4\z$ are the same sets in $\z$. Therefore, two sieves $R$ and $R'$, defined by $R_1 = 2\z$, $R'_1 = \{0,2\}+4\z$ and $R_i = R'_i$ for $i > 1$ will be equivalent. Yet, they will not be the same, as $\cb_R \neq \cb_{R'}$. This leads us to define the notion of a minimal sieve, whose elements of $\cb_R$ are as small as they possibly can norm wise (so for example, $R'$ cannot be a minimal sieve, as we could replace $R'_1$ by $2\z$ and $2<4$).
		
		\subsection{Minimal Sieves}
		
		Given a sieve $R$, there are two main ways of producing an equivalent sieve $R'$. The composition of these will be what we will call a dilation. The first operation to consider, is one where we take a sieve $R$, and for each $\gb_i \in \cb_R$ choose $\gc_i$ such that $\gb_i \mid \gc_i$ and the $\gc_i$ are pairwise coprime. Then, $\gb_i + \gc_i = \gb_i$, and so a sieve $R'$ supported on $\cb_{R'} = \{\gc_i: i \in \mathbb{N}\}$ with $R'_i = R_i + \gc_i = R_i$ is equivalent to $R$. We give two examples of this operation.
		
		\begin{example}
			Let $R$ be the squarefree sieve supported on $\cb_R= \{p^2\z: p \text{ prime}\}$ and such such that $R_p = p^2 \z$. It is equivalent to the sieve $R'$ supported on $\cb_{R'} = \{p^3\z:p \text{ prime}\}$ and defined by $$R'_p = \{jp^2: 0\leq j < p\} + p^3\z = p^2\z + p^3\z  .$$ 
			
			Alternatively, let $q_i$ denote the $i-$th prime congruent to $1 \mod 4$, and $W$ the sieve $W_i = q_i^2\z$. Taking $r_i$ to be the primes congruent to  $3 \mod 4$, it is clear that $W$ is equivalent to the sieve $W'$ supported on $\cb_{W'} = \{q_i^2r_i^2: i \in \mathbb{N}\}$ and defined by $$W'_i  = \{jq_i^2: 0\leq j < r_i^2\} + q_i^2r_i^2\z = q_i^2\z + q_i^2r_i^2\z.$$
		\end{example}
		
		The second operation we can apply on $R$ is as follows. Take a finite set $\{\gb_1,\dots,\gb_m\}$ of elements of $\cb_R$, and write $\gc = \prod_{i=1}^m \gb_i.$ Consider the sieve $R'$ supported on $(\cb_R\cup \{\gc\}) \setminus \{\gb_1,\dots,\gb_m\}$ where we have removed from $R$ all the $R_{\gb_i}$, and instead have $R'_\gc = \bigcup R_{\gb_i} + \gc$. Since there was some $x_i \not \in R_{\gb_i}$ for every $i$, the Chinese Remainder Theorem guarantees that $R'_\gc \neq \co_K$. Take the following example.
		
		\begin{example}
			Let $R$ be an Erdős $\cb-$free system. Let $\cp$ be a partition of $\mathbb{N}$ into finite sets (so that $\cp$ is a collection of disjoint finite sets that together cover $\mathbb{N}$). For any element $A$ of $\cp$, let $b_A := \prod_{i\in A}b_i$, and write  $\cb' := \{ b_A \z: A \in \cp \}$. The sieve $R'$ supported on $\cb'$, and defined by $R'_A = \bigcup_{i\in A} b_i\z + b_A\z$ is equivalent to $R$, as $ b_i\z + b_A\z = b_i\z$ for any $i \in A$, and so $$ \frp^c = \bigcup_{A \in \cp} R'_A = \bigcup_{A \in \cp} \bigcup_{i\in A}  R_i = \bigcup_{i \in \mathbb{N}} R_i = \fr^c. $$
		\end{example}

		We can combine these two operation on sieves, to obtain the following general operation that we call  \index{Dilation of Sieves}\textit{dilation}. 
		
		\begin{definition}
			\label{def: dilation}
			Let $R$ be a sieve supported on $\cb_R$, and consider some partition $\cp$ of $\mathbb{N}$ into finite sets. For every $A \in \cp$, let $\ga_A$ be some ideal satisfying $(\ga_A,\gb_i) =1$ if $i \not \in A$, and such that for any $B \in \cp$, $(\ga_A,\ga_B) =1$ if $A\neq B$. Finally, let $\gc(A) = \text{lcm}(\{\ga_A\}\cup\{\gb_i:i\in A\})$ and $\cb' := \{\gc(A): A \in \cp \}.$ Given this initial data we define the corresponding dilation of $R$ to be the sieve $R'$ supported on $\cb'$ and defined by $$R'_A = \bigcup_{i \in A} R_i + \gc(A).$$
		\end{definition}
		
		We remark that since $\gb_i \mid \gc(A)$ for every $i \in A$, $R'_A$ is just the union of all the $R_i$ with $i \in A$. The choice to write the "$+ \gc(A)$" is a stylistic choice, so that we always write $R_\gb$ in the form $R_\gb = S + \gb$.
		
		In order to show that a dilation does indeed define a sieve, we have to show that the elements of $\cb'$ are pairwise coprime. To see this, note that if $(\gc(A),\gc(B)) \neq 1$, then one of the following four things must happen: either $(\ga_A,\ga_B) \neq 1$, $(\ga_A, \gb_j) \neq 1$ for some $j \in B$,$(\ga_B, \gb_i) \neq 1$ for some $i \in A$, or $(\gb_i, \gb_j) \neq 1$ for some $i \in A$ and $j \in B$. The first three options cannot happen by the definition of dilation, and the last one cannot happen since the elements of $\cb_R$ are assumed to be pairwise coprime. Furthermore, $R'$ is equivalent to $R$, since $\gc(A) + \gb_i = \gb_i$ for any $i \in A$, and so $$\frp^c = \bigcup_{A \in \cp }  R'_A =  \bigcup_{A \in \cp } \bigcup_{i \in A }  R_i = \fr^c.$$
		
		On the other hand, given some sieve $R$ and $\gb \in \cb_R$, it may be possible to write $$R_\gb = \bigcup_{i \in A} S_i + \gb_i,$$ where $A \subset \mathbb{N}$ and  $S_i\subset \co_K$ are non-empty finite sets, and the $\gb_i$ form a set of pairwise coprime ideals distinct from $\gb$ such that $\gb_i \mid \gb$.  In this case, we can define a  sieve $R'$ supported on $\cb\setminus\{\gb\}\cup\{\gb_1,\dots,\gb_{|A|}\}$ such that $R'_{\gb_{i}} = S_i + \gb_i $, and $R \sim R'$. We call this operation a \textit{contraction} of $R$, since it is always the inverse of a dilation. Indeed, if $R$ is our initial sieve, and we choose a finite set $A$, and an ideal $\gb$ such that $\text{lcm}(\{\gb_i:i\in A\})|\gb$, then by dilating we will obtain a sieve $R'$ supported on $\cb\cup\{\gb\}\setminus\{\gb_1,\dots,\gb_{|A|}\}$ such that $R'_\gb = \bigcup_{i \in A} R_i + \gb_i$. By contracting, it is clear we get the original sieve. 
		
		More generally, because a dilation is defined as an operation on all ideals simultaneously, we also define a (general) contraction as the repetition of this operation for each ideal for which it is possible.\index{Contraction of Sieves}
		
		\begin{definition}
			\label{def: contraction}
			Let $R$ be a sieve. We say a sieve $R'$ is a contraction of $R$ if there is a partition $\cp$ of $\cb_{R'}$ into finite sets and a bijection from $\cb_R$ to $\cp$ sending $\gb$ to $A_\gb \in \cp$, such that either $A_\gb = \{\gb\}$ and $R_\gb = R'_{\gb}$ or $$R_\gb  = \bigcup_{\gb' \in A_\gb} R'_{\gb'}, $$ where for every $\gb' \in A_\gb$ we have that $\gb' \mid \gb$, but $\gb' \neq \gb$.
		\end{definition}
		
		We show that dilations and contractions preserve the Erdős and light tail conditions for any Følner sequence $I_N$.

		\begin{lemma}
			\label{lm: Erds and light tails preserved by change of base}
			Let $R$ be a sieve and $R'$ a dilation of $R$. Then $R'$ will be an Erdős sieve with weak (respectively, strong light tails) if and only if $R$ is Erdős and has weak (respectively, strong light tails) for a Følner sequence $I_N$. 
		
		\end{lemma}
		
		\begin{proof}
			We will use the same notation as in \Cref{def: dilation}. Since $R'$ is a dilation of $R$, there must exist some partition $\mathcal{P}$ of $\mathbb{N}$ into non-empty finite sets, and some collection $\{\ga_A:A \in \cp\}$, such that writing $\gc(A) = \textnormal{lcm}(\{\ga_A\}\cup\{\gb_i:i\in A\})$ and $\mathcal{C} = \{\gc(A):A \in \cp\}$, the sieve $R'$ is supported on $\mathcal{C}$ and for any $\gc \in \mathcal{C}$ we have $$R'_\gc = \bigcup_{\gb \mid \gc} R_\gb + \gc.$$
			
			Take any $\gc \in \mathcal{C}$, and suppose that $\gb_{1}, \dots, \gb_{k}$ are the element of $\cb_R$ that divide $\gc$.
			
			Then,  $$\vol(R'_\gc) = \frac{|\bigcup_{i=1}^k R_{\gb_{i}} + \gc |}{N(\gc)} \leq \sum_{i=1}^k \frac{|R_{\gb_{i}}+\gc|}{N(\gc)} =  \sum_{i=1}^k\vol(R_{\gb_i}). $$ Since every $\gb$ divides at least one $\gc$,  it follows that $\sum_{\gc \in \mathcal{C} } \vol(R'_\gc) \leq  \sum_{\gb \in \cb} \vol(R_\gb) $, so $R'$ is Erdős if $R$ is. On the other hand, if $x \not \in R'_\gc$, then $x \not \in R_\gb$ for every $\gb \mid \gc$. It follows that $$\prod_{\gb \mid \gc}(1-\vol(R_\gb)) \leq (1-\vol(R'_\gc)),$$ and so $\prod_{\gb \in \cb}(1-\vol(R_\gb)) \leq \prod_{\gc\in \mathcal{C} }(1-\vol(R'_\gc))$, which implies that $R$ is Erdős if $R'$ is.
			
			Next we show that $R'$ has weak/strong light tails if and only if $R$ does. Since the light tail properties don't depend on a choice of the order of the support of the sieve, we are free to reorder $\cb_R$ and $\mathcal{C}$, so we order them in such a way, that there is a function $f$ defined by the rule that $\gb_i \mid  \gc_j$ is equivalent to $f(j) \leq i <f(j+1)$. Then we have $$R'_L = \bigcup_{f(L) \leq  j < f(L+1)}R_j,$$ and so $$\bigcup_{i>L} R'_i = \bigcup_{j\geq f(L+1)} R_j.$$
			
			Similarly, we have that $$\bigcup_{i \leq L} R'_i = \bigcup_{j<f(L+1)} R_j,$$ which gives the equality 
			$$ \bigcup_{i>L}R'_i\setminus \bigcup_{j \leq L} R'_j  =\bigcup_{i \geq f(L+1)} R_i \setminus \bigcup_{j \leq f(L+1)} R_j   ,$$

			We now show the equivalence in the case of strong light tails. If $R$ is Erdős and has strong light tails, then the right hand side of 
			\begin{equation}
				\label{eq: a_l subseq}
				\overline{d}_I\left(\bigcup_{i>L}R'_i\right)  = \overline{d}_I\left(\bigcup_{j \geq f(L+1)} R_j \right)   ,
			\end{equation}
			is a subsequence of a sequence that converges to $0$, and so the left hand side must converge to $0$, that is, $R'$ must have strong light tails. On the other hand, notice that the sequence $$a_L = \overline{d}_I\left(\bigcup_{i > L} R_i \right)$$ is monotonically non increasing. If $R'$ has strong light tails, then \Cref{eq: a_l subseq} shows that $a_L$ has a subsequence that converges to $0$, which implies that the entire sequence must converge to $0$, so $R$ must also have strong light tails.
			
			The equivalence in the case of weak light tails follows analogously, only having to point out that the sequence $$b_L = \overline{d}_I\left(\bigcup_{i > L} R_i\setminus \bigcup_{j \leq L} R_j \right)$$ is also monotonically non increasing, since $$\bigcup_{i > L+1} R_i\setminus \bigcup_{j \leq L+1} R_j \subset \bigcup_{i > L} R_i\setminus \bigcup_{j \leq L} R_j  $$

			for every $L$.
		\end{proof}

		This motivates us to define the notion of a \index{Sieve!Minimal}\textit{minimal} sieve.
		\begin{definition}
			\label{def: minimal sieve}
			We say a sieve $R$ is minimal if for every $\gb \in \cb_R$, there is no collection of pairwise coprime ideals $\gb_1 , \dots, \gb_r $ distinct from $\gb$ with $\gb_i \mid \gb$ for every $1\leq i\leq r$, and finite sets $S_i \subset \co_K$ such that  $$R_\gb = \bigcup_{i=1}^r (S_i + \gb_i).$$ Equivalently, we say that $R$ is minimal if it has no contractions or, if there is no sieve $R'$ such that $R$ is a dilation of $R'$.
			
		\end{definition}
		
		\begin{remark}
			Throughout this section, we will always assume that given a sieve $R$ and $\gb \in \cb_R$, we have $R_\gb \neq \emptyset$. The reason is as follows: consider a sieve $R'$ such that $\cb_{R'} = \cb_R \cup A$, with $R'_\gb = R_\gb$ if $\gb \in \cb_R$ and $R'_\gb =\emptyset$ if $\gb \in A$, to be a dilation of $R$. A sieve obtained from $R$ by removing those  $\gb $ in $\cb_R$ such that $R_\gb = \emptyset$ is then a contraction of $R$. In particular, a minimal sieve $R$ will then necessarily, as a sieve which has no contractions, be a sieve such that $R_\gb\neq \emptyset$ for every $\gb \in \cb_R$.
		\end{remark}

		\begin{example}
			\label{ex: B-free systems are minimal}
			The simplest example of minimal sieves are those coming from $\cb-$free systems. Say that $R_\gb = \gb$ for every $\gb \in \cbr$. Taking any $\gb' \mid \gb$ distinct from $\gb$, we cannot possibly have $x+\gb' \subset R_\gb = \gb$, given that $x+\gb'$ will contain $[\gb:\gb'] >1$ congruence classes mod $\gb$. 
			
			For another example, consider a sieve $R$ given by $R_p = p^2\z$ for every prime different from $5$, and such that $R_5 = \{0,5,6,10, 15,20\} + 25 \z$. Then, $R_5 = 5\z \cup (\{6\}+25\z)$, which cannot be written as the union of congruence classes mod $5$ (the only number that divides $25$ distinct from $1$ and $25$). Since $p^2\z$ cannot be written as the union of congruence classes mod $p$, it follows that $R$ is minimal. 
			
			Finally, let $R$ be the sieve defined by $R_1 = \{1,2,7,13,17,19,25\} + 30 \z$, and $R_i = p_{i+3}^2\z$ for $i \geq 2$. Then, $R_1 = (1+6\z) \cup (2+15\z). $ Since $6$ and $15$ are not coprime, and no other combination of the form $x+b\z$ is contained in $R_1$ with $b\mid 30$, it follow that $R$ is minimal, despite of the fact that for every $x \in R_1$, there is some $\ga \mid 30\z$ such that $x+\ga \subset R_1$.
			
			For an example of non-minimal sieves, let $W$ be a sieve such that $W_5 = \{0,5,10, 15,20\} + 25 \z$. We can write $W_5 = 5 \z$, so $W$ is not minimal. Alternatively, consider a sieve $W'$ such that $W'_6 =\{0,2,3,4\} + 6 \z$.  Since $W'_6 = 2\z \cup 3\z$, we also see that $W'$ is not minimal.
		\end{example}

		Intuitively, for any sieve $R$, we can keep contracting it, until we end up with a minimal sieve. In the next lemma we show that this is indeed the case.

		\begin{lemma}
			\label{lm: every sieve is equivalent to minimal sieve}
			Let $R$ be an Erdős sieve. There is a minimal sieve $R'$ equivalent to $R$. Additionally, if $R$ has weak/strong light tails for some Følner sequence $I_N$, then so does $R'$.
		\end{lemma}

		\begin{proof}
			Fix any $\gb \in \cb_R$. The sieve $R'$ can be constructed recursively using the following algorithm. If $R_\gb$ cannot be written as the union of some $W_{\gb'}$, with $\gb'$ in some finite set $A_\gb$ such that $\gb' \mid \gb$ and $\gb' \neq \gb$, then we add $\gb$ to $\cb_{R'}$ and set $R'_\gb := R_\gb$. Otherwise, we replace $\gb$ in $\cb_R$ by all of the $\gb' \in A_\gb$, and set $R_{\gb'} := W_{\gb'}$. Repeatedly applying this process now for every $R_{\gb'}$, we will eventually terminate in such a way that no $R'_{\gb'}$ can be further contracted for any $\gb' \in \cb_{R'}$ that divides our original $\gb$, and that the union of $R'_{\gb'}$ for all such $\gb'$ equals $R_\gb$. The number of steps until termination is finite, as it cannot be greater than the number of divisors of $\gb$. Repeating the process now for each $\gb \in \cb_R$ will produce the entire minimal sieve $R'$. 
			
			Letting $A_\gb$ be the set of all $\gb' \in \cb_{R'}$ that divide $\gb$, we will have that the collection $\{A_\gb\}_{\gb \in \cb_R}$ is a partition of $\cb_{R'}$, and $R_\gb = \bigcup_{\gb' \in A_\gb} R'_{\gb'}. $ It follows that $R$ is a dilation of $R'$, and so by \Cref{lm: Erds and light tails preserved by change of base}, the sieve $R$ is Erdős with weak/strong light tails for $I_N$ if and only if $R'$ also is.
		\end{proof}
		
		\begin{example}
			Consider the sieve $R$ with $\cb_R = \{p_{2i}p_{2i+1}^2: i \in \mathbb{N}\}$ defined by $$R_i = \{jp_{2i+1}^2: 0\leq j < p_{2_i}\} + p_{2i}p_{2i+1}^2\z.$$ This sieve is equivalent to the sieve $R'$ defined by $R'_i = p_{2i+1}^2\z$, which is minimal (as shown in \Cref{ex: B-free systems are minimal}). Also, since $R'$ has strong light tails (by  \Cref{thm: K etale algebra light tails for B_n finite}), so does $R$.
		\end{example}
		
		We now want to study the extent to which any particular sieve is uniquely equivalent to a minimal sieve. The weak light tails property will play an important role here. We start with two lemmas, which will help us characterize when $R \sim R'$.

		\begin{lemma}
			\label{lm: if contains coprime ideal then no weak light tails}
			Let $R$ be an Erdős sieve. If there is some ideal $\gc$ coprime to every element of $\cb_R$ such that $y+\gc \subset \fr^c$ for some $y \in \co_K$, then $R$ does not have weak light tails for any Følner sequence.
		\end{lemma}
		
		\begin{proof}
			Fix some integer $L \geq 1$. Due to our hypothesis, we have that for every $N$,
			$$ \left|\left(\bigcup_{i=L+1}^\infty R_i  \setminus \bigcup_{i=1}^L R_i \right) \cap I_N \right| \geq \left|\left((y+\gc) \setminus \bigcup_{i=1}^L R_i\right)  \cap I_N\right|. $$  By applying  \Cref{lm: equidistribution for Følner sequences} (for the sets $(y+\gc)^c,R_1,\dots,R_L$), we get that $$\lim_{N\rightarrow\infty} \frac{1}{|I_N|} \left|\left((y+\gc) \setminus \bigcup_{i=1}^L R_i\right)  \cap I_N\right| = \frac{1}{N(\gc)}\prod_{i=1}^L \left(1-\frac{|R_i|}{N(\gb_i)}\right).$$ Since $R$ is Erdős, the limit converges to some positive constant bigger than  $0$ as we take $L$ to infinity, which concludes the proof.
		\end{proof}
		
		\begin{remark}
			\label{rmk: if contains coprime ideal then no weak light tails}
			Assume that $R$ is an Erdős sieve for which there is some $x \not \in R_1$ such that $$x+\gb_1 \subset \bigcup_{i \geq 2} R_i.$$ Then, because  $ (x+\gb_1) \cap R_1  =\emptyset$, the computation done in  \Cref{lm: if contains coprime ideal then no weak light tails} also implies that $R$ does not have weak light tails for any Følner sequence. 
		\end{remark}

		\begin{lemma}
			\label{lm: aux for Characterization of equivalent sieves}
			Let $R$ and $R'$ be Erdős sieves such that $R$ has weak light tails for some Følner sequence $I_N$. If $R \sim R'$, then for any $\gb' \in \cb_{R'}$ we have $$R'_{\gb'} \subset \bigcup_{\gb \in \cb_{R}:(\gb,\gb') \neq 1} R_{\gb}.$$ In particular, there must be some $\gb\in \cb_R$ such that $(\gb,\gb') \neq 1$.
		\end{lemma}
		
		\begin{proof}
			We start by fixing some $\gb' \in \cb_{R'}$ and show that if there is no $\gb \in \cb_R$ such that $(\gb,\gb') \neq 1$, then $R \not \sim R'$. Indeed, if this is the case, then $(\gb',\gb)  =1$ for every $\gb \in \cb_R$. Take any $b$ such that $b+\gb' \subset R_{\gb'}$.  \Cref{lm: Fr intercepted by ideal} implies that $\fr \cap (b+\gb')  \neq \emptyset$, but $\frp \subset (b+\gb')^c$. It follows that $\fr \not\subset \frp$, so $R \not \sim R'.$

			In order to prove the result, we show that if there is some $\gb' \in \cb_{R'}$ such that  $$R'_{\gb'} \not \subset \bigcup_{\gb \in \cb_{R}:(\gb,\gb') \neq 1} R_{\gb}, $$ then $R \not \sim R'$. To do  this, we  assume that $R \sim R'$, and we will reach a contradiction.

			If $R \sim R'$, then $R'_{\gb'} \subset \bigcup_{\gb\in \cb_R }  R_\gb$. Writing $\gc = \text{lcm}(\{\gb \in \cb_R: (\gb,\gb') \neq 1 \}\cup\{\gb'\})$, and taking some $x$ such that $x \in R'_{\gb'} \setminus\bigcup_{(\gb,\gb') \neq 1} R_{\gb} $, we have that $x+\gb' \subset R'_{\gb'}$ and $(x + \gb)\cap  R_\gb = \emptyset$ for every $\gb$ such that $(\gb,\gb') \neq 1$. Therefore, \begin{equation}
				\label{eq: help lemma aux for Characterization of equivalent sieves 1}
				(x+\gc) \cap \bigcup_{\gb\in \cb_R:(\gb,\gb') \neq 1} R_{\gb} = \emptyset
			\end{equation}
			and it follows that
			\begin{equation}
				\label{eq: help lemma aux for Characterization of equivalent sieves 2}
				x+\gc \subset R'_{\gb'} \setminus \bigcup_{\gb\in \cb_R:(\gb,\gb') \neq 1} R_{\gb} \subset \bigcup_{\gb \in \cb_R}  R_{\gb} \setminus \bigcup_{\gb\in \cb_R:(\gb,\gb') \neq 1} R_{\gb} = \bigcup_{\gb\in \cb_R:(\gb,\gb') = 1   }  R_\gb.
			\end{equation}

			However, this cannot happen. Let $W$ be the sieve supported on the set $$\cb_W = \{\gc\} \cup \{\gb \in \cb_R: (\gb,\gb') = 1 \}$$ and defined by $$W_\gc = \bigcup_{\gb\in \cb_R:(\gb,\gb') \neq 1   }  R_\gb + \gc,$$ with $W_\gb = R_\gb$ if $\gb\in \cb_W$ is different from $\gc$. We see that $W$ is a dilation of $R$. Therefore, if $R$ has weak light tails for $I_N$, so does $W$ by  \Cref{lm: Erds and light tails preserved by change of base}. Yet, we can rewrite \Cref{eq: help lemma aux for Characterization of equivalent sieves 1} as $(x+\gc) \cap W_\gc = \emptyset$ and \Cref{eq: help lemma aux for Characterization of equivalent sieves 2} as  $(x+\gc) \subset \bigcup_{\gb \neq \gc} W_\gb$, which by \Cref{rmk: if contains coprime ideal then no weak light tails} implies that $W$ does not have weak light tails for any Følner sequence. We obtain the desired contradiction.	
		\end{proof}
		
		We can use  \Cref{lm: aux for Characterization of equivalent sieves} to provide a property that allows us to determine whether two sieves with weak light tails are equivalent.
		
		\begin{lemma}
			\label{lm:Characterization of equivalent sieves}
			Let $R$ and $R'$ be Erdős Sieves with weak light tails for some Følner sequence (not necessarily common to both). Then $R \sim R'$ holds if and only if for every $\gb' \in \cb_{R'}$, there is some $\gb \in \cb_{R}$ such that $(\gb,\gb') \neq 1 $, with $$R'_{\gb'} \subset \bigcup_{\gb \in \cb_{R}:(\gb,\gb') \neq 1} R_{\gb},$$ and if for every $\gb \in \cb_{R}$, there is some $\gb' \in \cb_{R'}$ such that $(\gb,\gb') \neq 1 $, with $$R_{\gb} \subset \bigcup_{\gb' \in \cb_{R'}:(\gb,\gb') \neq 1} R'_{\gb'}.$$
		\end{lemma}
		
		\begin{proof}
			If $R \sim R'$, then because both sieves have weak light tails, the fact that $R$ and $R'$ satisfy this property follows directly from  \Cref{lm: aux for Characterization of equivalent sieves}. Hence, all we have to show is that if this property holds, then $R \sim R'$. But this is clear, since for every $\gb \in\cb_R$ and $\gb' \in\cb_{R'}$, we are assuming that $R_\gb \subset \frp^c$ and $R'_{\gb'} \subset \fr^c$, which shows that $\fr = \frp$, and so $R \sim R'$. 
		\end{proof}
		
		With this, we can show that if $R$ has weak light tails, then the minimal sieve $R'$ obtained by repeated contraction of $R$ is the unique minimal sieve equivalent to $R$ with weak light tails. In order to prove this result, we need the following lemma.
		
		\begin{lemma}
			\label{lm: if x gb in union must belong to one}
			Let $\gb$ be an ideal of $\co_K$, and $\gb_1,\dots,\gb_r$ a collection of ideals such that $(\gb,\gb_i) \neq 1$. Let $R_{\gb_i}$ be a collection of congruence classes modulo $\gb_i$ such that $R_{\gb_i} \neq \co_K$. If $$x+\gb \subset \bigcup_{i=1}^r R_{\gb_i},$$ then there is some $j$ such that $x+\gb \subset R_{\gb_j}$.
		\end{lemma}
		
		\begin{proof}
			
			The first step is to show that if $x + \gb \subset R_{\gb_1} \cup R_{\gb_2}$, then $x + \gb$ is contained in either $R_{\gb_1}$ or $R_{\gb_2}$. Once we have this, then for any collection of ideals $\gb_1, \dots, \gb_r$, we can consider $\gc = \prod_{i=2}^r \gb_i$ and $R_\gc = \bigcup_{i=2}^r R_{\gb_i}$, such that the condition $x + \gb \subset \bigcup_{i} R_{\gb_i}$ becomes $x + \gb \subset R_{\gb_1} \cup R_{\gc}$. This will now imply that either $x + \gb \subset R_{\gb_1}$ or $x + \gb \subset R_{\gc}$. Using induction on $R_\gc$, we conclude that there is some $j$ such that $x + \gb \subset R_{\gb_j}$.
			
			By replacing $R_{\gb_i}$ by $R_{\gb_i}-x$, we reduce the problem to showing that if $\gb \subset R_{\gb_1} \cup R_{\gb_2}$, then $\gb$ is a subset of either $R_{\gb_1}$ or $R_{\gb_2}$. To show this, let us assume to the contrary, that $\gb$ is not a subset of either of these sets, but it is a subset of their union. Let $\phi: \co_K/\gb_1\gb_2 \rightarrow \co_K/\gb_1 \oplus\co_K/\gb_2 $ be the map that sends $x$ to $(x + \gb_1, x + \gb_2)$. Since $\gb_1$ and $\gb_2$ are coprime, the Chinese Remainder Theorem implies that it is a bijection. Let $\iota: \gb \rightarrow \co_K/\gb_1\gb_2$ be the map that sends $x$ to $x+\gb_1\gb_2$. Our hypothesis is that $\iota(\gb) \subset \phi^{-1}(R_{\gb_1}\times \co_K/\gb_2)\cup \phi^{-1}(\co_K/\gb_2 \times R_{\gb_2})$, but there are $x,y\in \gb$ such that $x \not \in R_{\gb_1}$ and $y \not \in R_{\gb_2}$. 
			
			Notice that the map $\phi$ is a bijection between $\iota(\gb)$ and elements of the form $(x+\gb_1,y+\gb_2)$ with $x,y \in \gb$. Indeed, it is injective by the Chinese Remainder Theorem, and it is clear that all elements of $\phi(\iota(\gb))$ are of the form given. So, to see the bijectivity, it is enough to check that both the domain and image have the same size. The image of $\gb$ inside of $\co_K/\gc$ for any ideal $\gc$ has cardinality $N(\gc)/N\left[(\gb,\gc)\right]$ (here $N\left[(\gb,\gc)\right]$ corresponds to the norm of the ideal $\gb+ \gc$), so the result follows from $$\frac{N(\gb_1\gb_2)}{N\left[(\gb,\gb_1\gb_2)\right]} = \frac{ N(\gb_1)}{N\left[(\gb,\gb_1)\right]} \frac{N(\gb_2)}{N\left[(\gb,\gb_2)\right]},$$ which is a consequence of $\gb_1, \gb_2$ being coprime, and multiplicativity of the norm. 
			
			Therefore, it cannot happen that there are $x,y\in \gb$ such that $x \not \in R_{\gb_1}$ and $y \not \in R_{\gb_2}$, since under these conditions $\phi^{-1}(x + \gb_1,y + \gb_2 )$ is not in $\phi^{-1}(R_{\gb_1}\times \co_K/\gb_2)\cup \phi^{-1}(\co_K/\gb_2 \times R_{\gb_2})$, but we have just shown it must be in $\iota(\gb)$.
		\end{proof}
		
		We can now show that if $R$ has weak light tails for $I_N$, we can contract it until we get a minimal sieve with weak light tails for $I_N$, and that this sieve is the unique minimal sieve equivalent to $R$ that has weak light tails for some Følner sequence.
		
		\begin{theorem}
			\label{thm: uniqueness of minimal sieve}
			Let $R$ be an Erdős sieve. If there exists a minimal sieve $R'$  equivalent to $R$ with weak light tails for any Følner sequence $I_N$, then it is the unique minimal sieve with weak light tails for $I_N$ that is equivalent to $R$.
		\end{theorem}
		
		\begin{proof}

			Let $W$ and $W'$ be two minimal sieves equivalent to $R$ with weak light tails for some (possibly distinct) Følner sequences, supported on the sets $\cb_W$ and $\cb_{W'}$ respectively. We want to show that $W = W'$. By transitivity, we have $W \sim W'$, and so  \Cref{lm:Characterization of equivalent sieves} tells us that for any $\gb \in \cb_W$, we have $$W_\gb \subset \bigcup_{\gc \in\cb_{W'}: (\gb,\gc) \neq 1 } W'_\gc. $$
			
			Take any $x \in W_\gb$.  By \Cref{lm: if x gb in union must belong to one},  we have that $x+\gb$ is contained in some $W'_\gc$. Hence, we have that $x+\gb + \gc \subset W'_\gc +\gc$, which means that $x+(\gb,\gc) \subset W'_\gc$. Since $W \sim W'$, it follows that $x+(\gb,\gc) \subset \cf_W^c$. We now show that this implies that $x+(\gb,\gc) \subset W_\gb$. Take any $y \in(\gb,\gc) $. If $(x+y+\gb)\cap W_\gb = \emptyset$, then because $x+y+\gb \subset \cf_W^c$,  \Cref{rmk: if contains coprime ideal then no weak light tails} would imply that $W$ does not have weak light tails for any Følner sequence. Since this is not the case, it follows that $(x+y+\gb) \subset W_\gb$ for any $y \in(\gb,\gc)$, so we must have $x+(\gb,\gc) \subset W_\gb$.

			Repeating this for every $x \in W_\gb$, we get that there are $\gc_1, \dots, \gc_r \in \cb_{W'}$ with $(\gb,\gc_i) \neq 1$, such that $x+(\gb,\gc_i) \subset W_\gb$. We get $W_\gb = \bigcup_{i=1}^r A_i + (\gb,\gc_i)$ for some finite sets $A_i$ (corresponding to those $x \in W_\gb$ such that $x+ (\gb,\gc_i) \subset W_\gb$). Since $W$ is minimal, this implies that there is at least one $i$ such that $(\gb,\gc_i) = \gb$, which implies that $\gb \mid \gc_i $. Since the elements of $\cb_{W'}$ are pairwise coprime,  $\gc_i$ must be the unique ideal in $\cb_{W'}$ that is not coprime to $\gb$, and so $W_\gb \subset W'_{\gc_i}$. Let us denote this $\gc_i$ by $\gc$.
			
			By reversing the argument, now using the minimality of $W'$, we will get that there is a unique ideal $\gb' \in \cb_W$ such that $\gc \mid \gb'$ and $W'_\gc \subset W_{\gb'}$. This together with $\gb \mid \gc$ implies that $\gb = \gb'$, and so, in fact, we must have $\gb = \gc$ and $W_\gb = W'_\gc$. Repeating this for every ideal, we conclude that $\cb_W = \cb_{W'}$, and that $W$ and $W'$ are the same sieve. Consequently, a sieve $R$ can only be equivalent to one unique minimal sieve with weak light tails for some $I_N$.                            
		\end{proof}
		
		\begin{remark}
			\label{rmk: equivalent sieves with and without weak light tails}
			First, notice that an Erdős sieve may be not equivalent to any minimal sieve with weak light tails. Indeed, the sieve $R_i = \{-i,i\} +p_i^2\z$ is such that $d_I(\fr) = 0$ for every Følner sequence, so if an Erdős sieve $R'$ with weak light tails for $I_N$ were to be equivalent to it, we would satisfy $0 = d_I(\fr) = d_I(\frp) >0$, which is absurd.
			
			On the other hand, notice that this theorem does not imply that if a sieve is minimal and has weak light tails, that then  there are no other minimal sieves equivalent to it, just that these sieves will not have weak light tails. Take the sieve $W$ defined by $$W_1 = 1+4\z \hspace{15pt} W_{2i} = 1+4(i-1) +p_{2i}^2\z \hspace{15pt}  W_{2i+1} = 1-4(i-1)+p_{2i+1}^2\z$$ and $W'$ defined by $W'_i = W_{i+1}$. In Example 5.14 of \cite{part1} we showed that $W$ has weak light tails for $I_N = [0,N]$, while $W'$ does not. Yet, both are minimal and they are equivalent. This shows that although $W$ and $W'$ are equivalent, they cannot be obtained from one another by contractions or dilations.
			
			Additionally, notice that if $R$ has weak light tails for some sequence $I_N$, and $R'$ is this unique minimal sieve with weak light tails for $I_N$, then $R$ will have weak light tails for some other Følner sequence $F_N$ if and only if $R'$ also does.
		\end{remark}
		
		Yet, if we assume that a sieve $R$ has strong light tails, then there is a unique minimal sieve $R'$ such that $R \sim R'$. This is a consequence of  \Cref{thm: uniqueness of minimal sieve} together with the following result, which is of interest in itself.
		
		\begin{theorem}
			\label{thm: Strong light tails preserved under isomorphism of sieves}
			Let $R$ and $R'$ be sieves, such that $R$ is Erdős and has strong light tails for some Følner sequence $I_N$. If $R \sim R'$, then $R'$ is also Erdős with strong light tails for $I_N$. 
		\end{theorem}
		
		\begin{proof}
			We will assume that both $\cb_R$ and $\cb_{R'}$ are ordered. We have that $$\nu_{R}(C^R_{\{0\}, \emptyset}) = \prod_{i}\left(1-\frac{|R_i|}{N(\gb_i)}\right) >0,$$ if and only if $R$ is Erdős. Since $R$ has strong light tails, we have that  gives us $$0 <d_I(\fr) = d_I(\frp) \leq \nu_{R'}(C_{\{0\}, \emptyset}),$$ which means that $R'$ must be Erdős. 
			
			We now show that $R'$ has strong light tails for $I_N$. By  \Cref{lm: aux for Characterization of equivalent sieves}, we know that since $R \sim R'$ and $R$ has strong light tails, we have for every $\gb' \in \cb_{R'}$,
			\begin{equation}
				\label{eq: aux bp in b}
				R'_{\gb'} \subset \bigcup_{\gb \in \cb_R:(\gb,\gb') \neq 1} R_{\gb}.
			\end{equation} $$$$ 
			
			Define a function $g:\mathbb{N}\rightarrow\mathbb{N}$ by $$g(i):= \min \{j \in \mathbb{N}: (\gb_j,\gb_i') \neq 1\}.$$ Notice that for any $j$, there are at most a finite number of $i's$ such that $g(i) = j$ (with an upper bound given by the number of prime divisors of $\gb_j$). It follows that $$\liminf_{l \rightarrow \infty} g(l) = \infty.$$ 
			
			By \Cref{eq: aux bp in b} every $R'_i$ is contained in $\bigcup_{j: (\gb_j,\gb'_i) \neq 1} R_j.$ It follows that for any $i$ we have  $$R'_i \subset \bigcup_{j \geq g(i)} R_{j},$$ and so, writing $G(L) = \min_{l \geq L}g(l)$, we have $$\bigcup_{i \geq L} R'_i \subset \bigcup_{j \geq G(L)} R_{j},$$ and so 			
			$$\overline{d}_I\left(\bigcup_{i \geq L} R'_i\right) \leq \overline{d}_I\left(\bigcup_{j \geq G(L)} R_{j}\right). $$
			
			Since $\liminf_{l \rightarrow \infty} g(l) = \infty,$ we have that $G(L)$ must go to infinity as we increase $L$. Since $R$ has strong light tails, the right hand side will go to $0$ as we take $L$ to infinity, so $R'$ must also have strong light tails.
		\end{proof}
		
		\begin{remark}
			Note that $R \sim R'$ with $R$ an Erdős sieve, does not necessarily imply that $R'$ is Erdős.  For example, $R_i = \{-i,i\} + p_i^2 \z$ is clearly Erdős. Taking any non Erdős sieve $R'$ such that $\frp = \fr$, for example $R'_p = (\z\setminus p\z)  + p\z$, we see that $R \sim R'$, but $R'$ is not Erdős. 
		\end{remark}
		
		We summarize the results of this subsection in the following theorem.
		
		\begin{theorem}
			\label{thm: equivalence minimal sieves main}
			Let $R$ be an Erdős sieve. 
			\begin{itemize}
				\item There exists a minimal Erdős sieve $R'$ such that $R \sim R'$.
				\item If $R$ has weak light tails for some Følner sequence $I_N$, there exists a minimal Erdős sieve $R'$  with weak light tails for $I_N$, such that if $W$ is minimal and $W \sim R$, then $W = R'$, or $W$ does not have weak light tails for any Følner sequence.
				\item If $R$ has strong light tails for $I_N$, then there exists a unique minimal sieve $R'$ (which will have strong light tails for $I_N$) such that $R \sim R'$. 
			\end{itemize}
		\end{theorem}
		
		Here, we write for two sieves $R$ and $R'$ that $R= R'$\index{Equality of Sieves}, if $\cb_R = \cb_{R'}$, and $R_\gb = R'_\gb$ for every $\gb \in \cb_R$. Since every Erdős $\cb-$free system is a minimal sieve which has strong light tails by  \Cref{thm: K etale algebra light tails for B_n finite}, we get the following corollary.

		\begin{corollary}
			\label{crl: Bfree systems free elements equal iff same support}
			Let $R$ and $R'$ be an Erdős $\cb-$free systems over an étale $\q-$algebra $K$. Then $R \sim R'$ if and only if $R = R'$.
		\end{corollary}

		\subsection{Union of Sieves}
		
		Suppose that we are given two sieves $R$ and $R'$, which we can assume to be minimal. We want to define the notion of union of sieves. If $R$ and $R'$ are defined over the same set of ideals $\cb$, then we can define this simply as $(R\cup R')_i := R_i \cup R'_i + \gb_i$.

		More generally, given two sieves $R$ and $R'$, it might be possible to dilate them in such a way that we obtain sieves $W$ and $W'$, such that $\cb_W = \cb_{W'}$, and then we can define $(R\cup R')_\gc = W_\gc \cup W'_\gc $ for $\gc \in \cb_W$.

		We therefore need to find a suitable set $\mathcal{C}$ on which both $W$ and $W'$ will be supported. The obvious choice would be $\cb_{R} \cup \cb_{R'}$, but we want $\mathcal{C}$ to be made of pairwise coprime ideals, so in general this won't do. Surely, if $\gb \in \cb_{R}$ is such that $(\gb, \gb') = 1$ for every $\gb' \in \cb_{R'}$, then it would make sense to add $\gb$ to $\mathcal{C}$. If this is not the case, we might still have that for example  $(\gb, \gb') \neq 1$ for some unique $\gb'\in \cb'$ (such that $\gb$ is also the unique element of $\cb$ such that $(\gb,\gb')\neq 1$). Then it would still make sense to add $\gc = \text{lcm}(\gb,\gb')$ to $\mathcal{C}$, and define $W_\gc = R_\gb + \gc$ and $W'_\gc = R'_{\gb'} + \gc$.

		These considerations lead us to the following definition. We define a graph $\mathcal{G}_{R,R'}$, whose vertices will be the elements of $\cb_R \cup \cb_{R'}$, and where there is an edge between $\gb$ and $\gb'$ if $(\gb,\gb') \neq 1$. It is clear that there can only be an edge between vertices if the corresponding ideals do not come from the same set (so  $\mathcal{G}_{R,R'}$ will always be a bipartite graph). The graph $\mathcal{G}_{R,R'}$ will be the union of distinct connected components, which may or may not be infinite. We will now use this graph to define the common base for our sieves, assuming that $\mathcal{G}_{R,R'}$ does not contain an infinite component.
		
		\begin{definition}
			Let $R$ and $R'$ be two sieves. Let $\mathcal{G}_{R,R'}$ be the graph with set of vertices $\cb_R \cup \cb_{R'}$, and an edge between $\gb, \gb'$ if $(\gb,\gb') \neq 1$. Let $\mathcal{C^*}(R,R')$ denote the collection of connected components of $\mathcal{G}_{R,R'}$.  Given $c \in \mathcal{C^*}(R,R')$ finite, we define $$\gc(c) := \text{lcm}(\{\gb: \gb \in c\}).$$  If every element of $\mathcal{C^*}(R,R')$ is finite, then we set $\mathcal{C}(R,R') = \gc(\mathcal{C^*}(R,R'))$ and say that $R$ and $R'$ have a common basis.\index{Sieves with a Common Basis}
		\end{definition} 
		
		We must verify that the ideals in $\mathcal{C}(R,R')$ are indeed pairwise coprime. Suppose that we can find some distinct $\ga_1,\ga_2$ such that $(\ga_1,\ga_2) \neq 1$. Then, there would be distinct connected components $c_1$ and $c_2$ of $\mathcal{G}_{R,R'}$ such that $\ga_i = \gc(c_i)$, and $(\gc(c_1),\gc(c_2)) \neq 1$. But it is a property of the least common multiple, that if $A$ and $B$ are finite sets, and $(\text{lcm}(A),\text{lcm}(B)) \neq 1$, then $(a,b) \neq 1$ for some $a\in A$ and $b \in B$. This would imply that there is an edge between elements of $c_1$ and $c_2$, which is impossible, as these were taken to be distinct connected components.
		
		\begin{remark}
			\label{rmk: sieves without common base}
			It might happen that two sieves don't have a common basis. Consider the example where we define two sieves by $R_i = p_{2i-1}p_{2i}\z$ and $R'_i = p_{2i}p_{2i+1}\z$. We see that for any $i$, $ p_{2i-1}p_{2i}$ is not coprime to $ p_{2i}p_{2i+1}$, which is not coprime to $ p_{2i+1}p_{2i+2)}$.
			
			Additionally, two sieves might not have have a common base, but be equivalent to sieves that do have a common base. Let $q_i$ denote the $i-$th prime that is congruent to $1 \mod 4$, and  $r_i$ the $i-$th prime that is congruent to $3 \mod 4$. Let $R$ be the sieve supported on the set $\cb_{R} = \{q_i^2r_i^2\z: i \in \mathbb{N}\}$ and defined by $R_i = q_i^2\z + q_i^2r_i^2\z$. Define also a sieve $R'$, with support on the set $\cb_{R'} = \{q_i^2r_{i+1}^2\z: i \in \mathbb{N}\}$,  by $R'_i = r_{i+1}^2\z + q_i^2r_{i+1}^2\z$. These sieves don't have a common basis, since for every $i$, $(q_i^2r_i^2, q_i^2r_{i+1}) \neq 1$ and $(q_{i+1}^2r_{i+1}^2, q_i^2r_{i+1}) \neq 1$, meaning that $\mathcal{G}_{R,R'}$ will be a connected infinite graph. Yet, we see that $R$ is equivalent to the sieve $W$ such that $\cb_W = \{q_i^2\z:i \in \mathbb{N}\}$ and $W_i = q_i^2\z$, and $R'$ is equivalent to the sieve $W'$ such that  $\cb_{W'} = \{r_{i}^2\z:i \in \mathbb{N}\}$ with $W'_1 = \emptyset$ and $W'_i = r_i^2\z$ otherwise. Although $R$ and $R'$ don't have a common base, $W$ and $W'$ do have one, since $\mathcal{G}_{W,W'}$ is a graph with no edges.
			
			Note that if $R$ and $R'$ are two sieves, and $W,W'$ are contractions of $R$ and $R'$ respectively, then there is a map from  $\mathcal{G}_{W,W'}$ into $\mathcal{G}_{R,R'}$ that sends to $\gb \in \cb_R$ all those $\gc \in \cb_{W}$ that divide $\gb$  (and similarly for ideals of $\cb_{R'}$ and $\cb_{W'}$). Additionally, if there is an edge from $\gc \in \cb_W $ to $\gc' \in \cb_{W'}$, this will be sent to the edge between the $\gb \in \cb_R$ divisible by $\gc$ and the  $\gb' \in \cb_{R'}$ divisible by $\gc'$. It follows that if there is an infinite connected component in $\mathcal{G}_{W,W'}$, it will be sent by this map into one in $\mathcal{G}_{R,R'}$. Hence, two sieves can only have a common base if the  minimal sieves to which they equivalent to (from \Cref{lm: every sieve is equivalent to minimal sieve}) have a common basis.
		\end{remark}
		
		If $R$ and $R'$ have a common basis, we define the sieves $W$ and $W'$ by $$W_{\gc(c)} = \bigcup_{\gb \in c\cap \cb_R} (R_\gb + \gc(c)) \hspace{30pt} W'_{\gc(c)} = \bigcup_{\gb' \in c\cap \cb_{R'}} (R'_{\gb'} + \gc(c)).$$
		Then $W$ is equivalent to $R$, since $R_\gb + \gc(c) = R_\gb$, and every $\gb \in \cb_R$ will belong to one connected component of $\mathcal{G}_{R,R'}$. Notice that it may happen that for some  $c \in \mathcal{C^*}(R,R')$, we have $c\cap \cb_R = \emptyset$ or $c\cap \cb_{R'} = \emptyset$. In this case, we will have $W_{\gc(c)} = \emptyset$ or $W'_{\gc(c)} = \emptyset$. Additionally, if $R$ and $R'$ are Erdős with weak/strong light tails for some $I_N$, then  \Cref{lm: Erds and light tails preserved by change of base} guarantees that so are $W$ and $W'$.

		This allows us to define the union of $R$ and $R'$ as the union of $W$ and $W'$. That is, as the sieve $R \cup R'$ defined over $\mathcal{C}$ by $$(R \cup R')_{\gc(c)} = W_{\gc(c)} \cup W'_{\gc(c)} = \bigcup_{\gb \in c\cap \cb_R} (R_\gb + \gc(c)) \cup \bigcup_{\gb' \in c\cap \cb_{R'}} (R'_{\gb'} + \gc(c)).$$ If $(R \cup R')_{\gc(c)} \neq \co_K$ for every $c \in\mathcal{C^*}$, then we say that the union of $R$ and $R'$ is well defined. \index{Union of Sieves}

			\begin{example}
				Let $K$ be an imaginary quadratic number field, and for every $p$ that splits in $K$, let $\gp_p$ be the prime such that $\gp_p \overline{\gp}_p = p \co_K$. Define a sieve $R$ as the $\cb-$free system supported on the set  $\cb_R = \{\gp_p^2: p \text{ splits in } \co_K\}$ and a sieve $R'$ as the $\cb-$free system supported on $\cb_{R'} = \{\overline{\gp}_p^2: p \text{ splits in } \co_K\}$. Then, we get a sieve $R\cup R'$ which is the $\cb-$free system supported on $\cb_R \cup \cb_{R'}$.
				
				For another example, let $q_i$ and $r_i$ be as in \Cref{rmk: sieves without common base}, and define the sieves $R$ and $R'$ by $R_i =q_i + q_i^2\z$, $R'_i = r_i + r_i^2\z$. Then, we have that $R \cup R'$ is the sieve supported on $\{q_i^2r_i^2\z:i \in \mathbb{N}\}$ and defined by $(R \cup R')_i = ((q_i+q_i^2\z) \cup ( r_i + r_i^2\z)) + q_i^2r_i^2\z$.
				
				On the other hand, if $R$ and $R'$ are sieves such that $R_1 = 0+2\z$ and $R'_1 = 1+2\z$, then their union is not well defined, since $R_1 \cup R'_1 = \z$.
			\end{example}
			
			Given any $\gb \in \cb_R \cup \cb_{R'}$, take the $\gc \in \mathcal{C}$ such that $\gb \mid \gc$. Then, it is clear from the definition that $R_\gb = R_\gb + \gc \subset (R \cup R')_{\gc}$. It follows that

			$$\mathcal{F}_{R \cup R'} = \left(\bigcup_{\gc \in \mathcal{C}} (R \cup R')_{\gc}\right)^c 
			= \left( \bigcup_{\gb \in \cb} R_\gb \cup \bigcup_{\gb \in \cb'} R'_\gb\right)^c = \fr \cap \frp. $$  
			
			Our interest in the union of sieves is twofold. First, as we have just shown, the $(R\cup R')-$free elements correspond to the intersection of $\fr$ and $\frp$. Second, the union of sieves allows us to take sieves with weak/strong light tails, and obtain new sieves with the same properties. This is shown in the following lemma.
			
			\begin{lemma}
				\label{lm: union of sieves}
				Let $R$ and $R'$ be two sieves for which their union is well defined and $I_N$ some Følner sequence. The sieve $R \cup R'$ is Erdős if and only if both $R$ and $R'$ are Erdős. If both $R$ and $R'$ have weak light tails for $I_N$, then so does $R \cup R'$. Additionally, $R \cup R'$ will have strong light tails for some $I_N$ if and only if both $R$ and $R'$ have strong light tails for $I_N$.

			\end{lemma}
			
			\begin{proof}
				By dilating $R$ and $R'$ if necessary, we are free to assume that $R$ and $R'$ are defined over a common base $\cb$, and that $(R \cup R')_\gb = R_\gb \cup R_{\gb'}$. Hence, using the fact that $$\max(|R_\gb|,|R_\gb'|) \leq |(R \cup R')_\gb| \leq |R_\gb|+|R_\gb'|,$$ we see that $R \cup R'$ is Erdős if and only if both $R$ and $R'$ also are.
				
				Similarly, we have that, ordering $\cb = \{\gb_1,\gb_2,\dots\}$,
				$$ \max\left(\left|I_N \cap \bigcup_{i > L} R_i \right|, \left|I_N \cap \bigcup_{i> L} R'_i \right|\right)  \leq \left|I_N \cap \bigcup_{i>L} (R\cup R')_{i}\right| \leq \left|I_N \cap \bigcup_{i> L} R_i \right| + \left|I_N \cap \bigcup_{i>L} R'_i \right|, $$  so $(R \cup R')$ will have strong light tails for $I_N$ if and only if both $R$ and $R'$ have strong light tails with respect to $I_N$.
				
				Note that if $x \not \in (R \cup R')_i$, then both $x \not \in R_i$ and $x \not \in R'_i$, and so
				$$\left|I_N \cap \bigcup_{i>L} (R\cup R')_{i} \setminus \bigcup_{j \leq L}(R\cup R')_{j}\right| \leq \left|I_N \cap \bigcup_{i\geq L} R_i\setminus \bigcup_{j<L} R_j \right| + \left|I_N \cap \bigcup_{i> L} R'_i\setminus \bigcup_{j<L} R'_j \right|.$$
			
				Therefore if both $R$ and $R'$ have weak light tails for $I_N$, so does $R \cup R'$. 
			\end{proof}

			\begin{remark}
				Contrarily to the strong light tails property, it might be the case that $R\cup R'$ has weak light tails for $I_N$, without $R$ and $R'$ having both weak light tails for $I_N$. Consider the example where $R$ is defined by $R_1 = 0+4\z $ and $R_{i} = 1+4i +p_{i}^2\z$ and $R'$ is defined by $R'_1 = 1+ 4\z$ and $R'_i = p_i^2\z$ for any $i>1$. Proceeding as we did in Example 5.14 of \cite{part1}, we can show that $R$ does not have weak light tails for $B_N$ (but $R'$ has strong light tails for $B_N$ by  \Cref{thm: K etale algebra light tails for B_n finite}). The sieve $W$ defined by $W_1 = \{0,1\}+4\z$,  $W_{i} = \{0,1+4i\} +p_{i}^2\z$ is the union of $R$ and $R'$. It has weak light tails for $B_N$, as it can be written as the union of two sieves with weak light tails, $T$ and $T'$, defined by $T_1 = 1+4\z$,$T'_1 = 0+4\z$, $T_i = R_i$  and $T'_i = R'_i$ for $i>1$. 
			\end{remark}

			We provide an example over $\q \times \q$.
			
			\begin{example}
				\label{ex: carefree couples}
				In \cite{Moree}, the notion of a \textit{carefree couple}\index{Carefree Couples} is defined as a pair $(x,y) \in \z \times \z$, such that $x$ and $y$ are coprime and $x$ is squarefree. The sieve $R$ defined by $R_p = p\z\times p\z$ is such that $(x,y) \in \fr$ is equivalent to $x$ and $y$ being coprime. The sieve $R'$ defined by $R'_p = p^2\z \times \z$ is such that $(x,y) \in \frp$ is equivalent to $x \not \in p^2\z$ for every $p$, that is, to $x$ being squarefree. Consequently, the set of carefree couples corresponds to $\fr \cap \frp = \cf_{R \cup R'}$.
				
				To see that $R \cup R'$ is well defined, notice that $R$ is equivalent to the sieve $W$ defined by $$W_p = \{(jp,0):0 \leq j \leq (p-1)\}+p^2\z \times p\z$$ and $R'$ is equivalent to the sieve $W'$ defined by $$W'_p = \{(0,j): 0 \leq j \leq p-1\} + p^2\z\times p\z.$$ Hence, $R$ and $R'$ have a common basis, and writing $(R \cup R')_p = W_p \cup W'_p+ p^2\z\times p\z$, we see that $|(R \cup R')_p| = 2p-1< p^3$, so $R \cup R'$ is well defined and Erdős.
				
				Note that both $R$ and $R'$ have strong light tails for $B_N$ by  \Cref{thm: K etale algebra light tails for B_n finite}.  Consequently, by  \Cref{lm: union of sieves}, $R \cup R'$ has strong light tails for $B_N$. We can therefore compute the density of carefree pairs to be $$d(\cf_{R \cup R'}) = \prod_p\left(1-\frac{2p-1}{p^3}\right) = \frac{1}{\zeta(2)}\prod_p\left(1-\frac{1}{p(p+1)}\right) $$ as was also shown in  \cite{Moree}.
				
				If instead we want to consider the set of \textit{strongly carefree couples}\index{Strongly Carefree Couples}, where $(x,y)$ is squarefree and $y$ is also squarefree, we could instead consider the sieve $$Z_p =  \{(jp,kp):0 \leq j,k \leq (p-1)\} \cup \{(j,0): 0 \leq j \leq p^2-1\} \cup \{(0,k): 0 \leq k \leq p^2-1\} + p^2\z \times p^2\z, $$ which is such that $\cf_Z$ is the set of strongly carefree couples. Again by \Cref{thm: K etale algebra light tails for B_n finite} this is the union of three sieves with strong light tails with respect to $B_N$, so we can use  \Cref{lm: union of sieves} to conclude that it has strong light tails for $B_N$. Since $|Z_p| = 3p^2 -(p+p+1)+1 = 3p^2-2p$, we conclude that the density of strongly carefree couples is $$d(\cf_Z) = \prod_p \left(1-\frac{3p^2-2p}{p^4}\right) = \frac{1}{\zeta(2)^2}\prod_p\left(1-\frac{1}{(p+1)^2}\right).$$
			\end{example}

			\section{Spectrum and Equivalence of Dynamical Systems}

			In this section, we investigate the dynamical system $(\Omega_R,S, \nu_R)$.  First, we will give condition for when $\Omega_R = \Omega_{R'}$. We then show that $(\Omega_R,S, \nu_R)$ is isomorphic to a rotation of a compact group, and use this result to compute the spectrum of this system.

			\subsection{Equality of Dynamical Systems} We want to characterize when $\Omega_{R} = \Omega_{R'}$, and show that whenever this is the case, then $\nu_R = \nu_{R'}$. We start by doing this for sieves supported on the same set, which was already done in \cite{Fabian}. We provide a proof since some of the lemmas used to prove this result are needed for when we extend it.
			
			\begin{lemma}
				\label{lm: there are admissible representatives of R_gb^c  }
				Let $R$ be an Erdős sieve. For every $\gb \in \cb_R$, there is some $R-$admissible set $A$ such that $A+\gb = R_\gb^c$. 
			\end{lemma}
			
			\begin{proof}
				Fix some $\gb \in \cb_R$, and define the set $$\Delta = \left\{\gb' \in \cb_R: \frac{|R_{\gb'}|}{N(\gb')} \geq \frac{1}{|R_\gb^c|} \right\}.$$ Because $R$ is Erdős, we have that $\Delta$ is finite. Therefore, we can use the Chinese Remainder Theorem to find $a_j$ with $1 \leq j \leq |R_\gb^c|$ such that each $a_j$ belongs to  a different congruence class not in $R_\gb$, and $a_j \not \in R_{\gb'}$ for any $\gb' \in \Delta$. 
				
				Let $A$ be the set containing each of these $a_j$. We will show that it is admissible, which implies the result. For any $\gb' \in \Delta \cup \{\gb\}$, we have by definition of $A$ that $A\cap R_{\gb'} = \emptyset$, so we are left with showing that for $\gb' \not \in \Delta \cup \{\gb\}$, we have $-A+R_{\gb'} \neq \co_K$.  But for any such $\gb'$, we have $|A|\frac{|R_{\gb'}|}{N(\gb')} < 1$. Therefore, the set $-A + R_{\gb'}$ cannot cover $\co_K$, and so $A$ is admissible.
			\end{proof}
			
			Using this we can show the following lemma.
			
			\begin{lemma}
				\label{lm: Conditions for Omega r subset omega s}
				Let $R, R'$ be two Erdős sieves supported on the same set $\cb$. Then, $\Omega_{R} \subset \Omega_{R'}$ if and only if for every $i$, there is some $\delta_i$ such that $\delta_i + R'_i \subset R_i$. 
			\end{lemma}

			\begin{proof}
				Assume that for every $i$ there is some $\delta_i$ such that $R'_i \subset (R_i-\delta_i)$. For any $A$ that is $R-$admissible there is some $\epsilon_i$ such that $(\epsilon_i + A) \cap R_i = \emptyset$,  so we have $$(\epsilon_i-\delta_i + A) \cap R'_i \subset (\epsilon_i-\delta_i + A) \cap   (R_i-\delta_i) = \emptyset,$$ which means that $A$ is also $R'$-admissible.
				
				Let us now prove the other implication. By  \Cref{lm: there are admissible representatives of R_gb^c  } we can, for each $i$, find some $R-$admissible set $A$ such that $A + \gb_i = R_i^c$. Since $\Omega_{R} \subset \Omega_{R'}$, $A$ is also $R'-$admissible, so there is some $\delta_i$ such that $A+\delta_i \subset (R'_i)^c$, which means that $-\delta_i + R'_i \subset R_i$. 
			\end{proof}
			
			It is now easy to show the following equivalence.
			
			\begin{lemma}
				\label{lm: OmegaR = OmegaS same base}
				Let $R$ and $R'$ be two Erdős sieves supported on the same base $\cb$. Then $\Omega_{R} = \Omega_{R'}$ if and only if for every $i$, there is some $\delta_i$ such that $\delta_i + R_i = R'_i$.
			\end{lemma}
			
			\begin{proof}
				If there is some $\delta_i$ such that $\delta_i + R_i = R'_i$, then $ A \cap (\delta_i + R_i) = A \cap R'_i, $ so it is clear $A$ is in $\Omega_{R}$ if and only if it is in $\Omega_{R'}$.
				
				On the other hand, assume that $\Omega_{R} = \Omega_{R'}$. By  \Cref{lm: Conditions for Omega r subset omega s}, for every $i$, there are some $\delta_i,\delta_i'$ such that $\delta_i + R_i \subset R'_i$, and $\delta'_i + R'_i \subset R_i$. But this implies that $\delta_i + \delta'_i + R'_i \subset R'_i$, and since $|\delta_i + \delta'_i + R'_i| = |R'_i|$, it follows that they must be the same. Therefore, we get the relations $$R_i \subset -\delta_i + R'_i  = \delta'_i + R'_i \subset R_i,$$ which imply that $R_i = \delta'_i +R'_i$.
			\end{proof}

			As a corollary, we get the following result.
			
			\begin{lemma}
				\label{lm: OmegaR = OmegaR' implies same meausre same base}
				Let $R$ and $R'$ be Erdős sieves. If $\Omega_R = \Omega_{R'}$ and $\cb_R = \cb_{R'}$, then $\nu_R = \nu_{R'}$.
			\end{lemma}
			\begin{proof} 
				
				By  \Cref{lm: OmegaR = OmegaS same base}, there is a sequence $\delta_i$ of elements of $\co_K$ such that $R_i = -\delta_i + R'_i$ for each $i$. Let $G := G_R = G_{R'}$ be the group defined in \Cref{eq:GR} and consider the map $V:G \rightarrow G$, such that $V(g)_i = g_i+\delta_i$. Notice that $\varphi_R = \varphi_{R'} \circ V$, given that $$y \in \varphi_R(g) \Leftrightarrow \mathlarger{\forall}_j \hspace{5pt} g_j + y \not \in R_j \Leftrightarrow \mathlarger{\forall}_j \hspace{5pt} V(g)_j + y \not \in \delta_j + R_j \Leftrightarrow \mathlarger{\forall}_j\hspace{5pt} V(g)_j + y \not \in R'_j  \Leftrightarrow  y \in \varphi_{R'}(V(g)). $$ 
				For any measurable $U$ we get  $$\nu_{R'}(U) = \mmp(\varphi_{R'}^{-1}(U)) = \mmp(V(\varphi_{R}^{-1}(U))) = \mmp(\varphi_{R}^{-1}(U)) = \nu_R(U),$$ using the fact that $V$ preserves the measure of $G$.
			\end{proof}

			When $R$ and $R'$ are not supported on the same set, we would like to have a similar result. If we could show that $\Omega_R = \Omega_{R'}$ already implies that $\cb_R = \cb_{R'}$, then we could simply remove this hypothesis from \Cref{lm: OmegaR = OmegaR' implies same meausre same base}. However, this is not the case, since, as we now show, dilations don't change $\Omega_R$ or $\nu_R$. Afterwards, we will show that if we restrict ourselves to considering minimal sieves, then it is the case that $\Omega_R = \Omega_{R'}$ implies that $\cb_R = \cb_{R'}$.
			
			\begin{lemma}
				\label{lm: isomorphism of dynamical systems for contraction dilation}
				Let $R$ be an Erdős sieve, and $R'$ a dilation of $R$. Then $\Omega_R = \Omega_{R'}$ and $\nu_R = \nu_{R'}.$ In particular, the identity map is an isomorphism of the dynamical systems $(\Omega_{R},S,\nu_{R})$ and $(\Omega_{R'},S,\nu_{R'})$.
			\end{lemma}
			
			\begin{proof}
				Let $R$ be an Erdős sieve. To define a dilation, let $\mathcal{P}$ be a partition of $\mathbb{N}$, $\mathcal{A}$ a collection of ideals $\ga_A$ indexed on $\mathcal{P}$ such that $(\ga_A,\ga_B) =1$ if $A \neq B$, and $\mathcal{C}$ the collection of ideals of the form $\gc(A) = \text{lcm}(\{\ga_A\}\cup\{\gb_i:i\in A\})$ for $A \in \mathcal{P}$. Let $R'$ be the associated dilation, that is, the sieve supported on $\mathcal{C}$ defined by $$R_A' = \bigcup_{i \in A} R_i + \gc(A).$$
				
				First, we point out that $\Omega_{R} = \Omega_{R'}$. To see this, take any $B \in \Omega_R$. For every $i$, there is some $\delta_i$ such that $(\delta_i + B + \gb_i ) \cap R_i = \emptyset$. Using the Chinese Remainder Theorem, define for every $A \in \mathcal{P}$ some $\delta_A$ such that $\delta_A \equiv \delta_i \mod \gb_i$ for every $i \in A$. Then
				\begin{equation}
					\label{eq: equality with deltas}
					(\delta_A + B) \cap R'_A = \bigcup_{i \in A} \left(\delta_A + B \right) \cap R_i \subset  \bigcup_{i \in A} \left(\delta_A + B + \gb_i \right) \cap R_i =  \bigcup_{i \in A} \left(\delta_i + B + \gb_i\right)  \cap R_i= \emptyset,
				\end{equation} so $B$ is $R'$-admissible. Conversely, if $B$ is $R'-$admissible, then, for any $A$, there is some $\delta_A$ such that $(\delta_A + B) \cap R'_A = \emptyset$, and so by the first equality in \Cref{eq: equality with deltas}, we have $\left(\delta_A + B \right) \cap R_i = \emptyset$ for every $i \in A$.
				
				We now have to show that $\nu_R = \nu_{R'}$. To do this, we consider $$G_R = \prod_i \co_K/\gb_i \hspace{ 10 pt} \text{ and } \hspace{10 pt} G_{R'} = \prod_{A \in \mathcal{P}} \co_K/\gc(A),$$ and consider the map $V: G_{R'}  \rightarrow \gr $ that is the product of the maps $V_A: \co_K/\gc(A)\rightarrow \prod_{i\in A} \co_K/\gb_i$ given by $V_A(x + \gc_A) = (x+ \gb_i)_{i \in A}$. We have that $\varphi_R \circ V  = \varphi_{R'}$, since for any $a \in \co_K$, $$a \in \varphi_{R'}(g) \Leftrightarrow \mathlarger{\forall}_{A \in \mathcal{P}} \hspace{5pt} a+g_A \not \in R'_A 
				\Leftrightarrow \mathlarger{\forall}_{A \in \mathcal{P}} \hspace{5pt} \mathlarger{\forall}_{i \in A} \hspace{5pt} a+g_A \not  \in R_i 
				\Leftrightarrow \mathlarger{\forall}_{i} \hspace{5pt} a+V(g)_i \not  \in R_i
				\Leftrightarrow a \in \varphi_R(V(g)),
				$$ where we are using the fact that $g_A + \gb_i = V(g)_i$ when $i \in A$. 
				
				Denoting the Haar measure in $G_R$ and $G_{R'}$ by $\mmp_R$ and $\mmp_{R'}$ respectively, it remains to show that $\mmp_R(U) = \mmp_{R'}(V^{-1}(U))$, since then it follows that for any measurable subset of $\Omega_{R'} = \Omega_{R}$, $$\nu_{R'}(U) =\mmp_{R'}(\varphi_{R'}^{-1}(U)) = \mmp_{R'}(V^{-1}(\varphi_{R}^{-1}(U))) = \mmp_R(\varphi_{R}^{-1}(U)) = \nu_R(U). $$
				
				In order to show that $\mmp_{R'}(U) = \mmp_R(V^{-1}(U))$, it is enough to prove that this holds for all cylinder sets $$C(x_1,\dots,x_k) := \{g \in G_R:g_i \equiv x_i \mod \gb_i\}$$ for every $k \geq 1$. Fix any $k$, and sequence $x_1,\dots,x_k$. We have that $\mmp_R(C(x_1,\dots,x_k)) = \prod_{i=1}^k N(\gb_i)^{-1},$ so now we have to show that this is also the value of $\mmp_{R'}(V^{-1}(C(x_1,\dots,x_k))).$ Let $A_1,\dots, A_l$ be elements of $\cp$ that cover the set $\{1,2,\dots,k\}$. By independence, we have that $$\mmp_{R'}(V^{-1}(C(x_1,\dots,x_k))) = \prod_{i=1}^l \mmp_{R'}(V_{A_i}^{-1}(C^{A_i}(x_1,\dots,x_k))),$$ where $$C^{A_i}(x_1,\dots,x_k) = \{g \in \prod_{j \in A_i} \co_K/\gb_j: g_j \equiv x_j \mod \gb_j \text{ for every } j \in A_i \cap [1,\dots,k]\}.$$
				Therefore, the result will follow if we can show that for any $i$ $$\mmp_{R'}(V_{A_i}^{-1}(C^{A_i}(x_1,\dots,x_k))) = \prod_{j \in A_i\cap [1,k]} N(\gb_j)^{-1}.$$ The ideal $\gc(A_i) = \text{lcm}(\{\ga_{A_i}\} \cup \{\gb_j\}_{j \in A_i} )$ can be written as the product of coprime ideals $\prod_{j \in A_i} \gb_i$ and some $\ga'_{A_i}$, which is uniquely defined since $\co_K$ has unique factorization of ideals into prime ideals. Hence  we have the commutative diagram
				\[\begin{tikzcd}
					{\co_K/\gc(A_i)} & {\co_K/\ga'(A_i)\times\co_K/\prod_{j\in A_i}\gb_j} & {\co_K/\prod_{j\in A_i}\gb_j} & {\prod_{j\in A_i}\co_K/\gb_j}
					\arrow["{\phi_1}", from=1-1, to=1-2]
					\arrow["V_{A_i}"', bend right=15, from=1-1, to=1-4]
					\arrow["\pi", from=1-2, to=1-3]
					\arrow["{\phi_2}", from=1-3, to=1-4]
				\end{tikzcd}\] 
				where $\phi_1, \phi_2$ are isomorphisms obtained by the Chinese Remainder Theorem, and $\pi$ is the projection on the second coordinate. From this, and using the fact that $\phi_1,\phi_2$ are bijections, we see that $$|V_{A_i}^{-1}(C^{A_i}(x_1,\dots,x_k))| = |\pi^{-1}(\phi_2^{-1}(C^{A_i}(x_1,\dots,x_k)))| = N(\ga'(A_i))|C^{A_i}(x_1,\dots,x_k)|.  $$  It is clear that $|C^{A_i}(x_1,\dots,x_k)| = \prod_{j \in A_i \cap [k+1,\infty[} N(\gb_i)$, and so $$\mmp_\cc(V_{A_i}^{-1}(C^{A_i}(x_1,\dots,x_k))) = \frac{|V_{A_i}^{-1}(C^{A_i}(x_1,\dots,x_k))|}{N(\gc(A_i))} = \frac{N(\ga'(A_i))\prod_{j \in A_i \cap [k+1,\infty[}N(\gb_j)}{N(\ga'(A_i)) \prod_{j \in A_i} N(\gb_j)} = \prod_{j \in A_i\cap [1,k]} \frac{1}{N(\gb_j)} $$ as we wanted to show. This implies that $\nu_R = \nu'_{R'}$, which concludes the proof of the lemma.
			\end{proof}
			
			We now show that if $\Omega_R = \Omega_{R'}$, for minimal sieves $R$ and $R'$, then they must be supported on the same set. This will allow us to show that if $\Omega_R = \Omega_{R'}$ then $\nu_R = \nu_{R'}$.
			
			\begin{lemma}
				\label{lm: OmegaR =omegar' implies equality of base}
				Let $R$ and $R'$ be minimal Erdős sieves. If $\Omega_{R} = \Omega_{R'}$, then $\cb_R = \cb_{R'}$. 
			\end{lemma}
			
			\begin{proof}
				Take any $\gb \in \cb_R$. We claim that there must be some $\gb' \in \cb_{R'}$ such that $(\gb,\gb') \neq 1$. To show this we assume that  $(\gb,\gb') = 1$ for every $\gb' \in \cb_{R'}$, and show that this implies that $\Omega_{R} \neq \Omega_{R'}$. In order to do so, we proceed similarly to how we did in the proof of  \Cref{lm: there are admissible representatives of R_gb^c  } to find some $A \in \Omega_{R'}$ that is not in $\Omega_R$.
				
				We consider $$\Delta := \{\gb'\in \cb_{R'}: \frac{|R'_{\gb'}|}{N(\gb')} \geq \frac{1}{N(\gb)}\}.$$ Since $R'$ is Erdős, $\Delta$ is finite. Using the Chinese Remainder Theorem, we can find a finite set $A$ such that $A+ \gb = \co_K$, and $A \subset (R'_{\gb'})^c$ for every $\gb' \in \Delta$. This means that $A \not \in \Omega_R$, but we now show that it is in $\Omega_{R'}$. For ideals $\gb\in \Delta$, we  know by definition of $A$ that $A \subset (R'_{\gb'})^c$. For ideals $\gb \not \in \Delta$ we procced as in in  \Cref{lm: there are admissible representatives of R_gb^c  }. For every such ideal we have that $|A|\frac{|R'_{\gb'}|}{N(\gb')}  < 1$, so $A+ R'_{\gb'}$ cannot cover $\co_K$, and so  $A \in \Omega_{R'}$. We obtain the desired contradiction, so for any $\gb \in \cbr$, there must be some $\gb' \in \cb_{R'}$ such that $(\gb,\gb') \neq 1$.
				
				We now write $V_{\gb} := \{\gb' \in \cb_{R'}: (\gb, \gb') \neq 1 \},$ which must be a non-empty set. We will show that it equals $\{\gb\}$. The first step is to show that there exists some $x \in \co_K$ such that \begin{equation}
					\label{eq: Rgb subset of union}
					R_\gb \subset x+ \bigcup_{\gb' \in V_\gb} R'_{\gb'}.
				\end{equation}
				Writing $\gc = \text{lcm}(\{\gb\}\cup V_\gb)$, this is equivalent to showing that $$R_\gb + \gc \subset x+ \bigcup_{\gb' \in V_\gb} R'_{\gb'} + \gc.$$
				
				Consider the sieve $W$ supported on $\cb_W = (\cb_{R'}\cup\{\gc\})\setminus V_\gb$ and defined by $W_\gc = \bigcup_{\gb' \in V_\gb} R'_{\gb'} + \gc$ and $W_{\gb'} = R'_{\gb'}$ for any other $\gb' \in \cb_W$. This is a dilation of $R'$, therefore \Cref{lm: isomorphism of dynamical systems for contraction dilation} implies that $\Omega_W = \Omega_{R'} = \Omega_{R}$. By \Cref{lm: there are admissible representatives of R_gb^c  }, we can find some finite $A \in \Omega_W$ such that $A+\gc = (W_\gc)^c$. Assume that for all $x\in \co_K$, we have $$R_\gb +\gc \not \subset x+ \bigcup_{\gb' \in V_\gb} R'_{\gb'} + \gc =  x + W_\gc.$$ Then we get that $(R_\gb + \gc) \cap x + (W_\gc)^c \neq \emptyset$ for all $x\in \co_K$. This implies that $$\emptyset \neq R_\gb \cap (x+A+\gc) \subset R_\gb \cap (x+A+\gb) $$ for all $x\in \co_K$, which shows that $A$ is not in $\Omega_{R}$, contradicting the fact that $\Omega_W = \Omega_{R}$. Consequently, there must be some $x\in\co_K$ such that \Cref{eq: Rgb subset of union} holds, as we wanted to show.
				
				'Visually', the proof now looks as follows. As in Section 6, we consider a bipartite graph with edges labeled $R_\gb$ or $R'_{\gb'}$, with an edge between $R_\gb$ and  $R'_{\gb'}$ if $(\gb,\gb') \neq 1$. We first claim that, writing $V_\gb$ as $\{\gb'_1,\dots,\gb'_r\}$ and taking some $\ga \in \cb_R$ different from $\gb$, this graph cannot have a subgraph that looks as follows.
				
				% https://q.uiver.app/#q=WzAsNixbMCwwLCJSX1xcZ2IiXSxbMiwwLCJSJ197XFxnYidfMX0iXSxbMiwxLCJSJ197XFxnYidfMn0iXSxbMiwzLCJSJ197XFxnYidfcn0iXSxbMiwyLCJcXHZkb3RzIl0sWzAsMiwiUl9cXGdhIl0sWzAsMSwiIiwwLHsic3R5bGUiOnsiaGVhZCI6eyJuYW1lIjoibm9uZSJ9fX1dLFsyLDAsIiIsMCx7InN0eWxlIjp7ImhlYWQiOnsibmFtZSI6Im5vbmUifX19XSxbMywwLCIiLDAseyJzdHlsZSI6eyJoZWFkIjp7Im5hbWUiOiJub25lIn19fV0sWzEsNSwiIiwwLHsic3R5bGUiOnsiaGVhZCI6eyJuYW1lIjoibm9uZSJ9fX1dXQ==
				
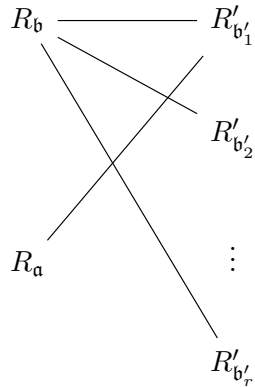
\[\begin{tikzcd}
					{R_\gb} && {R'_{\gb'_1}} \\
					&& {R'_{\gb'_2}} \\
					{R_\ga} && \vdots \\
					&& {R'_{\gb'_r}}
					\arrow[no head, from=1-1, to=1-3]
					\arrow[no head, from=1-3, to=3-1]
					\arrow[no head, from=2-3, to=1-1]
					\arrow[no head, from=4-3, to=1-1]
				\end{tikzcd}\]
				
				This means that this graph can be written as the union of disjoint graphs of the form 
				
				% https://q.uiver.app/#q=WzAsNSxbMCwwLCJSX1xcZ2IiXSxbMiwwLCJSJ197XFxnYidfMX0iXSxbMiwxLCJSJ197XFxnYidfMn0iXSxbMiwzLCJSJ197XFxnYidfcn0iXSxbMiwyLCJcXHZkb3RzIl0sWzAsMSwiIiwwLHsic3R5bGUiOnsiaGVhZCI6eyJuYW1lIjoibm9uZSJ9fX1dLFsyLDAsIiIsMCx7InN0eWxlIjp7ImhlYWQiOnsibmFtZSI6Im5vbmUifX19XSxbMywwLCIiLDAseyJzdHlsZSI6eyJoZWFkIjp7Im5hbWUiOiJub25lIn19fV1d
				
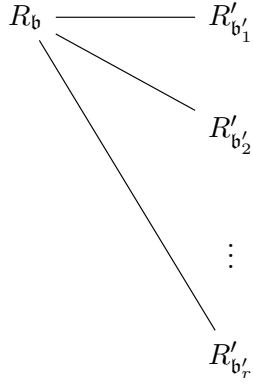
\[\begin{tikzcd}
					{R_\gb} && {R'_{\gb'_1}} \\
					&& {R'_{\gb'_2}} \\
					&& \vdots \\
					&& {R'_{\gb'_r}}
					\arrow[no head, from=1-1, to=1-3]
					\arrow[no head, from=2-3, to=1-1]
					\arrow[no head, from=4-3, to=1-1]
				\end{tikzcd}\] or the equivalent mirrored graphs (meaning that for some $R'_{\gb'}$, it will be connected to some $R_{\gb_1}, \dots,R_{\gb_l}$), and then we will use the minimality of $R$ and $R'$ to show that we must have $r = 1$.
				
				More formally, take $y \in R_\gb$. By \Cref{eq: Rgb subset of union}, there is some $x$ (not depending on $y$) such that $y-x+\gb \subset \bigcup_{\gb' \in V_\gb} R'_{\gb'}$. By \Cref{lm: if x gb in union must belong to one}, it follows that $y-x + \gb \subset R'_{\gb'_i}$ for some $i$. Consequently, we must have that $y-x+(\gb,\gb'_i) \subset R'_{\gb'_i}$. We now claim that there must be some $x_i$ such that $y-x_i+(\gb,\gb'_i) \subset R_\gb$ if $y-x \subset R'_{\gb'_i}$.
				
				Indeed, assume that $y-x-t+(\gb,\gb'_i) \not \subset R_\gb$ for all $t\in \co_K$. We know that there is some $t'_i\in \co_K$ such that $$y-x+(\gb,\gb'_i) \subset R'_{\gb'_i} \subset t'_i+\bigcup_{\ga \in \cb_R: (\ga,\gb'_i) \neq 1 }R_\ga,$$ using \Cref{eq: Rgb subset of union} applied to $R'_{\gb'_i}$. If $y-x-t'_i+(\gb,\gb'_i) \not \subset R_\gb$, then there must be some $z \in (\gb,\gb'_i)$ such that $y-x-t'_i+z+\gb \not \subset R_\gb$. Consequently, we would have, using \Cref{lm: if x gb in union must belong to one}, that  $y-x-t'_i+\gb \subset R_\ga$ for some $\ga \in \cb_R$, distinct from $\gb$. But this is impossible, since this would imply that $\co_K = \ga+\gb \subset R_\ga$. It follows that if $y-x \subset R'_{\gb'_i}$, then taking $x_i = x+t'_i$, we must have $y-x_i + (\gb,\gb'_i) \subset R_\gb$.
				
				Choosing representatives $y_1,\dots,y_{|R_\gb|}$ for $R_\gb$ in $\co_K$, let $A_i$ be the set of $y_j$ such that $y_j + \gb \subset R'_{\gb'_i}$. Using the Chinese Remainder Theorem, we can find some $x$ such that $x \equiv x_i \mod (\gb,\gb'_i)$ for all $i$ (with the $x_i$ whose existence we showed in the last paragraph). We get that $$R_\gb \subset \bigcup_{i=1}^r A_i + (\gb,\gb_i') = \bigcup_{i=1}^r A_i +x-x_i + (\gb,\gb_i') = x+\bigcup_{i=1}^r A_i -x_i+ (\gb,\gb_i')  \subset x+ R_\gb.$$ As finite subsets of $\co_K/\gb$, both $R_\gb$ and $x+R_\gb$ have the same cardinality, so $R_\gb \subset x+R_\gb$ implies equality. Consequently, we get that
				$$\bigcup_{i=1}^r A_i + (\gb,\gb_i') = R_\gb.$$ 
				
				By minimality of $R$, we conclude that there must be some $i$ such that $ (\gb,\gb_i') = \gb$, which by coprimality of the $\gb'_i$, shows that there must be a unique $\gb'$ such that $V_\gb = \{\gb'\}$ and $\gb \mid \gb'$. Using the minimality of $R'$, we can use the same argument to show that the set of those $\ga \in \cb_R$ such that $(\ga,\gb') \neq 1$ must also have only one element which is divisible by $\gb'$. This unique element must be $\gb$, and since $\gb \mid \gb'$ and $\gb' \mid \gb$, we must have an equality $\gb = \gb'$, which shows that $\cb_R = \cb_{R'}$ as we wanted to show.
			\end{proof}
			
			Together, Lemmas \ref{lm: OmegaR = OmegaS same base} and \ref{lm: OmegaR =omegar' implies equality of base} give the following theorem.
			
			\begin{theorem}
				\label{thm: Equaivalence for omegaR = omegaRp}
				Let $R$ and $R'$ be minimal Erdős sieves. Then $\Omega_R = \Omega_{R'}$ if and only if $\cb_R = \cb_{R'}$, and for every $\gb \in \cb_R$, there is some $\delta_\gb \in \co_K$ such that $R_\gb = \delta_\gb + R'_\gb$.
			\end{theorem}
			
			Since every sieve is equivalent to some minimal sieve, this means that for every $R$ there is some minimal sieve $R'$ such that their associated dynamical systems are isomorphic. We get the following result.
			
			\begin{theorem}
				\label{thm: Omegar = Omegar' implies equal measures}
				Let $R$ and $R'$ be Erdős sieves. If $\Omega_R = \Omega_{R'}$, then $\nu_R = \nu_{R'}$.
			\end{theorem}
			
			\begin{proof}
				We can assume without loss of generality that $R$ and $R'$ are minimal, since by  \Cref{lm: every sieve is equivalent to minimal sieve} these sieves can be contracted until they are minimal, and by  \Cref{lm: isomorphism of dynamical systems for contraction dilation} both $\Omega_R$ and $\nu_R$ are preserved by contractions.
				
				Hence, we can apply  \Cref{lm: OmegaR =omegar' implies equality of base} to conclude that $\cb_R= \cb_{R'}$. Since $\Omega_{R} = \Omega_{R'}$, the result now follows from  \Cref{lm: OmegaR = OmegaR' implies same meausre same base}.			
			\end{proof}

			\subsection{Isomorphisms of Dynamical Systems}
			
			We now show that the system $(\Omega_{R}, S,\nu_R)$ is isomorphic to a rotation on a compact group. This generalizes what had been previously done (see for example \cite{Abdalaoui},  \cite{Bartnicka}, \cite{celarosi} or \cite{Dymek}), but the use of sieves significantly simplifies the proof of this result. 
			
			A sketch of the proof when $R$ is a $\cb-$free system goes as follows. Given some $g \in G_R$, let $R(g)$ be the sieve defined by \begin{equation}
				\label{eq:R(g)}
				R(g)_i = -g_i+R_i.
			\end{equation} Note that $\varphi_R(g) = \cf_{R(g)}$. Let $G'_R$ be the set of those $g \in G_R$ such that $R(g)$ has strong light tails for $B_N$. By showing that if $R$ is minimal, then $R(g)$ is minimal, \Cref{thm: equivalence minimal sieves main} shows that $\varphi_R$ restricted to $G'_R$ will be a bijection, and so the result will follow if we show that
			\begin{equation}
				\label{eq: nuR of sieves with strong light tails is 1}
				\nu_R(\{\cf_{R(g)}\in \Omega_{R}: R(g) \text{ has strong light tails for }B_N \}) =1.
			\end{equation} 
			
			The reason why this sketch does not work for general sieves, is that it does not deal with two technical hurdles, with which we will now address. First, the sieve $R$ may not be minimal. But this is easily dealt with, since by \Cref{lm: every sieve is equivalent to minimal sieve} there is a minimal sieve $R'$ that can be obtained from $R$ by successive contractions, and by \Cref{lm: isomorphism of dynamical systems for contraction dilation} the systems $(\Omega_R,S,\nu_R)$ and $(\Omega_{R'},S,\nu_{R'})$ will be isomorphic. Therefore, we are free to assume that $R$ is minimal.
			
			The second hurdle comes from assuming that the map sending $g$ to $R(g)$ is a bijection. This is clear when $R$ is a $\cb-$free system, but for general sieves it may be the case that we have $x+R_i = R_i$ without $x \in \gb_i$.
			
			\begin{example}
				Take a sieve such that $R_1 = 2\z = \{0,2\}+ 4\z$. Then, $2+R_1 = R_1$, although $2 \not \in 4\z$. Taking some $A \in Y_R$ such that $A +4\z = \{0,2\}+ 4\z$, we have that $1+A + 4\z=3+A+ 4\z$, despite both having empty intersection with $R_1$.
			\end{example} 
			
			Note that in this example the sieve $R$ is not minimal. Indeed, in the special case of sieves over $\q$,  if $R$ is a minimal, then $x+R_i = R_i$ implies that $x \in \gb_i$.  This is because if $x+ R_i = R_i$, then $x\z + R_i = R_i$, which means that $R_i = \bigcup_{r \in R_i} r+ (x,b_i)\z.$ Since $R$ is minimal, this requires that $(x,b_i) = b_i$, and so $x \in b_i\z$. 
			
			Yet, if $R$ is not a sieve over $\q$, this is no longer the case. This is because $x\z$ stops being an ideal of $\co_K$. We provide an example.
			
			\begin{example}
				Let $R$ be a sieve over $\q[i]$ for which $R_1 = \{i,1+i\}+2\z[i]$. We have that $$\co_{\q[i]} = \{0,1,i,1+i\}+2\z[i],$$ and $2\z[i] = (1+i)^2\z[i]$ is the square of a prime. The set $R_1$ contains elements from both congruence classes modulo $ (1+i)\z[i]$ so it is minimal. Yet, we have that $1+ R_1 =  R_1$, in spite of $1 \not \in 2\z[i]$.
			\end{example} 
			
			To solve this, the key insight is to notice that those $x$ such that $x+R_i = R_i$ form a (additive) subgroup of $\co_K$ that contains $\gb_i$. Indeed, if $x+R_i = R_i$, then $-x+R_i = R_i$, and it is clear that if $x,y$ belong to this subgroup, so does $x+y$. Defining for each $R_i$ the set  $$F(R_i) = \{x\in \co_K: x+ R_i = R_i\}$$ of those elements of $\co_K$ that fix $R_i$, we now define the group $\grf$ by \begin{equation}
				\grf := \prod_i \co_K/F(R_i).
			\end{equation} 
			
			If, given $g,g' \in G_R$, we have $R(g) = R'(g)$, we get that for all $i\in \mathbb{N},$ $-g_i + R_i = -g'_i + R_i$, which implies that $(g'_i-g_i) \in F(R_i)$, and so $g$ and $g'$ must be mapped into the same element of $\grf$ under the map that sends $(g_i)_{i \in \mathbb{N}} \in G_R$ to $(g_i + F(R_i))_{i \in \mathbb{N}} \in \grf$.
			
			Writing 
			\begin{equation}
				\label{eq: definition of G'RF}
				G'_{R,F} := \{g \in \grf: R(g) \text{ has strong light tails for }B_N\}
			\end{equation}
			we now have that by \Cref{thm: equivalence minimal sieves main} the map from $G'_{R,F}$ to $\Omega_{R}$ that sends $g$ to $\cf_{R(g)}$ is a bijection. We are now missing two things, in order to show that  $(\Omega_{R}, S,\nu_R)$ is isomorphic to $(\grf,T^F,\mmp^F)$, where $T^F$ is the action of $\co_K$ in $\grf$ defined by $T^F_a(g)_i = g_i +a$, and $\mmp^F$ is the Haar measure in $\grf$.
			
			The first is that the map $\vrf: \grf \rightarrow \Omega_{R}$, that sends $g$ to $\cf_{R(g)}$ is a factor map. The second is \Cref{eq: nuR of sieves with strong light tails is 1}. We start with the first. An equivalent way of defining $\vrf$ is though the equivalence $$a \in \vrf(g) \Leftrightarrow \mathlarger{\forall_{i}} \hspace{5pt} (a+g_i) \cap R_i = \emptyset. $$
			
			In Lemma 3.12 of \cite{part1}, we showed that $(\Omega_{R}, S, \nu_R)$ is a factor of $(G_R,T,\mmp)$. We now show that the same holds true for the system $(\grf, T, \mmp^F)$. 
			
			\begin{lemma}
				\label{lm: vrf factor map}
				The map $\vrf$ is a factor map from $(\grf, T^F, \mmp^F)$ to $(\Omega_{R}, S, \nu_R)$.
			\end{lemma}
			
			\begin{proof}
				
				We have to show that for any measurable $U \subset \grf$, $\nu_R(U) = \mmp^F(\vrf^{-1}(U))$ and that $$S \circ \vrf = \vrf \circ T^F.$$ Then, the result will automatically follow, since the image of $\vrf$ in $\Omega_R$ will have measure $1$.
				
				We start by defining $\phi_i:\co_K/\gb_i \rightarrow \co_K/F(R_i)$ to be the projection $\phi_i(x) = x+F(R_i)$ which is well defined since $\gb_i \subset F(R_i)$.  By taking the product of all the $\phi_i$, we obtain a map $\phi:G_R \rightarrow \grf $. We have that $\varphi_R = \vrf \circ  \phi$,  given that $a \in \varphi_R(g)$ is equivalent to $ \forall_i \hspace{5pt} (a+ g_i) \cap R_i = \emptyset$,  which is equivalent to $ \forall_i \hspace{5pt} (a+F(R_i) + g_i) \cap R_i = \emptyset$ by definition of $F(R_i)$. We get the following commutative diagram. 
				\[\begin{tikzcd}
					G_R && \grf \\
					\\
					& {\Omega_R}
					\arrow["\phi", from=1-1, to=1-3]
					\arrow["{\varphi_R}"', from=1-1, to=3-2]
					\arrow["\vrf", from=1-3, to=3-2]
				\end{tikzcd}\]
				
				Let $U$ be a cylinder set (which form a  base for the topology of $\grf$), that is, a set so that there are finite sets $S \subset \mathbb{N}$ and $U_i \subset \co_K/F(R_i)$ for $i \in S$ such that $$U = \{g \in \grf: g_i \in U_i \text{ for } i \in S\}.$$  By Lagrange's Theorem we have that $|\phi_i^{-1}(U_i)| = |U_i||F(R_i)|$ and $|\co_K/F(R_i)| = N(\gb_i)/|F(R_i)|$, therefore $$\mmp(\phi^{-1}(U)) = \prod_{i \in S} \frac{|U_i||F(R_i)|}{N(\gb_i)} = \prod_{i \in S} \frac{|U_i|}{|\co_K/F(R_i)|} = \mmp^F(U).  $$
				As desired, we obtain $$\nu_R(U) = \mmp(\varphi_{R}^{-1}(U)) = \mmp(\phi^{-1}(\vrf^{-1}(U))) = \mmp^F(\vrf^{-1}(U)).$$ 
				
				It remains to prove that $S \circ \vrf = \vrf \circ T^F$. Since $\phi(T_a(g)) = T^F_a(\phi(g))$, and $\phi$ is surjective, we get that for any $a \in \co_K$, $g \in \grf$, there is some $g' \in G_R$ such that $\phi(g') = g$. Then we have that $S_a(\vrf(g))$ is equal to (using that $S_a(\varphi_R(g')) = \varphi_{R}(T_a(g'))$ as shown in Lemma 3.12 of \cite{part1}), $$ S_a(\vrf(\phi(g'))) = S_a(\varphi_R(g')) = \varphi_{R}(T_a(g')) = \vrf(\phi(T_a(g'))) = \vrf(T_a^F(\phi(g'))) = \vrf(T_a^F(g)). $$ It follows that $S \circ \vrf = \vrf \circ T^F$ which concludes our proof.
			\end{proof}
			
			\begin{remark}
				\label{rmk: grf minimal rotation}
				The map $\phi$ used in the proof of  \Cref{lm: vrf factor map} is surjective and continuous, since the pre-image of cylinder sets in $\grf$ will be cylinder sets in $G_R$. We have that $G_R= \overline{\{T_a(\textbf{0}):a \in \co_K\}}$. Write $\textbf{0}^F$ for the identity element of $\grf$. Using continuity and surjectivity of $\phi$, we get $$\overline{\{T^F_a(\textbf{0}^F):a \in \co_K\}} = \overline{\{\phi(T_a(\textbf{0})):a \in \co_K\}} \supset \phi\left(\overline{\{T_a(\textbf{0}):a \in \co_K\}}\right) = \phi(G_R) = \grf. $$ Consequently, we have that $\grf = \overline{\{T^F_a(\textbf{0}^F):a \in \co_K\}}$, that is, $(\grf, T^F)$ is a minimal rotation of a compact group.
			\end{remark}
			
			It remains to show \Cref{eq: nuR of sieves with strong light tails is 1}. Given a sieve $R$ we define \begin{equation}
				\label{eq: CS_R}
				\cs_R := \{R': \Omega_R = \Omega_{R'} \text{ and } \cb_R = \cb_{R'}\}
			\end{equation} to be the set of all sieves supported on the same set as $R$ and with the same admissible sets.  By  \Cref{lm: OmegaR = OmegaS same base}, there is a bijection $\Phi:\grf \rightarrow \cs_R$ that is the map that sends $g \in \grf$ to the sieve $\Phi(g)$ defined by $$\Phi(g)_i = -g_i + R_i.$$ 
			
			Let $\Psi: \cs_R \rightarrow \Omega_{R}$ be the map that sends $R'$ to $\frp$. We have that $\vrf = \Psi \circ \Phi$, since $$a \in \vrf(g) \Leftrightarrow \mathlarger \forall_i \hspace{2pt} a+g_i \not \in R_i \Leftrightarrow \mathlarger \forall_i \hspace{2pt} a \not \in \Phi(g)_i \Leftrightarrow  a\in \cf_{\Phi(g)},$$ that is, we have the following commutative diagram.
			
			\[\begin{tikzcd}
				\grf && \cs_R \\
				\\
				& {\Omega_R}
				\arrow["\Phi", from=1-1, to=1-3]
				\arrow["\vrf"', from=1-1, to=3-2]
				\arrow["\Psi", from=1-3, to=3-2]
			\end{tikzcd}\]
			
			We now define \begin{equation}
				\label{def: sigmaR}
				\sigma_R = \Phi_*\mmp^F
			\end{equation}  to be the pushforward of $\mmp^F$ in $\cs_R$.  \Cref{thm: ergodic theorem} (the Ergodic Theorem) together with  \Cref{thm:Fr is generic} give us the following result.
			
			\begin{lemma}
				\label{lm: sieves with weak light tails have measure 1}
				Let $R$ be an Erdős sieve and $I_N$ a tempered Følner sequence. We have $$\sigma_R(\{R'\in\cs_R: R' \text{ has weak light tails with respect to } I_N\}) = 1.$$ 
			\end{lemma}
			
			\begin{proof}
				Since $\vrf = \Psi \circ \Phi$, we have that $$\Psi_*\sigma_R = (\Psi \circ\Phi)_* \mmp^F = (\vrf)_*\mmp^F = \nu_R,$$ where the last equality was shown in  \Cref{lm: vrf factor map}. It follows that $$\nu_R(\Psi(\cs_R)) = \sigma_R(\Psi^{-1}(\Psi(\cs_R))) = \sigma_R(\cs_R) = 1.$$ 
				
				By  \Cref{thm:Fr is generic}, we have that $$\{R'\in\cs_R: R' \text{ has weak light tails with respect to } I_N\} = \Psi^{-1}(\Psi(\cs_R) \cap \Gen(\nu_R,I_N)),$$ so the result follows from showing that $\nu_R(\Psi(\cs_R) \cap \Gen(\nu_R,I_N)) = 1.$ But since $I_N$ is tempered and $\nu_R$ is ergodic,   \Cref{thm: ergodic theorem} implies that $\nu_R(\Gen(\nu_R,I_N)) = 1$, which concludes the proof.
			\end{proof}
			
			The following theorem together with the fact that $\Psi_*\sigma_R = \nu_R$ implies that \Cref{eq: nuR of sieves with strong light tails is 1} holds.
			
			\begin{theorem}
				\label{thm: R is OmegaR=OmegaRp to some sieve with strong light tails}
				Let $R$ be an Erdős sieve, and $I_N$ a tempered Følner sequence. We have $$\sigma_R(\{R'\in\cs_R: R' \text{ has strong light tails with respect to } I_N\}) = 1.$$
			\end{theorem}
			
			\begin{proof}
				Let $$H_{R} = \bigoplus_i \co_K/F(R_i).$$ We can write elements $h \in H_R$ as sequences $h = (h_1,h_2,\dots)$ such that $h_i = 0$ except for a finite number of indices. Let $V$ be the action of $H_R$ in $\cs_R$ given by $V_h(R')_i = R'_i + h_i$.  \Cref{thm: R has strong light tails if finite translations} can be restated as saying that $$\{R'\in\cs_R: R' \text{ has strong light tails for } I_N\} = \bigcap_{h \in H_R} V_h(\{R'\in\cs_R: R' \text{ has weak light tails for } I_N\}).$$ Clearly $\sigma_R$ is $V$ invariant, so, by  \Cref{lm: sieves with weak light tails have measure 1}, the right hand side is a countable intersection of sets of measure $1$. Therefore, the left hand side also has measure $1$, as we wanted to show.
			\end{proof}
			\begin{remark}
				In particular, for every sieve $R$ there is some sieve $R'$ with strong light tails with respect to $B_N$ such that $\Omega_R = \Omega_{R'}$.
			\end{remark}
			
			With this, it is now easy to show the desired isomorphism.

			\begin{theorem}
				\label{thm: main theorem}
				Let $R$ be an Erdős sieve.  Then $(\Omega_R,S,\nu_R)$ is isomorphic to $(\grf,T^F, \mmp^F)$.
			\end{theorem}
			
			\begin{proof}
				Let $G'_{R,F}$ be the set defined in \Cref{eq: definition of G'RF} and $$L_R := \{R'\in\cs_R: R' \text{ has strong light tails with respect to } B_N\}.$$ We have to show that the map $\vrf$ that sends $g \in \grf$ to $\cf_{R(g)}$ is injective when restricted $G'_{R,F}$, and that $\mmp_F(G'_{R,F}) = 1$. Since $\sigma_R = \Phi_*\mmp_F,$ and $G'_{R,F} = \Phi^{-1}(L_R),$ the fact that $\mmp_F(G'_{R,F}) = 1$ follows directly from \Cref{thm: R is OmegaR=OmegaRp to some sieve with strong light tails}, which states that $\sigma_R(L_R)=1$. 
				
				It remains to show that $\vrf$ when restricted to $G'_{R,F}$ is injective. Since $\vrf = \Psi \circ \Phi$, and $\Phi$ is a bijection, it remains to show that $\Psi$ restricted to $L_R$  is injective. As we have pointed out before, by \Cref{lm: isomorphism of dynamical systems for contraction dilation}, we are free to assume that $R$ is minimal. We now claim that this implies that every $R(g)$ is also minimal. Indeed, if $R(g)$ was not minimal, then there would be some $i$ and some finite set $A$ of proper divisors of $\gb_i$ such that $R(g)_i$ can be written as the union of congruence classes modulo the elements of $A$. But if $R(g)_i = \bigcup_{\gb \in A} E_i + \gb $, then $R_i = \bigcup_{\gb \in A} (g_i+E_i) + \gb $. Hence, we see that $R(g)$ is minimal if and only if $R$ also is. \Cref{thm: equivalence minimal sieves main} implies that $\Psi$ restricted to $L_R$  is injective, which completes the proof.
			\end{proof}
			
			In the particular case where $R$ is a sieve over $\q$, we know by  \Cref{lm: isomorphism of dynamical systems for contraction dilation} that by successively contracting $R$, we will obtain an equivalent minimal sieve $R'$ such that $(\Omega_R,S,\nu_R)$ is isomorphic to $(\Omega_{R'},S,\nu_{R'})$. In this case, we have seen that $F(R'_i) = \gb_i$ for every $i$. Therefore, we get the following corollary\footnote{We point out that  \Cref{thm: main theorem} was already known for sieves over $\q$, see Lemma 2.2.21 of \cite{Kulaga}.} of  \Cref{thm: main theorem}.

			\begin{corollary}
				\label{crl: Sieves over Q are all isomorphic to odometer}
				Let $R$ be an Erdős sieve over $\q$. Let $R'$ be the minimal sieve equivalent to $R$ obtained by successively contracting $R$, and let $\cb_{R'}$ be the set on which $R'$ is supported. Then $(\Omega_R,S,\nu_R)$ is isomorphic to the system $(G_{R'},T,\mmp)$ where $G_{R'}$ is the group $$G_{R'} = \prod_{b \in \cb_{R'}} \z/b\z.$$
			\end{corollary}

			This means that over $\z$, if $R$ is a minimal Erdős sieve, then the dynamical system associated to $\Omega_R$ does not depend on the congruence classes being sieved, only on the support $\cb_R$ of $R$. This is rather surprising since by  \Cref{lm: OmegaR = OmegaS same base} we would expect two random sieves supported on the same set to have very different sets of admissible sets. For example, let $R$ be the squarefree sieve $R_p = p^2\z$, and $R'$ the sieve defined by $R'_p = \{0,1\}+p^2\z$. Then, $\Omega_{R'} \subset \Omega_R$, but have $\nu_R(\Omega_{R'}) = 0$. Still, the associated measure theoretical dynamical systems will be isomorphic.

			Indeed, not even the support  $\cb_R$ characterizes $(\Omega_R,S,\nu_R)$, as we now show. Given a set of pairwise coprime ideals $\cb$, let $\cp(\cb)$ be the set
			\begin{equation}
				\label{eq: def of cp}
				\cp(\cb) := \{\gp^{\max_{\gb \in \cb}v_\gp(\gb)}: \gp \text{ prime ideal of } \co_K\},
			\end{equation} where $v_\gp$ is the $\gp-$adic valuation. For example, if $\cb = \{p_{2i}^2p_{2i+1}^2\z\}$, then $\cp(\cb)$ is the set of all the primes squared. Let $G(\cb) = \prod_{\gb \in \cb} \co_K/\gb$. We will now show that $(G(\cb),T)$ and $(G(\cp(\cb)),T)$ are always topologically conjugate. 
			
			By the Chinese Remainder Theorem, we have for any $\gb \in \cb$ isomorphisms $$V_\gb: \co_K/\gb  \rightarrow \prod_{\gp \mid \gb}\co_K/\gp^{v_\gp(\gb)}. $$
			The product of all these maps gives a map $V:G(\cb) \rightarrow G(\cp(\cb)) $, which, since every $V_\gb$ is a bijection and satisfies $V_\gb(x+y) = x + V_\gb(y)$ for every $x \in \co_K$, is our desired isomorphism. Since both systems $(G(\cb),T)$ and $(G(\cp(\cb)),T)$ are uniquely ergodic, this implies that  $(G(\cb),T,\mmp)$ and $(G(\cp(\cb)),T,\mmp)$ are also isomorphic.
			
			In the next subsection we will compute the spectrum of these dynamical systems, which will show that for minimal sieves over $\q$, the set $\cp(\cb_R)$ characterizes these dynamical systems. 
			
			\begin{remark}\index{Measure of Maximal Entropy}
				
				When $R$ is a sieve over $\q$,  Theorem 2.2.25 in \cite{Kulaga} shows that there is a unique invariant measure of $(\Omega_R,S)$ that has maximum entropy. For sieves over an étale $\q-$algebra $K$ of degree greater than $1$, it is currently not known whether this is the case.
			\end{remark}

			\subsection{Spectrum Computation} 
			
			We now compute the spectrum of  $(\Omega_R,S,\nu_R)$. This is relevant to us since by the Halmos-von Neumman Theorem (\Cref{thm: Halmos Von Neumann}), this is an invariant of measure theoretical dynamical systems. 
			
			\begin{theorem}
				\label{thm:Spectrum}
				Let $R$ be an Erdős sieve. Then, we have $$\sigma_p(\Omega_R,S, \nu_R) =  \{\chi \in \widehat{\co_K}: \text{ there exists some } C \subset \mathbb{N} \text{ finite such that } \hspace{2pt}   \chi\vert_{\bigcap_{j\in C}F(R_j)} = 1\} .$$ 
				
			\end{theorem}
			
			\begin{proof}
				By  \Cref{thm: main theorem}, our problem reduces to computing the spectrum of $(\grf,T^F,\mmp^F)$. For any finite $C \subset \mathbb{N}$, we have a surjection $$\co_K \mathlarger \twoheadrightarrow \prod_{j\in C} \co_K/F(R_j),$$ obtained from composing the surjections from $\co_K$ to $\prod_{j\in C}\co_K/\gb_j$ and from this to $\prod_{j\in C}\co_K/F(R_j)$. Hence, we get an isomorphism $$\co_K\bigcap_{j\in C} F(R_j)  \rightarrow  \prod_{j\in C}\co_K/F(R_j)$$ which takes $x$ and sends it to $(x_j+F(R_j))_{j\in S}$. Note that $\gb_j \subset F(R_j)$, so the product of the $\gb_j$ is contained in $\bigcap_{j\in C} F(R_j)$. Since these are finite abelian groups, we get an isomorphism between the character groups $$
				\widehat{\mathcal{O}_K / \bigcap_{j \in C} F(R_j)} \rightarrow \widehat{\prod_{j \in C} \mathcal{O}_K \left/ F(R_j) \right.} .
				$$
				
				Hence, given any character $\chi$ of $\co_K$ such that $ \chi\vert_{\bigcap_{j\in C}F(R_j)} = 1$ for some finite $C \subset \mathbb{N}$, there are characters $\chi_j$ of $\co_K$ such that $\chi_j\vert_{F(R_j)} = 1$ and $\chi = \prod_{j\in C} \chi_j$. Consequently, by considering the map $\zeta_\chi:\grf \rightarrow \mathbb{C}$ given by $\zeta_\chi(g) = \prod_{j\in C}\chi_j(g_j)$, we have $$\zeta_\chi(T_a(g)) = \prod_{j\in C}\chi_j(a+g_j) = \prod_{j\in C}\chi_j(a)\chi_j(g_j) = \chi(a)\zeta_\chi(g).$$ Therefore we have that $$ \{\chi \in \widehat{\co_K}: \text{ there exists some } C \subset \mathbb{N} \text{ finite such that } \hspace{2pt}   \chi\vert_{\bigcap_{j\in C}F(R_j)} = 1\} \subset \sigma_p(\grf,T, \mmp^F).$$
				
				Let $\chi' \in \sigma_p(\grf,T, \mmp^F).$ We want to show that there is some finite $S$ such that $\chi'\vert_{\bigcap_{j\in C}F(R_j)} = 1$. Note that the functions $\zeta_\chi$ with $\chi$ some character such that $\chi\vert_{\bigcap_{j\in C}F(R_j)} = 1$ correspond exactly to the characters of $\grf$, as we have $$\widehat{\grf} = \bigoplus_i \co_K/F(R_i).$$
				
				By Parseval's Theorem, these form an orthonormal basis of $L^2(\grf)$. Let $f\in L^2(\grf)$ be a non-zero eigenfunction of the Koopman representation with eigenvalue $\chi'$, that is,  we have $f(T_{a}(g)) = \chi'(a)f(g)$ for all $a\in \co_K$ and $g \in \grf$. Writing  $f = \sum_{\zeta_{\chi} \in \widehat{\grf}} c_{\chi} \zeta_{\chi}(g)$, we get $$\sum_{\zeta_{\chi} \in \widehat{\grf}} c_{\chi}(\chi(a)-\chi'(a)) \zeta_{\chi}(g) = 0.$$ By linear independence, this means that for every $a \in \co_K$, and for every $\chi$ such that $\zeta_{\chi} \in \widehat{\grf}$ and $c_\chi  \neq 0$, we have $\chi(a) = \chi'(a)$. Consequently, there must be a unique $\chi$ such that $\zeta_{\chi} \in \widehat{\grf}$ and $c_\chi  \neq 0$ for which we have $\chi = \chi'$. This concludes the proof.
			\end{proof}
			
			\begin{example}
				Let $R$ be a minimal Erdős sieve over $\q$, supported on the set $\cb_R= \{b_1,b_2,b_3,\dots\}$. We can identify a character $\chi$ such that $\chi \vert_{b_i\z} = 1$ with $\chi(1)$ which will be a root of unity of degree $b_i$. By  \Cref{thm:Spectrum}, we conclude that the spectrum of  $(\Omega_R,S, \nu_R)$ is the union of all roots of unity of degree $\prod_{j \in S}b_j$ over all finite  $S \subset \mathbb{N}$. As we pointed out, this shows that  $(\Omega_R,S, \nu_R)$ is not determined by $\cb_R$, but rather by $\cp(\cb_R)$ (as defined in \Cref{eq: def of cp}).
			\end{example}

			\section{$X_R$ and  $\Omega_{R}$}

			The objective of this section is to generalize point (3) of Sarnak's Program for sieves. As proven in \cite{Bartnicka}, if  $R$ is an Erdős $\cb-$free system, then $X_R = \Omega_R$. The following example shows that this will not hold for general sieves, even if they are Erdős with strong light tails for some Følner sequence.
			
			\begin{example}
				\label{ex: sieve such that X different Omega}
				Let $R$ be the sieve supported on the set $\cb_R = \{p^2\z:p \text{ prime}\}$ defined by $R_p = \{0,1\}+p^2\z$ for $p\geq 3$ and $R_2 = \emptyset$. The set $A = \{2,4\}$ is admissible, since $A\cap R_p = \emptyset$ for every $p \geq 3$. However, we see that if $x \not \in \fr$, then either one of $x-1$ or $x+1$ are not in $\fr$. It follows that $d(A,\fr+a) = 1$ for every $a$, so $A \not \in X_R$. 
			\end{example}
			
			This leads us to a number of distinct questions. First, given a sieve $R$ when does a sieve $R$ satisfy $X_R = \Omega_R$? We answer this question by characterizing when a set $A$ is in $X_R$, with two different conditions, shown in  \Cref{lm: Conditions for A in X_R} and  \Cref{thm: Finite Admissible sets have positive density}.

			Secondly, note how from the point of view of measure theoretical dynamics, it is not relevant that $X_R \neq \Omega_R$, as long as $\nu_R(X_R) =1 $. In \Cref{thm: nu_R(X_R) = 1} we show that this happens if and only if $R$ has weak light tails with respect to at least one Følner sequence.
			
			In the context of general $\cb-$free systems, there is a space $\widetilde{X}_R$ contained in $\Omega_R$ that is sometimes considered (see \cite{Dymek}), which corresponds to the smallest hereditary system that contains $X_R$ (see \Cref{def: hereditary}). In \Cref{thm: tilde X R = Omega R} we show that $\widetilde{X}_R$ must equal $\Omega_{R}$ whenever $R$ has weak light tails for some Følner sequence $I_N$.
			
			Finally, in the previous section we showed that if $R$ and $R'$ are minimal Erdős sieves, then $\Omega_R = \Omega_{R'}$ implies that $\cb_R = \cb_{R'}$, and for every $\gb \in \cb_R$, there is some $\delta_\gb \in \co_K$ such that $R_\gb = \delta_\gb + R'_\gb$. In  \Cref{thm: equivalence XR =XRp and OmegaR = OmegaRp}, we show that if $R$ and $R'$ have weak light tails for some (not necessarily common) Følner sequence, then $X_R = X_{R'}$ if and only if $\Omega_R = \Omega_{R'}$. We conclude by providing in \Cref{prop: lim lambda infty implies XR = OmegaR} an easy to check condition that is sufficient but not necessary in order for $X_R$ to equal $\Omega_R$.

			We start with the following lemma, which gives an answer to the first question.
			
			\begin{lemma}
				\label{lm: Conditions for A in X_R}
				Let $R$ be an Erdős sieve that has weak light tails for some Følner sequence $I_N$. An $R-$admissible set $A$ belongs to $X_R$ if and only if for every  of its finite subsets $A'$ and sequences $b_1, \dots, b_l $ of elements of $\co_K$ not in $A$, there are indexes $i_1, \dots, i_l$ such that $-b_j + R_{i_j} \not \subset -A'+ R_{i_j}$ and if $i_j = i$ for every $j$ in some finite set $Q$, then $\bigcap_{j \in Q} (-b_j + R_i) \setminus (-A'+R_i) \neq \emptyset$, .
			\end{lemma}
			
			\begin{proof}
				To show that $A \in X_R$, we have to show that for any finite set $M$, there is some $x \in \co_K$ such that $S_x(\fr) \cap M = A \cap M$. Fix $M$, and write $A' = A \cap M$, $B = M \setminus A'$. The equality $S_x(\fr) \cap M = A \cap M$ is equivalent to $A' \subset S_x(\fr)$ and $b \not \in S_x(\fr)$ for every $b \in B$. The first condition $A' \subset S_x(\fr)$ is equivalent to $x \in \frp$, where $R'$ is the sieve defined by the condition $R'_i = -A' + R_i$ for all $i$. This sieve can be written as the union of $|A'|$ sieves, and since $A'$ is admissible (given that it is a subset of $A$), $-A' + R_i$ is always distinct from $\co_K$. It follows that $R'$ is an Erdős sieve with weak light tails for some Følner sequence, and therefore satisfies the local global principle. 
				
				The second condition $b \not \in S_x(\fr)$, which can be written as $x+b \not \in \fr$, is equivalent to there being some $i$ such that $x + b  \in R_i$. By our hypothesis, there is for every $b \in B$, some index $i_b$ and $x_{i_b} \in \co_K$ such that $x_{i_b} \in -b + R_{i_b}$, but $x_{i_b} \not \in R'_{i_b}$. We now want to apply the local global principle of $R'$ for the congruence relations $x_{i_b} \mod \gb_{i_b}$. 
				If the $i_b$ are all unique, this is well defined, and we will  get some $x$ that satisfies both of our desired conditions, so we will get $S_x(\fr) \cap M = A \cap M$. 
				
				If the $i_b$ are not unique, then we proceed as follows. Order $B$ so that we can write $B = \{b_1,b_2,\dots,b_l\}$. We define an equivalence relation on the set of numbers from $1$ to $l$ such that $j \sim k$ if and only if $i_{b_j} = i_{b_k}$. For any congruence class $C$, we associate to it $i_C$, which equals $i_{b_j}$ for any $j \in C$. By hypothesis, there is some $x_C \in \co_K$ such that $x_C \in \bigcap_{j \in C} (-b_j + R_{i_C}) \setminus (R'_{i_C}) \neq \emptyset$. Now, using the local global principle for $R'$, we can find some $x \in \frp$ such that $x \equiv x_C \mod \gb_{i_C}$ for every $C$, which means that for any $j \in C$, $x+b_j \in R_{i_C}$. Again, these two conditions together imply that  $S_x(\fr) \cap M = A \cap M$.

				On the other hand, suppose that there is a sequence $b_1, \dots, b_l $ of elements not in $A$, and some finite subset $A'$ of $A$, such that there is no choice of $i_j$ for which $-b_j + R_{i_j} \not \subset -A' + R_{i_j}$ for every $j$, with $\bigcap_{j \in Q} (-b_j + R_i) \setminus (-A'+R_i) \neq \emptyset$, if $i_j = i$ for every $j$ in some finite set $Q$. We have two cases. First, it might be that there is some $j$ such that  $-b_j + R_{i} \subset -A' + R_{i}$ for every $i$. Then, there is no $x \in \frp$ (that is, for which $A' \subset S_x(\fr)$) such that $b_j \not \in S_x(\fr)$, as such an $x$ would have to be in $-b_j + R_i$ for some $i$, while simultaneously not being in $-A +R_i$ for every $i$. Therefore, for any finite set $M \subset \co_K$ containing $b_j$, there is no $x$ such that  $S_x(\fr) \cap M = A'$, and so $A$ (along with any other admissible set containing $A'$) cannot be in $X_R$.
				
				Alternatively, we could have that for every $j$, there is a finite positive number of indexes $i_j$ such that $-b_j + R_i \not \subset -A + R_i$. Write  $M = A' \cup \{b_1, \dots, b_l\}$. We will show that there is no $x \in \frp$ such that $S_x(\fr) \cap M = A'$. We will do this by contradiction, assuming that this is the case, and concluding that there is a set of indexes $i_j$ contradicting our hypothesis. 
				
				Assume that such an $x \in \frp$ exists. Then, for each $j$, there is some $i_j$ such that $b_j \in -x+R_{i_j}$, that is, $x \in -b_j + R_{i_j}$. Since $x \not \in R'_{i_j}$ (as it is an element of $\frp$), this means that for each $i_j$ we have $x \in (-b_{j} + R_{i_j})\setminus  R'_{i_j}$, that must therefore be a non-empty set for every $i_j$. We are left with showing that if we have $i_j = i$ for $j$ in a finite set $Q\subset \{1,\dots,l\}$, then $$\bigcap_{j \in Q} (-b_j + R_i) \setminus R'_i \neq \emptyset.$$ But this is clear, since $x $ must belong to this set.
				
				We conclude that if there is no choice of $i_j$ for which $-b_j + R_{i_j} \not \subset -A' + R_{i_j}$ for every $j$, with $\bigcap_{j \in Q} (-b_j + R_i) \setminus (-A'+R_i) \neq \emptyset$, if $i_j = i$ for every $j$ in some finite set $Q$, then $A' \not \in X_R$, and indeed, there is no $x$ such that $S_x(\fr) \in C^R_{A',\{b_1,\dots,b_l\}}$. 
			\end{proof}
			
			\begin{remark}
				\label{rmk:infinite indexes Xr OmegaR}
				Let $R$ be a sieve satisfying the hypothesis of  \Cref{lm: Conditions for A in X_R}. If we can show that for every finite admissible set $A'$ and $x \not \in A'$, there are infinitely many $i$ such that $-x + R_i \not \subset -A'+R_i$, then for any finite collection of $x_1, \dots, x_k$ not in some admissible set $A$, we can find indexes $i_1, \dots, i_k$ all distinct such that $-x_j + R_{i_j} \not \subset -A + R_{i_j}$. By  \Cref{lm: Conditions for A in X_R}, it follows that this is a sufficient condition to show that  $X_R = \Omega_{R}$. We will sometimes use this in what follows. Yet, this condition is too strong, as it may happen that we have some $x$ such that $-x + R_i \not \subset -A+R_i$ for only finitely many $i$, but $A \in X_R$.
				
				Take the example of the sieve $R$ defined by $R_1 = \{3,4\} + 8\z$ and $R_i = \{4,5,6\} + p_i^2\z$, whenever $i \geq 2$. The set $A = \{0,3\}$ is clearly an admissible set for $R$, and we have $-A+R_1 = \{0,1,3,4\}+8\z$, $-A+R_i = \{1,2,3,4,5,6\}+p_i^2\z$ when $i \geq 2$.  We now consider the numbers $1$ and $2$, which are not in $A$. The set $-1+R_1$ equals $ \{2,3\}+8\z$ which is not contained in $-A+R_1$, and neither is $-2+R_1 = \{1,2\}+8\z$.  But for any other $i$, both $-1+R_i = \{3,4,5\} + p_i^2\z$ and $-2+R_i = \{2,3,4\} + p_i^2\z$ are contained in $-A+R_i $. Since $2 \in (-1+R_1) \cap (-2+R_2)$, but it is not in $-A+R_1$, and for every other $j \not \in \{0,1,2,3\}$, we have $-j +R_i = \{-j+4,-j+5,-j+6\} + p_i^2\z  \not \subset \{1,2,3,4,5,6\}+p_i^2\z$ if $i$ is big enough,  \Cref{lm: Conditions for A in X_R} implies that $A \in X_R$. 
				
			\end{remark}
			
			We can tweak the previous example, to show that it is not enough to just assume that for every $b \not \in A$, there is some $i $ such that $-b + R_{i} \not \subset -A+ R_{i}$, to show that $A \in X_R$. Let $R$ be now the sieve defined by $R_1 = 0+8\z$, and $R_i = \{4,5,6\} + p_i^2\z$, whenever $i \geq 2$. We again consider the admissible set $A = \{0,3\}$, which satisfies $-A+R_1 = \{-3,0\}+8\z$, and the elements $1$ and $2$, which are not in $A$. Clearly both $-1 + R_1$ and $-2+R_1$ are not contained in $-A+R_1$, but this fails when we take $i \geq 2$. Notice how $A$ cannot be in $X_R$, since if $\fr$ has a "hole" (a sequence $x,x+1,\dots, x+k$ all of which are not in $\fr$), then it must either be of size $1$, or of size $\geq 3$, but $A$ has a hole of size $2$.

			\begin{theorem}
				\label{thm: nu_R(X_R) = 1}
				Let $R$ be an Erdős sieve. Then, there is a Følner sequence $I_N$ with respect to which $R$ has weak light tails if and only if $$\nu_R(X_R) = 1.$$
				
			\end{theorem}
			
			\begin{proof}
				If $R$ has weak light tails with respect to some $I_N$, then $\fr$ is generic with respect to $I_N$ by  \Cref{thm:Fr is generic}, which implies that $\nu(X_R) = 1$ by  \Cref{lm: x generic implies full meausure orbit closure}.
				
				On the other hand, assume that $\nu_R(X_R) = 1$. By \Cref{thm: R is OmegaR=OmegaRp to some sieve with strong light tails}, it follows that there is some sieve $R'$ such that $\Omega_R = \Omega_{R'}$, $R'$ has strong light tails for $B_N$ and $\frp \in X_R$. This means that for every $N$, there are $a_N$ such that $$S_{a_N}(\fr)\cap B_N = \frp \cap B_N. $$ Let $I_N$ be the Følner sequence $$I_N:= a_N+B_N.$$ Then we have $$\fr \cap I_N = \fr \cap (a_N+B_N) = a_N + (S_{a_N}(\fr) \cap  B_N) = a_N + (\frp \cap B_N).$$ By  \Cref{thm: Omegar = Omegar' implies equal measures}, we know that since $\Omega_R = \Omega_{R'}$ we have $\nu_R = \nu_{R'}$. Consequently,  we have $$d_I(\fr) = d(\frp)  = \nu_{R'}(C^{R'}_{\{0\},\emptyset}) = \nu_R(C^R_{\{0\},\emptyset}) .$$ Therefore $R$ has weak light tails for $I_N$ by  \Cref{thm:Fr is generic}.		
			\end{proof}

			We provide some extra intuition to the fact that if $R$ has weak light tails for some $I_N$ then $\nu(X_R) = 1$. When $R$ has weak light tails for some Følner sequence $I_N$, the Mirsky measure $\nu_{R}$ quantifies the prevalence of patterns in $\fr$ (in the sense that $\nu_R(C^R_{A,B}) = d_I(\{x\in \co_K: x+A \subset \fr \text{ and } (x+B)\cap \fr = \emptyset\})$). Therefore any set of the form $C^R_{A,B}$, where $A,B$ describe a finite pattern that does not appear in $\fr$, should have measure $0$. Consequently, the admissible sets that contain a finite pattern that does not appear in $\fr$ should all be contained in a set of measure $0$. By \Cref{lm: Conditions for A in X_R}, any admissible set $A$ that contains any of these patterns, does not belong to $X_R$, hence we should expect for $\nu_R(X_R)$ to be $1$. This same line of thinking was used in \cite{Dymek} to show that an analogue of $\nu_R(X_R) =1$ holds for every pseudosieve over $\z$ of the form $R_i = b_i\z$ (see Corollary 4.3 in \cite{Dymek}).

			By \Cref{lm: Conditions for A in X_R}, if a finite admissible set $A$ is not in $X_R$, then there is some $B$ such that $\nu_R(C^R_{A,B}) = 0$. The following theorem, which generalizes a result in  \cite{Bartnicka}, elucidates the relation between these implications.

			\begin{theorem}
				\label{thm: Finite Admissible sets have positive density}
				Let $R$ be an Erdős sieve with weak light tails for any Følner sequence $I_N$. Then a finite set $A$ is $R-$admissible, if and only if $\nu_R(C^R_{A,\emptyset}) > 0$.
				
				Additionally, an $R-$admissible set $A$, we have $A\in X_R$  if and only if for any finite $A' \subset A$ and $B\subset \co_K$ disjoint from $A$, we have $\nu_R(C^R_{A,B}) > 0$. 
				
				In particular, we have that $X_R = \Omega_{R}$ if and only if for every finite $R-$admissible set $A$ and finite $B$ disjoint from $A$, we have $\nu_R(C^R_{A,B}) > 0$.
			\end{theorem}
			
			\begin{proof}
				
				If $A$ is a finite set and $\nu_R(C^R_{A,\emptyset}) > 0$, then there is some $Q \in C^R_{A,\emptyset}$ which by definition is in $\Omega_R$ and $A \subset Q$. Consequently, $A$ must also be in $\Omega_R$.
				
				On the other hand, using \Cref{eq:Formula for nu_R with B empty}, we have that $\nu_R(C^R_{A,\emptyset}) > 0$ is equivalent to $$\prod_i\left(1-\frac{|-A+R_i|}{N(\gb_i)}\right) >0.$$ 
				
				Since $R$ is Erdős, this is equivalent to $-A+R_i \neq \co_K$ for all $i$, which is equivalent to $A \in\Omega_{R}$ by definition.
				
				We now show that if $X_R = \Omega_R$, then for any finite admissible set $A$ and some $B$ disjoint from $A$ we have $\nu_R(C^R_{A,B}) > 0$.
				Take any finite $A \in \Omega_R$, and $B$ a finite set disjoint from $A$. Let us write $B = \{b_1,\cdots,b_r\}$. Since $A$ is admissible we have that $\nu_R(C^R_{A,\emptyset}) > 0$. We will use this to prove that $\nu_R(C^R_{A,B}) > 0$.

				By hypothesis, since $A \in \Omega_{R}$, it is in $X_R$. By  \Cref{lm: Conditions for A in X_R}  there is a set $\mathcal{S} = \{s_1,\dots, s_r\}$ of indexes, such that $-b_j +R_{s_j} \not \subset -A +R_{s_j} $ for every $j$, with $\bigcap_{k \in T} (-b_k + R_i) \setminus (-A+R_i) \neq \emptyset$, if $s_k = i$ for every $k$ in some finite set $T$. We partition $\{1,2,\dots,r\}$ by an equivalence relation where $i$ and $j$ are equivalent if $s_i = s_j$. For $j$ in an equivalence class $C$ where $s_j = i$, we choose $x_j$ that satisfy $x_j \in \bigcap_{k \in C} (-b_k + R_i) \setminus (-A+R_i)$.
				
				Given any $h \in \varphi_R^{-1}(C^R_{A,\emptyset})$, we define $$g_i = \begin{cases}
					x_j,  & \text{ if } i = s_j \in \mathcal{S}, \\
					h_i,
					& \text{ otherwise.} 
				\end{cases} $$ 
				At most a finite number of distinct $h$ can produce the same $g$ by this process (they can only be distinct for indexes in $\mathcal{S}$), so by showing that every such $g$ belongs to $\varphi_R^{-1}(C^R_{A,B})$, we will get that $\nu_R(C^R_{A,B}) >0$.  This corresponds to showing that $g_i \not \in -A + R_i$ for every $i$, and that for every $1\leq j \leq r$, there is some $k_j$ such that $b_j + g_{k_j} \in R_{k_j}$.
				
				If $i = s_j \in \mathcal{S}$, then $b_j + g_{s_j}  \in R_{s_j}$ since $x_j \in -b_j + R_{s_j}$, so we see that $g$  is contained in $ \varphi_R^{-1}(C^R_{\emptyset,B})$. It remains to show that $g_i \not \in -A + R_i $ for every $i$. By the definition of $x_j$, this is immediate for $i \in S$. If $i \not \in S$, then $g_i = h_i$ for some $h \in \varphi_R^{-1}(C^R_{A,\emptyset})$, which is equivalent to $h_i \not \in -A+R_i$ for every $i$. This conclude the proof that $\nu_R(C^R_{A,B}) >0$.
				
				We now show the converse. Take any $A \in \Omega_R$. To show that $A$ is in $X_R$, we will show that for any finite set $M$, there is some $x \in \co_K$ such that $A\cap M = S_x(\fr) \cap M$. Let $A' := A \cap M$ and $B := M\setminus A'$. Then $A'$ is finite admissible and disjoint from $B$, so $\nu_R(C^R_{A',B})>0$ by hypothesis. Since $R$ has weak light tails for $I_N$, we have by \Cref{thm:Fr is generic} that $$ \frac{1}{|I_N|}\sum_{a \in I_N} \one_{C^R_{A',B}}(S_a(\fr)) \rightarrow  \nu_R(C^R_{A',B}) > 0,$$
				and so we conclude that there must exist some $x \in \co_K$ such that $S_x(\fr) \cap M = A' = A \cap M$. 			
			\end{proof}

			In \cite{Dymek}, pseudo sieves of the form $R_i = b_i \z$ are considered. In this case, not only can we have $X_R \subsetneq \Omega_R$, but there can also exist a third system $\widetilde{X}_R$\index{Smallest Hereditary Subshift} between these two, of the smallest hereditary subshift containing $X_R$.
			
			\begin{definition}
				\label{def: hereditary}
				We say a set $X \subset \{0,1\}^{\co_K}$ is \index{Hereditary System}\textit{hereditary} if for $A \in X$, if $B \subset A$, then $B \in X$.
			\end{definition}
			
			We have that \begin{equation}
				\label{eq: tilde(X) R}
				\widetilde{X}_R = \overline{\{A: A \subset B \text{ for some }B \in X_R\}},
			\end{equation}given that this set contains $X_R$, is closed, hereditary, and it must be contained in any hereditary closed system containing $X_R$.
			
			\begin{example}
				Consider the sieve $R$ defined by $$R_{1} = 0 + 4\z \hspace{15pt} R_{2i} = (i+1) + p_{2i}^2\z \hspace{15pt} R_{2i+1} = -(i+1) +  p_{2i+1}^2\z. $$ We have that $\fr = \{-1,1\}$ since $0 \in R_1$, $i \in R_{2(i-1)}$ if $i\geq 2$, and $i \in R_{2(-i)-1}$ if $i \leq -2$. The set $\fr$ has arbitrarily large holes so $\{\emptyset\} \cup (\co_K+ \fr) \subset X_R$. We show this is an equality. Take any $Y$ not in $\{\emptyset\} \cup (\co_K+ \fr)$. If $|Y| > 2$, it is clear that for any $N$ such that $|B_N \cap Y| >2$, we cannot have $S_a(\fr) \cap B_N = Y \cap B_N$, so $Y \not \in X_R$. If $Y = \{y_1,y_2\}$ with $|y_1 - y_2| \neq 2$, taking $N \geq 5$, we again see that it is impossible to have  $S_a(\fr) \cap B_N = Y \cap B_N$. Finally, taking $|Y|= \{y\}$, we see that there cannot be closer elements to $Y$ in $\co_K+\fr$  than $S_{-y+1}(\fr)$ and $S_{-y-1}(\fr)$ but they don't equal it.
				
				Because $X_R$ does not contain any set with one element, we have that $X_R$ is not hereditary. Taking $\widetilde{X}_R$ to be the set that contains $X_R$ and all singletons (sets of the form $\{x\}$ with $x \in \co_K$), we get an hereditary system that contains $X_R$. Note that this is a very different set from $\Omega_R$, since it is countable, in opposition to  $\Omega_R$ which is uncountable.    
			\end{example}
			
			This example shows that we can have $X_R \subsetneq \widetilde{X}_R \subsetneq \Omega_{R}$ if $R$ does not have weak light tails for any Følner sequence. Yet, as a consequence of  \Cref{thm: Finite Admissible sets have positive density}, this cannot happen if $R$ has weak light tails for some Følner sequence $I_N$. 
			
			\begin{theorem}
				\label{thm: tilde X R = Omega R}
				Let $R$ be an Erdős sieve with weak light tails for some Følner sequence $I_N$. Then, $$\widetilde{X}_R = \Omega_{R}.$$
			\end{theorem} 
			
			\begin{proof}
				
				We start by showing that any finite $R-$admissible set $A$ is in $ \widetilde{X}_R$. By  \Cref{thm: Finite Admissible sets have positive density}, if $A$ is finite then $\nu_R(C^R_{A,\emptyset}) > 0$. Because $R$ has weak light tails for some Følner sequence $I_N$, this implies (by  \Cref{thm:Fr is generic}) that there is some $x\in \co_K$  such that $S_x(\fr) \in C^R_{A,\emptyset}$, that is, such that $A \subset S_x(\fr)  $. Since $\widetilde{X}_R$ contains $X_R$, we must have $S_x(\fr) \in \widetilde{X}_R$, and by the hereditary property, we necessarily have $A \in \widetilde{X}_R $. 
				
				Now, take any arbitrary $A \in \Omega_{R}$, and let $A_1\subset A_2\subset \dots$ be a sequence of finite subsets of $A$ such that $\bigcup_i A_i = A$. As we have seen, all of these belong to $ \widetilde{X}_R$. The sequence $A_i$ is Cauchy, since for any $N$, taking the smallest index $m$ such that $\bigcup_{i=1}^\infty A_i \cap B_N = \bigcup_{i=1}^m A_i \cap B_N$, we have that $d(A_i,A_j) \leq 1/N$ if $i,j \geq m$. Since $ \widetilde{X}_R$ is closed, the sequence $A_i$ converges to some $Y \in  \widetilde{X}_R$. But this sequence converges in $\Omega_R$ to $A$, so we must have $Y=A$, that is, $A \in \widetilde{X}_R$.
			\end{proof}
			
			In particular, this shows that if $R$ is Erdős and has weak light tails from some Følner sequence, then $X_R$ is hereditary if and only if it is equal to $\Omega_{R}$.
			
			Given two sieves $R$ and $R'$, both with weak light tails with respect to some Følner sequences, we would now expect that if $X_R = X_{R'}$, then $\widetilde{X}_R = \widetilde{X}_{R'}$ which by  \Cref{thm: tilde X R = Omega R} will imply that $\Omega_R =\Omega_{R'}$. Indeed, this is an equivalence, as we now show.

			\begin{corollary}
				\label{crl: OMEGAR= OMEGAR' AND XR = OMEGAR implies XR = XR'}
				Let $R$ and $R'$ be two Erdős sieves with weak light tails for some (not necessarily common) Følner sequence. Then $X_R = X_{R'}$ if and only if  $\Omega_R =\Omega_{R'}$.
			\end{corollary}
			
			\begin{proof}
				Suppose that $X_R = X_{R'}$.	Then, the smallest hereditary set that contains $X_R$ and the smallest hereditary set that contains $X_{R'}$  must agree as can be seen from \Cref{eq: tilde(X) R}. This means that $\widetilde{X}_R = \widetilde{X}_{R'}$, which by  \Cref{thm: tilde X R = Omega R} implies that $\Omega_R =\Omega_{R'}$.

				We now show that if $\Omega_R =\Omega_{R'}$, then $X_R = X_{R'}$. Since $\Omega_R =\Omega_{R'}$, we always have that $C^R_{A,B} = C^{R'}_{A,B}$. By  \Cref{thm: Finite Admissible sets have positive density}, we know that $A \in X_R$ if and only if for every finite $A' \subset A$, and $B \cap A = \emptyset$, we have $\nu_R(C^R_{A',B}) >0$. Since $\Omega_R =\Omega_{R'}$ implies that $\nu_R = \nu_{R'}$ (\Cref{thm: Omegar = Omegar' implies equal measures}), this holds if and only if we have $\nu_{R'}(C^{R'}_{A',B}) >0$, so $X_R=X_{R'}$.

			\end{proof}

			\begin{remark}
				There is another proof of the fact that if $X_R = X_{R'}$, then $\Omega_{R} = \Omega_{R'}$. Indeed, suppose that we have sieves $R$ and $R'$ such that $X_R = X_{R'}$ but $\Omega_{R} \neq \Omega_{R'}$. Then, there is a set $A \in\Omega_{R'} $ that is not $R-$admissible. Consequently, this set has a finite subset $A'$ that is also not $R-$admissible (there must be some $i$ such that $-A+R_i = \co_{K}$, so we just take a finite subset $A'$ such that $-A'+R_i = \co_{K}$). Since $R'$ is Erdős, we will have $\nu_{R'}(C^{R'}_{A',\emptyset}) >0$. But $X_R \cap C^{R'}_{A',\emptyset} = \emptyset$, since $A' \not \subset S_a(\fr)$ for any $a \in\co_K$. Consequently, we must have $\nu_{R'}(X_{R'}) = \nu_{R'}(X_{R}) < 1$. This contradicts  \Cref{thm: nu_R(X_R) = 1}, so we must have that $\Omega_{R} = \Omega_{R'}$.
			\end{remark}
			
			Let $R$ and $R'$ be minimal Erdős sieves with weak light tails for some (not necessarily common) Følner sequences. Using  \Cref{thm: Equaivalence for omegaR = omegaRp} together with  \Cref{crl: OMEGAR= OMEGAR' AND XR = OMEGAR implies XR = XR'}  we get the following result.

			\begin{theorem}
				\label{thm: equivalence XR =XRp and OmegaR = OmegaRp}
				Let $R$ and $R'$ be minimal Erdős sieves with weak light tails for some (not necessarily common) Følner sequences. The following are equivalent. 
				
				\begin{enumerate}
					\item $\cb_R = \cb_{R'}$, and for every $\gb \in \cb_R$, there is some $\delta_\gb\in \co_K$ such that $R_\gb = \delta_\gb+ R'_\gb,$ 
					\item $X_R = X_{R'}$,
					\item $\Omega_{R} = \Omega_{R'}.$
				\end{enumerate} 
			\end{theorem}

			By using  \Cref{lm: Conditions for A in X_R}, we can give a full characterization of those sieves $R$ for which $X_R = \Omega_{R}$. Yet, it is not very easy to use this condition to say if a given sieve $R$ satisfies $X_R = \Omega_{R}$ or not. We now provide an example of a condition that is easy to verify, and that is sufficient (but not necessary) for a sieve $R$ to satisfy  $X_R = \Omega_{R}$.
			
			To do this, we start by defining a function $\lambda$ on the powerset of $\co_K$ such that for $S \subset \co_K$ we have
			\begin{equation}
				\label{eq: lambdaS}
				\lambda(S) := \min_{x,y \in S} |x-y|.
			\end{equation}  We now want to consider sieves $R$ such that $\limsup_{i} \lambda(R_i) = \infty$. For any $\cb$-free system this clearly holds, since $\lambda_1(\gb_i)$ goes to infinity as the norm of the ideal $\gb_i$ grows. Meanwhile, for the sieve $R$ defined by $R_p = \{0,1\} + p^2\z$, we have $\limsup_{i} \lambda(R_i) = 2$.
			
			\begin{proposition}
				\label{prop: lim lambda infty implies XR = OmegaR}
				Let $R$ be an Erdős sieve with weak light tails for some Følner sequence, such that $\limsup_{i} \lambda(R_i) = \infty$. Then $X_R = \Omega_{R}$.
			\end{proposition}
			
			\begin{proof}
				As pointed out in \Cref{rmk:infinite indexes Xr OmegaR}, it is enough to show that for any finite admissible $A$ and $x \not \in A$ there are infinitely many indexes $i$ such that $-x +R_i \not \subset -A'+R_i$.
				
				To show this, take any $N$ so big that $A-x \subset B_N$. The condition $\limsup_{i} \lambda(R_i) = \infty$ implies that there are infinitely many $i$'s such that $R_i \cap (r_i + B_N) = \{r_i\}$ for any $r_i \in R_i$. Since $A-x \subset B_N$ and $x \not \in A$, the condition $(r_i + B_N) \cap R_i  = \{r_{i}\}$ implies that $r_{i} + (A-x) \cap R_{i} = \emptyset$ (since $A-x$ does not contain $0$), and so $(-x+r_i)\not \in (-A+R_i)$. Since $r_i$ was an arbitrary element of $R_i$, it follows that $(-x+R_i)\cap(-A+R_i) = \emptyset$ holds for infinitely many $i$, as we wanted to show.		
			\end{proof}
			
			\begin{remark}
				\label{rmk: sieve with not infinite lambda but X and Omega equal}
				The condition  $\limsup_{i} \lambda(S_i) = \infty$ is too powerful. Indeed, consider the sieve $R$ defined by $R_p = \{0,1\}+p^2\z$, with $\cb_R= \{p^2\z: p \text{ prime}\}$. The set $\{0,2\}$ is not admissible, and indeed, if $A$ is any admissible set,  $x \in A$ implies $x+2\not \in A$. We now show that $X_R = \Omega_{R} $, by showing that for every finite admissible set $A$, $x\not \in A$, and $i$ large enough, we have $$-x+R_i = \{-x,-x+1\}+p_i^2\z \not \subset -A+R_i.$$ 
				
				If $i$ is big enough, computations in $\z/p_i^2\z$ are the same as in $\z$, and we have the equality $$-A+R_i = (-A +p_i^2\z)\cup (-A + 1+ p_i^2\z). $$ Therefore, $\{-x,-x+1\}+p_i^2\z  \subset -A+R_i$ implies $-x \in  -A$ or $-x \in -A+1$. The first case can't happen since $x \not \in A$, so we get $-x \in -A+1$, that is $x+1 \in A$. Similarly, we can't have $-x+1 \in -A+1$, so we get $-x+1 \in -A $, that is $x-1 \in A$. But then, $x-1$ and $x+1$ must both be in $A$, which cannot happen since $A$ is admissible and $(x+1)-(x-1) = 2$.
			\end{remark}

			Comparing the examples of the sieves in \Cref{rmk: sieve with not infinite lambda but X and Omega equal} and \Cref{ex: sieve such that X different Omega}, we see that small changes in a sieve can change whether $X_R = \Omega_{R}$, since in the remark we get a sieve $R$ for which this holds, but in \Cref{ex: sieve such that X different Omega} we provide a sieve $R'$  obtained from $R$ by removing just one ideal, such that $X_{R'} \neq \Omega_{R'}$.  
			
			Consider a case where we have two sieves $R$ and $R'$ which are Erdős $\cb-$free systems. They are both minimal, so  \Cref{lm: OmegaR =omegar' implies equality of base} implies that $\Omega_R = \Omega_{R'}$ if and only if $\cb_R = \cb_{R'}$. Since Erdős $\cb-$free systems have strong light tails for $B_N$ (as shown in  \Cref{thm: K etale algebra light tails for B_n finite}),  \Cref{thm: equivalence XR =XRp and OmegaR = OmegaRp} together with \Cref{crl: Bfree systems free elements equal iff same support} give the following result, which generalizes Proposition 2.3.1 of \cite{Kulaga} for Erdős $\cb-$free systems over étale $\q-$algebras.
			
			\begin{corollary}
				\label{crl: equivalence XR OmegaR for b free systems}
				Let $R$ and $R'$ be Erdős $\cb-$free systems over an étale $\q-$algebra $K$. Then the following are equivalent
				
				\begin{enumerate}
					\item $\cb_R = \cb_{R'}$,
					\item $\fr = \frp$,
					\item $X_R = X_{R'}$,
					\item $\Omega_{R} = \Omega_{R'}.$
				\end{enumerate} 
			\end{corollary}
			
			If $R$ is an Erdős $\cb-$free system over a number field $K$ of degree $n$, then  \Cref{prop: lim lambda infty implies XR = OmegaR} implies that $X_R = \Omega_R$, using the fact that $\lambda(R_i)$ will grow to infinity as $i$ increases by Corollary 4 in \cite{Fraczyk}. Hence, the equivalence between (3) and (4) in \Cref{crl: equivalence XR OmegaR for b free systems} becomes 'trivial' in the sense that it is stating the same thing twice. Yet, it may happen that this is not the case if $R$ is a sieve over an étale-$\q$ algebra.
			
			\begin{example}
				Let $R$ be the sieve over $\q\times \q$ given by $R_p = p^2\z \times \z$ for all primes $p \geq 2$. Writing $\cs$ to be the squarefree numbers, we have $\fr = \cs\times \z$. The set $A = \{(0,0),(0,2)\}$ is admissible, since $|-A+R_p| = 2 < p^2$ for all primes $p$. Yet, it is impossible for $A \cap B_N$ to be equal to $ S_a(\fr) \cap B_N$ for any $a \in \co_K$ and $N \geq 3$, given that $(x,y) \in S_a(\fr)$ implies that $(x,y+1) \in S_a(\fr)$, but $(0,0) \in A$ while $(0,1) \not \in A$. It follows that $A \in \Omega_{R} \setminus X_R$.
			\end{example}

			\section{Applications to Number Theory}
			
			In this section we provide number theoretic applications of the theory of Erdős sieves. First we investigate sets $C$ such that there are infinite $A,B \subset \z$, with $d(B)>0$ and $A+B \subset C$. We then look at squarefree values of polynomials as $R-$free numbers of specific sieves. We conclude with an Ergodic Prime Number Theorem that generalizes the results of \cite{Wang}.
			
			\subsection{Infinite Patterns in $\fr$}
			
			We again consider the measure space $(\cs_R, \sigma_R)$ as defined in Equations (\ref{eq: CS_R}) and (\ref{def: sigmaR}). Given a sieve $R$ and a set $A \subset \co_K$, we will write $A+R$ to be the sieve defined by $$(A+R)_i = A+R_i$$ for every $i\in \mathbb{N}$. We have the following result.
			
			\begin{theorem}
				\label{thm: meaure of sieves such that A+R' has weak light tails is 1}
				Let $R$ be an Erdős sieve and $A$ an infinite $R-$admissible set. If $$\prod_i \left(1-\frac{|-A+R_i|}{N(\gb_i)}\right) >0,$$ then $$\sigma_R(\{R'\in \cs_R: (-A+R') \text{ has strong light tails with respect to } I_N\}) = 1$$ for any tempered Følner sequence $I_N$.
			\end{theorem}
			
			\begin{proof}
				When $A$ is finite, note that if $R$ has strong light tails for $I_N$, then so does $-A+R$, given that $-A+R = \bigcup_{a\in A}(-a+R)$, and the union of sieves with strong light tails has strong light tails by \Cref{lm: union of sieves}. Hence, for finite $A$ the result follows directly from \Cref{thm: R is OmegaR=OmegaRp to some sieve with strong light tails}.
				
				To show the result for infinite $A$, we start by showing that $$\sigma_R(\{R'\in \cs_R: (-A+R') \text{ has weak light tails with respect to } I_N\}) = 1$$ and then proceed as in the proof of \Cref{thm: R is OmegaR=OmegaRp to some sieve with strong light tails}.
				
				Let $A$ be an infinite $R-$admissible set, and let us write $\Gen_A(\Omega_{R},I_N)$ to be the set of points of $\Omega_{R}$ such that $$\lim_{N \rightarrow \infty}\frac{1}{|I_N|}\sum_{a \in I_N} \one_{C^R_{A,\emptyset}}(S_a(x)) = \nu_R(C^R_{A,\emptyset}). $$
				The function $\one_{C^R_{A,\emptyset}}$ is not continuous, but it is in $L^1(\Omega_{R})$, given that it is a bounded function in a compact space. Hence, the Pointwise Ergodic Theorem  (\Cref{thm: ergodic theorem}) implies that for any tempered Følner sequence $I_N$, we have \begin{equation}
					\label{eq: nu Gen A is 1 }
					\nu_R(\Gen_A(\Omega_{R},I_N)) = 1.
				\end{equation}
				
				By our hypothesis that $$\prod_i \left(1-\frac{|-A+R_i|}{N(\gb_i)}\right) >0,$$ we have that $\nu_R(C^R_{A,\emptyset})>0$ and that $-A+R'$ is an Erdős sieve for any $R' \in \cs_R$.

				Let $\Psi_R:\cs_R\rightarrow\Omega_{R}$ be  the map that sends $R'$ to $\frp$. For $\frp$ to be in $\Psi_R(\cs_R) \cap \Gen_A(\Omega_{R},I_N)$ is equivalent to $$\nu_R(C^R_{A,\emptyset}) = \lim_{N \rightarrow \infty}\frac{1}{|I_N|}\sum_{a \in I_N} \one_{C^R_{A,\emptyset}}(S_a(\frp)) = d_I(\{x \in \co_K: x+ A \subset \frp \}) = d_I(\cf_{(-A+R')}).  $$ 
				Using $\nu_R(C^R_{A,\emptyset}) = \nu_{(-A+R)}\left(C^{(-A+R)}_{\{0\},\emptyset}\right)$, this is equivalent to $(-A+R')$ having weak light tails for $I_N$ by \Cref{thm:Fr is generic}. 
				
				Therefore showing that $$\sigma_R(\{R'\in \cs_R: (-A+R') \text{ has weak light tails with respect to } I_N\}) = 1$$ follows from showing that $\nu_R(\Psi_R(\cs_R) \cap \Gen_A(\Omega_{R},I_N)) = 1$. From the proof of \Cref{lm: sieves with weak light tails have measure 1} we know that $\nu_R(\Psi_R(\cs_R)) = 1$. Together with \Cref{eq: nu Gen A is 1 } the result follows. 
				
				Let $$H_R = \bigoplus_i \co_K/F(R_i)$$ and $V$ be the action of $H_R$ is $\cs_R$ given by $V_h(R)_i = h_i+R_i$. By \Cref{thm: R has strong light tails if finite translations}, it is equivalent for the sieve $-A+R$ to have strong light tails for $I_N$, and for $-A+V_h(R)$ to have weak light tails with respect to $I_N$ for every $h \in H_R$. Hence, the set $$\{R'\in \cs_R: (-A+R') \text{ has strong light tails with respect to } I_N\}$$ is equal to $$\bigcap_{h\in H_R} V_{h}^{-1}(\{R'\in \cs_R: (-A+R') \text{ has weak light tails with respect to } I_N\}).$$ Since $\sigma_R$ is invariant under $V$, this is a countable union of sets of measure 1, consequently so is their intersection, as we wanted to show.
			\end{proof}
			
			Of course, this does not imply that if $R$ has strong light tails for some $I_N$, then $-A+R$ will also have strong light tails, just that there will be some $g_i \in \grf$ and a sieve $R = R(g_i)'$ such that $-A+R'$ has strong light tails for $I_N$.
			
			\begin{example}
				Let $R$ be the cubefree sieve given by $R_i = p_i^3\z$ and define a set $A = \{a_1,a_2,\dots\}$ by choosing the $a_i$ in such a way that
				
				\begin{enumerate}
					\item $a_1 = -1,$
					\item $a_i \equiv -i \mod p_i^3$ for all $i$,
					\item $a_i \equiv a_j \mod p_j^3$ for all $j<i$.
					
				\end{enumerate} 
				
				The Chinese Remainder Theorem guarantees us that we can build such a set. Additionally, we have that $|-A+p_i^3\z| \leq i$ for all $i$, and so $-A+R$ will be an Erdős sieve. Yet, it is clear that it does not have weak light tails for $I_N := [1,N]$, since $\cf_{-A+R}\cap I_N = \emptyset$, given that $i \in -A+R_i$ for all $i \in \mathbb{N}$. 
			\end{example}

				When $R$ is a sieve with weak/strong light tails, then for any admissible finite $A$ we will have that $-A+R$ will have weak/strong light tails, as implied by \Cref{lm: union of sieves}. On the other hand, if $A$ is any admissible set, we have that if $-A+R$ has strong light tails, then $R$ has strong light tails. This is because if $-A+R$ has strong light tails, then for any $a \in A$, we get that $a-A+R$ has strong light tails. Since $R_i \subset a-A+R_i $ for every $i$, this implies that $R$ will have strong light tails. 
				
				Yet, the same does not happen when we replace strong by weak light tails, even assuming that $A$ is finite.
				
				\begin{example}
					Let $A = \{-2,-1,0\}$ and $R$ be the sieve defined by $$R_1 = \{0,2\}+8\z \hspace{15pt} R_i = 8(i-1) +\{0,1\}+p_i^2\z.$$ As in Example 5.14 of \cite{part1}, we can show that $R$ does not have weak light tails for any Følner sequence. But $-A+R_1 = \{0,1,2,3,4\} + 8\z$, and $$-A+R_i = 8(i-1)+\{0,1,2,3\}+p_i^2\z,$$ so $-A+R$ does have weak light tails for $I_N = [0,N]$.
				\end{example} 
				
				As a corollary of 	\Cref{thm: meaure of sieves such that A+R' has weak light tails is 1}, we get the following result.
				
				\begin{corollary}
					\label{crl: A B subset of FR}
					Let $A$ be a subset of $\co_K$ and $R$ an Erdős sieve such that $$\prod_i \left(1-\frac{|-A+R_i|}{N(\gb_i)}\right) >0.$$ Then, there exists a $g = (g_i)_{i \in \mathbb{N}}$ in $\grf$ and some $B$ with $d(B)>0$ such that $$A+B \subset \cf_{R(g)},$$ where $R(g)$ is defined as in \Cref{eq:R(g)}.
				\end{corollary}
				
				\begin{proof}
					By \Cref{thm: meaure of sieves such that A+R' has weak light tails is 1}, there exists some $g\in\grf$ such that $-A+R(g)$ has strong light tails. Taking $B$ to equal $\cf_{-A+R(g)}$, we have that $d(B) >0$ by hypothesis, and $A+B \subset \cf_{R(g)}$. 
				\end{proof}
				
				In \cite{Moreira}, it was shown that for any $C\subset \mathbb{N}$, if $\overline{d}(C)>0$, then there are infinite $A,B \subset \mathbb{N}$ such that $A+B \subset C$. Yet, Host has shown in Proposition 2 of \cite{Host}, that there are sets $C$ with $\overline{d}(C)>0$ such that if there are infinite $A$ and $B$ are such that $A+B \subset C$, then we must have $\overline{d}(A) = \overline{d}(B) = 0$. It would be interesting to know if such an example can appear with $C = \fr$ for some sieve $R$ with strong light tails for $B_N$. We conjecture that this is not the case.
				
				\begin{conjecture}
					\label{cnj: the only one}
					For every Erdős sieve  $R$ with strong light tails for $B_N$, there exists some infinite $A \subset \co_K$ such that $-A+R$ has strong light tails for $B_N$.
				\end{conjecture}
				
				Recently the case where $R$ is the squarefree sieve over $\q$ was solved by van Doorn and Tao in \cite{van Doorn} for the Følner sequence $I_N  = [0,N]$ (see Theorem 8 of this paper). Although their construction does not directly apply for $B_N$, we adapt their argument for completeness, as it is not immediate that the constructed object is indeed a sieve with strong light tails for $I_N$. In order to do this, we will use the following lemma (see Lemma 13 in  \cite{van Doorn}). 
				
				\begin{lemma}
					\label{lm:taovandoorn}
					There exists an infinite sequence of positive integers $m_j$ such that 
					\begin{enumerate}
						\item $m_j \equiv 0 \mod p^2$ for $p \leq 3\exp\exp j$, and
						\item $m_j+a \not \equiv 0 \mod p^2$ if $p> 3 \exp \exp j$ and $1 \leq a \leq \frac{p}{(\log\log p)^2}.$
					\end{enumerate}
				\end{lemma}
				
				Let $A = \{m_j:j \in \mathbb{N}\}$ be a set such that each $m_j$ satisifes the properties in \Cref{lm:taovandoorn}. We claim that if $R$ is the squarefree sieve defined by $R_p = p^2\z$, then the sieve $-A+R$ has strong light tails for $I_N$. Let $W$ be the sieve defined by $W_p = (-A+R_p)\setminus R_p$. Since $-A+R = W \cup R$, and $R$ has strong light tails for $I_N$, \Cref{lm: union of sieves} implies that it is enough to show that $W$ has strong light tails for $I_N$.
				
				Assume that $a \in W_p \cap I_N = I_N\cap (-A+p^2\z)\setminus p^2\z$ for some $p$. This means that there is some $m_j \in A$ such that $a \equiv -m_j \not \equiv 0 \mod p^2$. It follows that $m_j \not \equiv 0 \mod p^2$, and so $p$ cannot be smaller than or equal to $3 \exp\exp j$, as this would contradict point (1) in \Cref{lm:taovandoorn}. We must therefore have that $p > 3 \exp \exp j$, and since $m_j+a \equiv 0 \mod p^2$, we must have that \begin{equation}
					\label{eq:TaoDoorn1}
					N \geq a > \frac{p}{(\log\log p)^2}.
				\end{equation} Note that here we are using the fact that $a$ is positive, and that is why this argument does not immediately apply with $I_N$ replaced by $B_N$. Taking the logarithm on both sides of \Cref{eq:TaoDoorn1} gives that $$\log p < \log(N) + 2 \log \log \log p,$$ and so $\log p \ll \log N$. Replacing this in \Cref{eq:TaoDoorn1} gives $$p \ll N(\log \log N)^2.$$ We conclude that for any fixed $L \geq 1$,  $$\bigcup_{L <p} W_p \cap I_N \subset \bigcup_{L <p \ll N(\log \log N)^2 } W_p \cap I_N. $$
				
				For any fixed $p$, we have that $m_j \equiv 0 \mod p^2$ for all $j$ such that $p \leq 3 \exp \exp j$. All $j$ that are bigger than $\log \log p$ satisfy this, so we conclude that $|W_p| \leq \log \log p$. Since each congruence class modulo $p^2$ intersects $I_N$ in up to $N/p^2+1$ points, we get 
				\begin{align*}
					\left|\bigcup_{L < p} W_p \cap I_N\right|
					&\le
					\sum_{L < p \ll N(\log \log N)^2}
					(\log \log p)\left(\frac{N}{p^2} + 1 \right) \\[4pt]
					&\ll
					N \sum_{L < p \ll N(\log \log N)^2}
					\frac{\log \log p}{p^2}
					\;+\;
					\sum_{L < p \ll N(\log \log N)^2}
					\log \log p .
				\end{align*}
				
				Since the series $\sum (\log \log p)/p^2$ converges, the first term will go to zero once we divide by $N$ and take $L$ to infinity. By using the prime number theorem to count primes up to $N(\log \log N)^2$, it is then easy to see that the second sum is in $o(N)$, which completes the proof.

				If \Cref{cnj: the only one} holds, then an obvious follow up question is which sorts of sets $A$ are such that $\fr \subset A$ for some Erdős sieve $R$ with strong light tails for $B_N$. First, this would surely imply that $\underline{d}(A)>0$. But it is not true that all such sets contain some $\fr$ as a subset, with $R$ an Erdős sieve with strong light tails. This follows quickly from the following observation: if $R$ and $R'$ are Erdős sieves, $R$ with strong light tails for $B_N$, and $\fr \subset \frp$, then $R'$ must also have strong light tails for $B_N$. This shows that simply taking $A = \frp$ for some sieve $R'$ which only has weak light tails for $B_N$, or does not have weak light tail at all, already gives us the example of some $A$ such that $\fr \not \subset A$ for every $R$ with strong light tails for $B_N$.
				
				\begin{proposition}
					\label{prop: slt subset imples slt}
					Let $R$ be an Erdős sieve with strong light tails with respect to some Følner sequence $I_N$. If $R'$ is an Erdős sieve such that $$\fr \subset \frp,$$ then $R'$ has strong light tails with respect to $I_N$.
				\end{proposition}
				
				\begin{proof}
					Our hypothesis is equivalent to the fact that $$\bigcup_i R_i \supset \bigcup_i R'_i.$$ Take any $i \in \mathbb{N}$. As part of the proof of \Cref{lm: aux for Characterization of equivalent sieves}, we showed that if $R'_i \subset \bigcup_j R_j$, then we have $$R'_i \subset \bigcup_{j:(\gb_j,\gb'_i) \neq 1} R_j. $$
					The result now follows by the same argument used in the proof of \Cref{thm: Strong light tails preserved under isomorphism of sieves}.
				\end{proof}
				
				\begin{example}
					Take the sieve $W$ defined by $W_1 = 4\z$ and $W_i = 1+4i + p_i^2\z$ for $i>1$. This sieve does not have weak light tails for any Følner sequence $I_N$ (as seen in Example 5.14. of \cite{part1}), but we have $d(\cf_W)>0$. By \Cref{prop: slt subset imples slt}, there is no sieve $R$ with strong light tails for some Følner sequence $I_N$ such that $\fr \subset \cf_W$.
				\end{example}

				\subsection{Squarefree Values of Polynomials}

				Let $f\in \z[X]$ be an irreducible polynomial. The set of squarefree values of $f$, that is $$ \Sigma_f = \{x \in \z: f(x) \text{ is squarefree}\}$$ is an important object of study in number theory.\index{Squarefree Values of a Polynomial} Consider the sieve $R^f$ such that $\cb_{R^f} = \{p^2\z:  p\text{ prime}\}$ and defined by $$R^f_p = \{m\in \z/p^2\z: f(m) \equiv 0 \mod p^2\} + p^2\z.$$ 
				Noting that $f(y+p^2k) \equiv f(y) \mod p^2$ for any $k \in \z$, we see that for $y$ to be in $\frf$ is equivalent to $f(y) \not \equiv 0 \mod p^2$ for every $p$, which is equivalent to $f(y)$ being squarefree. That is, for any irreducible polynomial $f \in \z[X]$ we have
				\begin{equation}
					\label{eq: sigmaf equals frf}
					\Sigma_f = \frf.
				\end{equation}
				
				Let $\rho_f$ be the function given by $$\rho_f(x) = |\{m\in \z/x\z: f(m) \equiv 0 \mod x\}|.$$ When $x = p^2$ for some prime, it is clear we have $\rho_f(p^2) = |R^f_p|$. We have the following result (see for example Lemma 2.2 in \cite{Brwoning}).
				
				\begin{lemma}
					\label{lm: rho smaller than degree}
					If $f$ is an irreducible polynomial of degree $d$  we have $$\rho_f(p^2) \leq d$$ for all but a finite number of primes $p$.
				\end{lemma}
				
				As a consequence of \Cref{lm: rho smaller than degree}, we see that if $f$ is irreducible, then $R^f$ is Erdős. Notice that this is not the case for the polynomial $f(X) = X^2$, that has only $-1$ and $1$ as squarefree values.
				
				When studying $\Sigma_f$, we usually want to show that it is infinite, by proving that, for $I_N = [0,N]$, 
				\begin{equation}
					\label{eq: formula if squarefree polynomials }
					|\Sigma_f \cap I_N| = N\prod_{p \text{ prime}} \left(1-\frac{\rho_f(p^2)}{p^2}\right) + o(N).
				\end{equation}
				
				Since $R^f$ is Erdős, \Cref{thm:Fr is generic} together with \Cref{eq: sigmaf equals frf}, imply that \Cref{eq: formula if squarefree polynomials } will hold for some polynomial $f$ if and only if $R^f$ has weak light tails for $I_N$. Many authors have worked on this problem, usually also providing an explicit bound for the $o(N)$ error term. Estermann showed in \cite{estermann} that \Cref{eq: formula if squarefree polynomials } holds for the polynomial $f(X) = X^2+l$ for any non-zero integer $l$. 
				
				We can easily show this, by showing that $R^f$ has strong light tails for $I_N$ if $f(X) = aX^2+bX+C$ is a degree 2 polynomial such that $R^f$ is Erdős. Indeed, if $R^f_p$ intersects $I_N$ in some point $x$, then $p^2\mid ax^2+bx+c$ which implies that $p \ll_f N$. Hence, there is some $C$ such that $$\left|I_N \cap \bigcup_{p> L}R^f_p\right| \leq \sum_{L < p < CN}\rho_f(p^2)\left(1+\frac{N}{p^2} \right) \leq N\sum_{L < p < CN}\frac{\rho_f(p^2)}{p^2} + \pi(CN). $$ Dividing both sides by $N$ and letting $N$ go to infinity will make the right hand side go to $0$ if $R^f$ is Erdős, showing that $R^f$ has strong light tails.
				
				The same approach cannot be used to show that any polynomial of degree $3$ has strong light tails, since if $p^2 \mid x^3$ with $|x| \leq N$, the best bound we can provide is $p \ll N^{3/2}$, and then we would be bounding $\left|I_N \cap \bigcup_{i> L}R^f_i\right|$ by $\pi(N^{3/2})$, which is worse than the trivial bound. When $f$ is an irreducible polynomial of degree $3$, Hooley showed in \cite{Hooley 2} that $R^f$ has strong light tails. More generally, he showed the following result.
				
				\begin{theorem}
					\label{thm: Hooley Polynomials}
					Let $f$ be an irreducible polynomial of degree $d \geq 3$. The sieve $R^{f,d-1}$ defined by $$R^{f,d-1}_p = \{m \in \z/p^{d-1}\z: f(m) \equiv 0 \mod p^{d-1}\z \} + p^{d-1}\z$$ has strong light tails.
				\end{theorem}
				
				More generally, taking any integer $l \geq 2$, we write $R^{f,l}$ to be the sieve defined by\index{$l-$free Values of Polynomials} \begin{equation}
					\label{eq: R fl}
					R^{f,l}_p = \{m \in \z/p^{l}\z: f(m) \equiv 0 \mod p^{l}\z \} + p^{l}\z.
				\end{equation} Because $p^{l+1}\mid f(m)$ implies that $p^{l}\mid f(m)$, we have that $R^{f,l+1}_p \subset R^{f,l}_p$ for every $l \geq 2$. Consequently, if $f$ is an irreducible polynomial of degree $d \geq 3$, \Cref{thm: Hooley Polynomials} implies that $R^{f,l}$ has strong light tails for all $l \geq d-1$. In the particular case where $f(X) = X^d+c$ is an irreducible polynomial, Heath-Brown showed in Theorem 1 of \cite{HeathBrown2} that if $l \geq (5d+3)/9$, then $R^{f,l}$ has strong light tails. 
				
				There is no specific irreducible polynomial of degree $d\geq 4$ for which it is currently known that $R^{f,2}$ has weak light tails for $I_N = [0,N]$. Yet, Filaseta showed in \cite{Filaseta} that almost all irreducible polynomials (with respect to their height) have infinitely many squarefree values. Browning and Shparlinski further improved this, by showing that for almost all irreducible polynomials \Cref{eq: formula if squarefree polynomials } holds (for the specific result, see Theorem 1.1 in \cite{Brwoning}).
				
				Finally, we remark that Granville has shown (see Theorem 1 of \cite{granville}) that \Cref{eq: formula if squarefree polynomials } holds for all irreducible polynomials, independently of degree, if we assume the abc-conjecture. A more extensive review of the history of this problem can be found in \cite{Kowalski}.
				
				The squarefree value of multivariate polynomials is also a problem of interest (see \cite{Bhargava} for an application to counting number fields with Galois group $S_n$ for $n \leq 5$). For this reason, in what will follow, we will work in the following level of generality. Given a number field $K$ of degree $n$, integers $v,l \geq 1$ and a polynomial $f \in \co_K[X_1,\dots, X_v]$, we write $R^{f,l}$ to be the sieve supported on $$\cb_R = \{\gp^l \times \dots \times \gp^l: \gp \text{ prime ideal of } \co_K\}$$ and defined by $$R^{f,l}_\gp = \{m \in \co_K/\gp^l \times \dots \times \co_K/\gp^l: f(m) \in \gp^l\} + \gp^l \times \dots \times \gp^l,$$ where $\gp^l \times \dots \times \gp^l$ is the Cartesian product of $v$ copies of $\gp^l$. Note that there may be $\gp$ such that $R^{f,l}_\gp = \emptyset$.
				
				Note that $R^{f,l}$ is always Erdős when $f$ is an irreducible polynomial. Over $\z$ with $l =2$ this is \Cref{lm: rho smaller than degree}. For $\z[X_1,\dots,X_v],$ with $l =2$, see the proof of Theorem 3.2 in \cite{Poone}. The same methods show that $$\left|R^{f,l}_\gp\right| = O_K(N(\gp)^{lv-2})$$ for any arbitrary number field $K$. Using \Cref{thm: density of x+A in fr} we automatically get the following result.
				
				\begin{corollary}
					\label{crl: admissible sets polynomials}
					Let $f \in \co_K[X_1,\dots, X_v]$ be an irreducible polynomial such that $R^{f,l}$ is a sieve with weak light tails with respect to a Følner sequence $I_N$, and $A$ a finite $R^{f,l}-$admissible set. Then, $$d_I(\{x \in \co_K/\gp^l \times \dots \times \co_K/\gp^l: \mathlarger\forall_{a \in A} \hspace{2pt} f(x+a) \text{ is } l-\text{free} \}) = \prod_{\gp} \left(1-\frac{|-A+R^{f,l}_\gp|}{N(\gp)^{lv}}\right).$$
				\end{corollary}
				
				Similarly, let $f_1 , \dots, f_m \in \co_K[X_1,\dots, X_v]$ be a collection of irreducible polynomials. If we want to study those $x$ such that $f_i(x)$ is $l-$free for every $i$, we then just have to consider the sieve $R = \bigcup_{i=1}^m R^{f_i,l}$, since $\fr =  \bigcap_{i=1}^m \mathcal{F}_{R^{f_i,l}}$. Given that weak light tails are preserved under union of sieves by \Cref{lm: union of sieves}, we get the following corollary. 
				
				\begin{corollary}
					\label{crl:intersection of squarefree values of polynomials}
					Let $f_1 , \dots, f_m \in \co_K[X_1,\dots, X_v]$ be irreducible polynomials such that $R^{f_i,l}$ has weak light tails for every $i$. Then, 
					
					$$d_I(\{x \in \co_K/\gp^l \times \dots \times \co_K/\gp^l:  \mathlarger\forall_{1\leq i \leq m} \hspace{2pt} f_i(x) \text{ is } l-\text{free} \}) = \prod_{\gp} \left(1-\frac{|\bigcup_{i=1}^m R^{f_i,l}_\gp|}{N(\gp)^{lv}}\right).$$
				\end{corollary}
				
				\begin{example}
					\label{ex: Dimitrov}
					In \cite{Dimitrov1}, Dimitrov showed that there are infinitely many $x$ such that $x^2+1$ and $x^2+2$ are squarefree, and provides the density of such $x$. We can use \Cref{crl:intersection of squarefree values of polynomials} to obtain this density easily. 
					
					Let $f(X) = X^2+1$ and $g(X) = X^2+2$. Since both are polynomials of degree $2$, we have that $R^f$ and $R^g$ have strong light tails with respect to $I_N = [0,N]$. We also have that $R^f_2 = R^g_2 = \emptyset$ and that $R_p^f \cap R_p^g = \emptyset$ for all $p$. Since $|R^f_p| = \left(\frac{-1}{p}\right)+1$ and $|R^g_p| = \left(\frac{-2}{p}\right)+1$ for $p> 2$ (where $\left(\frac{a}{p}\right)$ denotes the Legendre symbol), it follows that the union $R^f \cup R^g$ is well defined and  $$|R_p^f \cup R_p^g| = \left(\frac{-1}{p}\right)+\left(\frac{-2}{p}\right)+2$$ for all $p \geq 2$, given that $f(x) = g(x)+1$. \Cref{crl:intersection of squarefree values of polynomials} now shows that
					$$d_I\left(\{x \in \z: x^2+1 \text{ and } x^2+2 \text{ are squarefree} \}\right) =\prod_{p>2} \left(1-\frac{\left(\frac{-1}{p}\right)+ \left(\frac{-2}{p}\right)+2 }{p^2}\right).$$
					
				\end{example}

				In what remains of this section, we will work over $\q$ with $l =2$, although similar results could be obtained more generally. A first question that is very natural, is of when do we have $R^f \sim R^g$, for some arbitrary polynomials $f,g \in \z[X]$. Since both sieves are supported on the same set, \Cref{lm:Characterization of equivalent sieves} shows that this is equivalent to $R^f = R^g$, if at least one of the sieves has strong light tails for some $I_N$.
				
				Consequently, we need to answer when $R^f = R^g$. Note that this does not require that $f = g$, if one of these polynomials is not irreducible (over $\z$).
				
				\begin{example}
					Let $f(X) = 2X^2+1$ and $g(X) = 2f(X)$. If $p=2$, it is easy to see that $R^f_2 = R^g_2 = \emptyset$. On the other hand, for any $p>2$, we have that $p^2\mid g(m)$ is equivalent to $p^2\mid f(m)$. Consequently, it follows that $R^f = R^g$.
				\end{example}
				
				More generally, if $f$ is an irreducible polynomial, and $\cp_f$ is the set of primes $p$ such that $R_p^f = \emptyset$, then for any $m = p_1\dots p_k$ with $p_i \in \cp_f$, we will have that $R^f = R^{mf}$. If we take both $f$ and $g$ to be irreducible, then $R^f = R^g$ implies that $f= g$.
				
				\begin{theorem}
					Let $f,g \in \z[X]$ be distinct irreducible polynomials. We have $R^f \neq R^g$.
				\end{theorem} 
				
				For this proof, we will make use of the resultant $\text{Res}(f,g) \in \z$ (see Section 3 in \cite{Cox} for the definition). If $f$ and $g$ are irreducible polynomials in $\z[X]$, then there are polynomials $a,b \in \z[X]$ such that $$a(X)f(X)+b(X)g(X) = \text{Res}(f,g).$$ Consequently, if $f(x) \equiv g(x) \equiv 0 \mod p$ for some prime $p$, we must have that $p \mid \text{Res}(f,g)$, and so there are only finitely many such primes $p$.
				
				\begin{proof}
					We start by noticing that there are infinitely many primes $p$ such that $f(x) \equiv 0 \mod p$. These correspond to primes that split completely in the splitting field of $f$, so there are infinitely many by the Cheboratev Density Theorem (see for example Theorem 13.4 in Chapter 7 of \cite{Neukirch})\footnote{This result was apparently first shown by Schur in \cite{Schur}, although we were not able to access this source.}. 
					
					Taking one such prime $p$ that does not divide the discriminant of $f$, nor the resultant of $f$ and $g$, we will have that there is some $x$ such that $f(x) \equiv 0 \mod p$, and by Hensel's Lemma (see for example (4.6) in Chapter 2 of \cite{Neukirch}) there is some $s$ such that $f(s) \equiv 0 \mod p^2$ and $s \equiv x \mod p$.  Given that $g(x)$ is not congruent to $0$ modulo $p^2$, we  have $g(s) \not \equiv 0 \mod p^2$, and so $s \in R^f_p$ while $s \not \in R^g_p$.
				\end{proof}
				
				\begin{remark}
					\label{rmk: products of polynomials}
					Take $f$ and $g$ to be two irreducible polynomials. We have pointed out that the there are only finitely many $p$ such that $$W_p = \{x \in \z/p^2\z: f(x) \equiv g(x) \equiv 0 \mod p\} + p^2\z$$ is non-empty. Noting that if $p^2$ divides $(fg)(x)$, then either $p^2$ divides one of $f(x)$ or $g(x)$, or alternatively, $p$ divides both $f(x)$ and $g(x)$, we get $$R_p^{fg} = R_p^f\cup R_p^g \cup W_p.$$ It follows that, assuming that this union is well defined, the sieve $R^{fg}$ will have weak light tails for $B_N$ if the same holds for $R^f$ and $R^g$ (by \Cref{lm: union of sieves}) and $\cf_{R^{fg}} = \Sigma_{fg}$ (as pointed out in \Cref{eq: sigmaf equals frf}). Consequently, if $f$ is a polynomial that can be written as the product of distinct irreducible polynomials $f_i$ all of which satisfying that $R^{f_i}$ has weak light tails for $B_N$, we will have that \Cref{crl: admissible sets polynomials} also holds for $f$.
				\end{remark}

				We conclude this section by showing that if $f$ is irreducible and $R^f$ has weak light tails for some $I_N$, then $X_{R^f} = \Omega_{R^f}$.
				
				\begin{proposition}
					\label{prop: XRf equals ORf}
					Let $f\in \z[X]$ be an irreducible polynomial such that $R^f$ has weak light tails for some $I_N$. Then $$X_{R^f} = \Omega_{R^f}.$$
				\end{proposition}
				
				\begin{proof}
					Let $f$ be an irreducible polynomial in $\z[X]$. We want to show that $\limsup_{p \rightarrow \infty} \lambda(R_p^f) = \infty$ (with $\lambda$ defined as in \Cref{eq: lambdaS}) and then the result will follow from \Cref{prop: lim lambda infty implies XR = OmegaR}. 
					
					Since $f$ is irreducible, so is $g(X) := f(X+y)$ for any $y \in \z$, given that the map that sends $f(X)$ to $f(X+y)$ is a ring automorphism of $\z[X]$. Consequently, there are only finitely many primes $p$ for which $f(x)\equiv g(x) \equiv 0 \mod p$ has a solution (at most those that divide the resultant of $f$ and $g$). Since any solution to $f(x)\equiv g(x) \equiv 0 \mod p^2$ would be a solution modulo $p$, it follows that for any $y$, there are only finitely many primes $p$ such that $f(x)\equiv f(x+y) \equiv 0 \mod p^2$ has a solution. 
					
					Consequently, for any $N\geq1$, we can always find some big enough $p$ such that $f(x) \equiv 0 \mod p^2$ has a solution, but $f(x+y) \not \equiv 0 \mod p^2$ for any non-zero $y \in [-N,N]$. It follows that $\limsup_{p \rightarrow \infty} \lambda(R_p^f) = \infty.$ which concludes the proof.		
				\end{proof}
				
				\begin{remark}
					If $f$ is the product of irreducible polynomials, this might not hold. Take $$f(X) = (2X+1)(2(X-1)+1).$$ We have that $R_2^f = \emptyset$. Additionally, we have that the resultant of $(2X+1)$ and $(2(X-1)+1)$ is $-4$ (a power of $2$), so $R^f_p = R^{(2X+1)}_p \cup R^{(2(X-1)+1)}_p$ for every $p  \geq 3$. It follows that $R^f_p$ is always of the form $R^f_p = \{x_p,x_p+1\}+p^2\z$, where $x_p$ is the unique solution to $2x_p +1 \equiv 0 \mod p^2$. By the same argument as in \Cref{ex: sieve such that X different Omega}, it follows that $X_{R^f} \neq \Omega_{R^f}$.
				\end{remark}

				\subsection{Ergodic Prime Number Theorem for $R$-free Numbers}
				
				Inspired by the work in \cite{Wang}, we will show in this section an ergodic Prime Number Theorem for $R-$free numbers.
				
				Given a function $a: \mathbb{N} \rightarrow \mathbb{N}$, we say it is \textit{Besicovitch almost periodic}\index{Besicovitch Almost Periodic Function} (see \cite{Bergelson2}), if for every $\epsilon > 0$, there is some trigonometric polynomial $P_\epsilon(x) = \sum_{j=1}^k c_j e(x\alpha_j)$ (where $e(x) = e^{2\pi i x}$) such that $$\lim_{N \rightarrow \infty} \frac{1}{N}\sum_{m = 1}^N |a(m) - P_\epsilon(m)| < \epsilon.$$ We will show that if $R$ is an Erdős sieve with weak light tails for the Følner sequence $I_N =[1,N]$, then $\one_{\fr}$ is a Besicovitch almost periodic function, in order to apply Corollary 1.26 from \cite{Bergelson2}. Let $v_p$ denote the $p-$adic valuation, and
				\begin{equation}
					\label{eq: omega}
					\Omega(m) = \sum_{p \text{ prime}} v_p(m).
				\end{equation} This result states the following.
				
				\begin{lemma}
					\label{lm:Richter}
					Let $(X,T)$ be a uniquely ergodic dynamical system with unique $T-$invariant measure $\mu$. Let $a(n):\mathbb{N}\rightarrow \mathbb{C}$ be a Besicovitch almost periodic function and let $M(a):= \lim_{N \rightarrow \infty} \frac{1}{N}\sum_{m=1}^Na(m)$ be its mean value. Then for every function $f \in C(X)$ and $x \in X$ we have 	$$\lim_{N\rightarrow\infty} \frac{1}{N}\sum_{m=1}^N a(m)f(T^{\Omega(m)}x) = M(a)\int_X f \,d \mu.$$
				\end{lemma}

				We get the following result.\index{Prime Number Theorem ! for Sieves with Weak Light Tails}
				
				\begin{theorem}
					\label{thm: ergodic prime number theorem for weak light tails}
					Let $R$ be an Erdős sieve with weak light tails for $I_N = [1,N]$. Let $(X,T)$ be a uniquely ergodic dynamical system, and $\mu$ its unique invariant measure. Then for every function $f \in C(X)$ and $x \in X$ we have
					$$\lim_{N\rightarrow\infty} \frac{1}{N}\sum_{m \in \fr \cap I_N} f(T^{\Omega(m)}x) = d_I(\fr)\int_X f \,d \mu.$$
				\end{theorem}
				
				\begin{proof}
					By \Cref{lm:Richter} it is enough to show that $\one_{\fr}$ is Besicovitch almost periodic. We start by pointing out that for any arithmetic sequence $a+b\z$, we have $$\one_{a+b\z}(x) = \frac{1}{b}\sum_{j= 0}^{b-1}e\left(j\frac{(x-a)}{b}\right),$$ which follows from character orthogonality for the cyclic group $\z/b\z$. Consequently, denoting by $\overline{R}_i$ the image of $R_i$ in $\z/b_i\z$, we have that for any $L$ the function $$P_L(x) = 1- \prod_{i \leq L} \prod_{y \in \overline{R}_i} (1-\one_{y+b_i\z}(x))$$ is a trigonometric polynomial, which is $0$ if $x \in R_i$ for some $i \leq L$, and $1$ otherwise.
					
					We can partition the set $I_N = [1,N]$ like $I_N = (\fr \cap [1,N]) \cup S_1 \cup S_2 $, where $$  S_1 =   I_N \cap \bigcup_{i \leq L} R_i  \text{  and  } S_2 = I_N \cap \left( \bigcup_{i > L} R_i\setminus  \bigcup_{j \leq L} R_j \right).    $$ If $m \in \fr$, then for all $L$, $P_L(m) = 1$, and so $|\one_{\fr}(m) - P_L(m)| = 0$. If $m \in S_1$, both values are $0$, and we get $|\one_{\fr}(m) - P_L(m)| = 0$. Finally, if $m \in S_2$, we get $|\one_{\fr}(m) - P_L(m)| = 1$, so $$\lim_{N \rightarrow \infty} \frac{1}{N}\sum_{m = 1}^N |\one_{\fr}(m) - P_L(m)| = \lim_{N \rightarrow \infty} \frac{1}{N}\sum_{m \in S_2} 1 = d_I\left( \bigcup_{i > L} R_i\setminus  \bigcup_{j \leq L} R_j \right).$$ Since $R$ has weak light tails if and only if this density goes to $0$ as $L$ goes to infinity, the result follows.
				\end{proof}
				
				\begin{remark}
					\label{rmk: primen number theorem numenclature}
					Let us explain why we call this result a 'Prime Number Theorem'\index{Prime Number Theorem ! for Number Fields}. It is well known (see \cite{Bergelson2}) that the Prime Number Theorem is equivalent to the fact that $$\lim_{N\rightarrow\infty}\frac{1}{N}\sum_{m =1}^N (-1)^{\Omega(n)} = 0.$$ When considering the system $X = \{-1,1\}$ with $T(x) = -x$ and $f(x) = x$, the result implies the Prime Number Theorem, as we get $$\frac{1}{N}\sum_{m =1}^Nf(T^{\Omega(m)x}) \rightarrow \frac{1}{2}(f(1)+f(-1)) = 0.$$ 
					
				\end{remark}

				We point out that  \Cref{thm: ergodic prime number theorem for weak light tails} implies all results in \cite{Wang}, since each of the sets the authors consider in this paper can be realized as R-free numbers for some sieve $R$ with proven strong light tails. As previously mentioned, Heath-Brown showed in Theorem 1 of \cite{HeathBrown2} that if $f(X) = X^d+c$ is an irreducible polynomial and $l \geq (5d+3)/9$, then $R^{f,l}$ (see \Cref{eq: R fl}) has strong light tails for $I_N$. Applying \Cref{thm: ergodic prime number theorem for weak light tails} gives the following result, which is Theorem 1.1 in \cite{Wang}.
				
				\begin{corollary}
					Let $f(X) = X^d+c$ be an irreducible polynomial, $l $ an integer such that $l \geq (5d+3)/9$, and $(X,T)$ a uniquely ergodic dynamical system, with $\mu$ its unique invariant measure. Then for every function $g \in C(X)$ and $x \in X$ we have
					$$\lim_{N\rightarrow\infty} \frac{1}{N}\sum_{1\leq m\leq N: m^d+c \text{ is squarefree}} g(T^{\Omega(m)}x) = \prod_p\left(1-\frac{|R^{f,l}_p|}{p^l}\right)\int_X g \,d \mu.$$
				\end{corollary}
				
				\Cref{thm: Hooley Polynomials} together with \Cref{rmk: products of polynomials} imply that if $f$ is a polynomial that can be written as the product of distinct irreducible polynomials all of degree smaller than or equal to $3$, then $R^f$ has strong light tails for $I_N$. Applying \Cref{thm: ergodic prime number theorem for weak light tails} for this sieve gives the following result, which corresponds to Theorem 4.1 in \cite{Wang}.
				
				\begin{corollary}
					Let $f \in \z[X]$ be a polynomial that can be written as the product of distinct irreducible polynomials all of degree smaller than or equal to $3$. Let $(X,T)$ be a uniquely ergodic dynamical system, with $\mu$ its unique invariant measure. Then for every function $g \in C(X)$ and $x \in X$ we have
					$$\lim_{N\rightarrow\infty} \frac{1}{N}\sum_{1\leq m\leq N: g(m) \text{ is squarefree}} g(T^{\Omega(m)}x) = \prod_p\left(1-\frac{|R^{f}_p|}{p^2}\right)\int_X g \,d \mu.$$
				\end{corollary}
				
				Finally, using \Cref{ex: Dimitrov}, we get the following result, which corresponds to Corollary 4.2 in \cite{Wang}.
				
				\begin{corollary}
					Let $(X,T)$ be a uniquely ergodic dynamical system, with $\mu$ its unique invariant measure. Then for every function $g \in C(X)$ and $x \in X$ we have
					$$\lim_{N\rightarrow\infty} \frac{1}{N}\sum_{\substack{1\leq m\leq N \\ m^2+1,m^2+2 \text{ squarefree}}} g(T^{\Omega(m)}x) = \prod_{p>2} \left(1-\frac{\left(\frac{-1}{p}\right)+ \left(\frac{-2}{p}\right)+2 }{p^2}\right)\int_X g \,d \mu.$$
				\end{corollary}

			\bigskip
			\footnotesize
			\noindent
			Francisco Ara\'{u}jo\\
			\textsc{Institute of Mathematics, Paderborn University, Warburger Str. 100, 33098 Paderborn, Germany}\par\nopagebreak
			\noindent
			\textit{E-mail address:} \texttt{faraujo@math.uni-paderborn.de}

			\end{document}